\documentclass[10pt]{article}
\usepackage
    {inputenc}
\usepackage[english]{babel}

\usepackage{amsmath,amsfonts,amssymb,amsthm,mathtools}

\usepackage{graphicx}
\usepackage{caption}
\usepackage{subcaption}
\usepackage{enumerate}
\usepackage{comment}
\usepackage{lscape}

\usepackage[a4paper, left = 3.3cm, right= 3.3cm, top= 3cm, bottom = 3.3cm, heightrounded]{geometry}
\usepackage[colorlinks=true, allcolors=blue]{hyperref}

\usepackage[style=alphabetic, backend=bibtex, doi=true,isbn=false,url=true]{biblatex} 
\usepackage{csquotes}
\newtheorem{theorem}{Theorem}
\newtheorem{proposition}{Proposition}
\newtheorem{lemma}{Lemma}

\newtheorem{remark}{Remark}

\newcommand{\Cbb}{{\mathbb{C}}}
\newcommand{\Ebb}{{\mathbb{E}}}

\newcommand{\Nbb}{{\mathbb{N}}}
\newcommand{\Pbb}{{\mathbb{P}}}
\newcommand{\Qbb}{{\mathbb{Q}}}
\newcommand{\Rbb}{{\mathbb{R}}}

\newcommand{\Fcal}{{\mathcal{F}}}
\newcommand{\Hcal}{{\mathcal{H}}}

\newcommand{\nuest}{{\nu_\text{est}}}

\newcommand{\Qcal}{{\mathcal{Q}}}

\DeclareMathOperator{\argmin}{{\text{arg min}}}

\renewcommand{\geq}{\geqslant}
\renewcommand{\leq}{\leqslant}

\title{Deconvolution of repeated measurements corrupted by unknown noise}

\author{J\'er\'emie Capitao-Miniconi$^*$, \'Elisabeth Gassiat$^*$ and Luc Leh\'ericy$^{**}$\\
{\small $^*$ Universit\'e Paris-Saclay, CNRS, Laboratoire de math\'ematiques d'Orsay, 91405, Orsay, France,}\\
{\small $^{**}$ Universit\'e Côte d’Azur, CNRS, LJAD, France.}}

\date{}

\begin{document}
\maketitle

\begin{abstract}

Recent advances have demonstrated the possibility of solving the deconvolution problem without prior knowledge of the noise distribution. In this paper, we study the repeated measurements model, where information is derived from multiple measurements of $X$ perturbed independently by additive errors. Our contributions include establishing identifiability without any assumption on the noise except for coordinate independence. We propose an estimator of the density of the signal for which we provide rates of convergence, and prove that it reaches the minimax rate in the case where the support of the signal is compact. Additionally, we propose a model selection procedure for adaptive estimation. Numerical simulations demonstrate the effectiveness of our approach even with limited sample sizes.
\end{abstract}

\section{Introduction}


Density deconvolution is one of the classical topics in nonparametric statistics and has garnered significant attention over the past decades. The aim is to identify the density of some random variable $X$, the signal, which cannot be observed directly, but is contaminated by some 
additive error $\varepsilon$, the noise, independent of $X$.
%
Most of the literature on the deconvolution problem considers 
the situation where the distribution of the noise is perfectly known, see
\cite{MR0997599}, \cite{MR1126324},  \cite{MR1054861}, \cite{MR1045189},  for early references, see also the book \cite{MR2768576}.
However, knowledge of the noise distribution is unrealistic in practice. In some situations, it is possible to get a pure noise sample so that the noise distribution may be estimated separately and plugged into methods that assume the noise distribution known, see \cite{MR2543693}, \cite{MR1460203}.
It has been proved recently \cite{LCGL2021} that, under very mild assumptions, it is possible to solve the deconvolution problem without knowing the noise distribution and with no sample of the noise, for multivariate signals.

In the present work, we are interested in the case where information can be drawn from repeated measurements of 
$X$, so that the multivariate signal is the repetition of $X$. 
This framework is known as model of repeated measurements. The observations are 

\begin{equation*}
Y_{i,j} = X_j + \varepsilon_{i,j} \ ; \ i = 1, \ldots, p \ ; \ j = 1, \ldots, n
\end{equation*}
where the random variables $X_j$ and $\varepsilon_{i,j}$, $j=1,\ldots,p$, $i=1,\ldots,n$ are independently distributed.

By making substractions of the coordinates of the observations, it is possible to estimate the noise distribution consistently when it is assumed to be symmetric, and consequently estimate the signal density, see \cite{MR3207174}, \cite{MR2396811} and \cite{MR3299125}. The symmetry assumption can be dropped using Kotlarski's identity \cite{MR0203769}. Further works use this identity to propose estimation strategies not relying on the symmetry assumption, see \cite{MR1625869}, \cite{MR3372551} and \cite{MR4349899}. However, all these works require that the characteristic function of the signal and the characteristic function of the noise vanish nowhere. The identifiability result of \cite{MR2396948} relaxes this non vanishing assumption, although it still prevents the characteristic function from vanishing on open non empty sets.

The aim of our work is to prove that it is not needed to assume non vanishing of the characteristic functions to be able to estimate the signal and the noise distributions.
Our main contributions are as follows.

\begin{itemize}
    \item We first prove an identifiability result in general settings, see Theorem \ref{thm:id}. Our only assumption on the noise is that its coordinates are independently distributed. We do not assume anything on its characteristic function, which may vanish anywhere. Our assumption on the signal is that its Laplace transform is finite everywhere and grows at most as the exponential of a power function, and is equivalent to an assumption on the tails of the distribution of the signal. We also do not assume anything on its characteristic function.
    \item We propose an estimation procedure for the density of the signal  $X$ in $\Rbb^d$. We prove in Theorem \ref{thm:density} that the integrated quadratic risk of our estimator 
    is upper bounded by $(\log\log(n)/\log(n))^{2\beta /\rho}$ where $\beta$ is a regularity parameter and $\rho \geq 1$ is a parameter that depends on the tail of the distribution ($\rho=1$ corresponds to a compactly supported distribution, and $\rho = 2$ to a sub-Gaussian distribution). In the compact case ($\rho = 1$), we show in Theorem \ref{thm:lowerbound} that this rate is minimax.
    \item We propose a model selection procedure to obtain an adaptive estimator with the same rate of convergence as the estimator with a known tail parameter $\rho$, see Theorem \ref{thm:adaptivity}. 
    To cover situations in which the estimation rate may be improved when the characteristic function of the noise does not vanish and is ordinary smooth,
    we construct a data-driven combination of our estimator $\widehat{f}$ and of the estimator $\widehat{f}^{CK}$ built in \cite{MR3372551}. We prove in Theorem \ref{thm:adaptnoise} that this combination achieves the best rate of convergence among the two, 
    even if we do not know that the noise is ordinary smooth.
    \item Finally, we present numerical simulations in Section \ref{sec:simu}. In various settings for the signal distribution and for the noise distribution, we find that our estimator has surprisingly good behaviour even with a small sample size ($n=500$). We compare our results with the experiments of \cite{MR3372551} and find that in most of the examples, our estimator outperforms the one from \cite{MR3372551}. Moreover, our simulations indicate that our procedure performs well even when the tails of the distribution of the signal are too heavy for our theoretical results to apply.
    Finally, we discuss the choice of the hyperparameters involved in our procedure and propose a data-driven method to select them.
\end{itemize}
Possible further works are discussed in Section \ref{sec:discussion}.
Detailed proofs can be found in Section \ref{sec:proofs}.


\section{Setting and identifiability Theorem}
\label{subsec:prelim}

Consider the repeated measurements model with 
2 repetitions:
\begin{equation}
\label{eq:model}
    Y = \begin{pmatrix} Y^{(1)}\\ Y^{(2)}\end{pmatrix} = \begin{pmatrix} X \\ X \end{pmatrix} + \begin{pmatrix} \varepsilon^{(1)} \\ \varepsilon^{(2)} \end{pmatrix}= r(X)+\varepsilon,
\end{equation}
in which, for $i \in \lbrace 1, 2 \rbrace$, 
$Y^{(i)}, X, \varepsilon^{(i)} \in \Rbb^d$, and $r : x \in \Rbb^d \mapsto (x,x) \in \Rbb^d \times \Rbb^d$. We assume that the random variables $X$ and $\varepsilon$ are independent, 
and we consider independent and identically distributed observations $Y_{j}$, $j=1,\ldots,n$, following model~\eqref{eq:model}. 

\medskip

We shall not assume that the distribution of the noise $\varepsilon$ is known, instead the only assumption we will make on the noise is the following.
\begin{itemize}
\item[\textbf{(H1)}]  $\varepsilon^{(1)}$ and $\varepsilon^{(2)}$ are independent random variables.
\end{itemize}

Let us now introduce our assumption on the Laplace transform of $X$.
\begin{itemize}
\item[\textbf{(H2)}]
There exists $\rho > 0$, $a>0$ and $b>0$ such that for all $\lambda \in \Rbb^d$,
$\Ebb \left[\exp \left(\lambda^\top X\right)\right]
\leq a \exp \left( b \| \lambda\|^\rho\right)$, where $\|\cdot\|$ denotes the Euclidean norm.
\end{itemize}
Note that by Chernoff's bound, this is equivalent to assuming that the tails of the distribution of $X$ satisfy $\Pbb(\|X\| \geq t) = O(\exp(c t^{1+1/(\rho-1)}))$ when $\rho > 1$, $X$ a.s. bounded when $\rho = 1$, and $X=0$ when $\rho < 1$.

Under \textbf{(H2)},
the characteristic function of the signal $X$ can be extended into a multivariate analytic function denoted by
\begin{eqnarray*}
\Phi_X : \Cbb^{d} & \longrightarrow& \Cbb\\
z &\longmapsto& \Ebb \left[ \exp \left(i z^\top X \right)\right]
\end{eqnarray*}

Obviously, if no centering constraint is put on the signal or on the noise, it is possible to translate the signal by a fixed vector $m \in \Rbb^d$ and the noise by $-m$ without changing the observation. The model can thus be identifiable only up to translation. We prove the following identifiability theorem.

\begin{theorem}
\label{thm:id}
Assume that $X$ and $X'$ satisfy \textbf{(H2)} and $\varepsilon \sim \Qbb$, $\tilde{\varepsilon} \sim \tilde{\Qbb}$ satisfy \textbf{(H1)}. Then $\Pbb_{r(X)} * \Qbb = \Pbb_{r(X')} * \tilde{\Qbb}$ implies $\Pbb_{X} = \Pbb_{X'}$ and $\Qbb = \tilde{\Qbb}$ up to translation.
\end{theorem}
The proof of Theorem \ref{thm:id} is detailed in Section \ref{proof:thm:id}. It may be seen as starting similarly as the proof of Theorem 2.1 in \cite{LCGL2021} and then taking into account the particular form of the characteristic function of the observations in model \eqref{eq:model}.

Our result improves on earlier results concerning the assumption on the noise distribution. Indeed, the identifiability result of \cite{MR0203769} (further used in \cite{MR1625869},  \cite{MR3372551} and \cite{MR4349899}) requires that the characteristic functions of both the noise and the signal vanish nowhere.
Our approach differs from the identifiability Lemma 2.1 of \cite{MR2396948}
in the sense that they do not make any assumption on the tails of the distribution of the signal but use some intricate assumption on the non zero sets of the characteristic functions of the noise and the signal, which excludes noise characteristic functions vanishing on non empty open subsets of $\mathbb R$. 


\section{Estimation procedure}

From now on, we assume that the distribution $\Pbb_X$ of the signal has a density $f$ with respect to the Lebesgue measure on $\Rbb^d$. We shall assume that $X$ satisfies \textbf{(H2)}, and we also assume that an upper bound $\rho_0$ is known on $\rho$, where $\rho$ is given by \textbf{(H2)}. 

\bigskip
 
The first step in the estimation procedure is the estimation of the characteristic function of the signal by a method inspired by the proof of the identifiability theorem. 
A key step in the proof of Theorem \ref{thm:id} is the fact that, if a multivariate analytic  function $\phi$ has growth as in \textbf{(H2)}, is such that $\phi(0)=1$, for all $t\in\Rbb^{d}$,
$\overline{\phi(t)}=\phi(-t)$, and satisfies, for all $t_1$, 
$t_2$, in  a neighborhood of $0$ in $\Rbb^{d}$,
\begin{equation*}
\phi(t_1+t_2) \Phi_X(t_1) \Phi_X(t_2)=\Phi_X(t_1 + t_2) \phi(t_1) \phi(t_2),
\end{equation*}
then $\phi=\Phi_X$ up to translation, that is up to multiplication by a factor $\exp(i c^\top t)$ for some constant vector $c \in \Rbb^d$. 
In other words, if we define, for $\nu > 0$,
\begin{multline*}
M(\phi; \nu \vert \Phi_X) = \int_{[-\nu,\nu]^d \times [-\nu,\nu]^d} | \phi(t_1+t_2) \Phi_X(t_1) \Phi_X(t_2)- \Phi_X(t_1 + t_2) \phi(t_1) \phi(t_2) |^2 \\
|\Phi_{\varepsilon^{(1)}}(t_1)\Phi_{\varepsilon^{(2)}}(t_2)  |^2 d t_1 d t_2,
\end{multline*}
then $\Phi_X$ is the only minimizer (up to translation) of $M(\cdot; \nu \vert \Phi_X)$ over a well chosen set of multivariate analytic functions. We will construct an estimator of the criterion $M(\cdot; \nu \vert \Phi_X)$ based on the observations and minimize it to get an estimator of the characteristic function of the signal. Let us now describe the details of this procedure. 

For any $S>0$, let $\Upsilon_{\rho,S}$ be the subset of multivariate analytic functions from $\Cbb^{d}$ to $\Cbb$ defined as follows.
\begin{multline*}
\Upsilon_{\rho,S} = \Bigg \{
\phi \text{ analytic } \text{s.t. } \forall z\in\Rbb^{d},\overline{\phi(z)}=\phi(-z), \phi(0) = 1 \\
\text{ and } \forall j \in \Nbb^d \setminus \{0\},
\left| \frac{\partial^{j} \phi(0)}{j! }\right| \leq \frac{S^{\|j\|_1}}{(\|j\|_1)^{ \|j\|_1/\rho}} \Bigg \},
\end{multline*}
where $j! = \prod_{a = 1}^d j_a!$ and $\partial^j = \partial^{j_1}_1 \ldots \partial^{j_{d}}_d$.
By Lemma 3.1 in \cite{LCGL2021}, the family of sets $\Upsilon_{\rho,S}$, $S>0$, summarizes Assumption \textbf{(H2)} with parameter $\rho$, so that for large enough $S$, $\Phi_X$ is the only minimizer (up to translation) of $M(\cdot; \nu \vert \Phi_X)$ over $\Upsilon_{\rho,S}$.
Fix some constant $\nu_{\text{est}}>0$ and define $M_n$ as, for all $\phi \in \Upsilon_{\rho,S,d}$,

\begin{equation*}
M_{n}(\phi) = \int_{[-\nu_{\text{est}},\nu_{\text{est}}]^d \times [-\nu_{\text{est}},\nu_{\text{est}}]^d} | \phi(t_1 + t_2) \tilde\phi_{n}(t_1,0) \tilde\phi_{n}(0,t_2) - \tilde \phi_{n}(t_1,t_2) \phi(t_1) \phi(t_2) |^2 d t_1 d t_2 ,
\end{equation*}
where for all $(t_1,t_2)\in\Cbb^d \times \Cbb^d$,
\begin{equation*}
\tilde \phi_{n}(t_1,t_2) = \frac{1}{n}\sum_{\ell=1}^{n} \exp\left\{it_1^\top Y_{\ell}^{(1)} + it_2^\top Y_{\ell}^{(2)}\right\}.
\end{equation*}
As $n$ tends to infinity, $M_{n}(\phi)$ converges a.s. to $M(\phi; \nu_{\text{est}} \vert \Phi_X)$. We shall minimize $M_n$ over multivariate polynomials. 
We thus introduce, for all $m \in \Nbb$, the set  $\Cbb_m[X_1, \ldots, X_d]$  of multivariate polynomials in $d$ variables with total degree at most $m$ and coefficients in $\Cbb$.
For any integer $m$ and real $\rho > 0$, define $\widehat \Phi_{n,m,\rho}$ as a (up to $1/n$) measurable minimizer of the functional $ \phi \mapsto M_n(\phi)$ over $\Cbb_m[X_1, \ldots, X_d]\cap \Upsilon_{\rho,S,d}$, that is, 
\begin{equation*}
    \widehat \Phi_{n,m,\rho}\in \Cbb_m[X_1, \ldots, X_d]\cap \Upsilon_{\rho,S,d}
\end{equation*}
and
\begin{equation*}
    \forall \phi\in\Cbb_m[X_1, \ldots, X_d]\cap \Upsilon_{\rho,S,d},\;
M_n(\widehat \Phi_{n,m,\rho})\leq M_n(\phi)+\frac{1}{n}.
\end{equation*}
We shall prove in Proposition \ref{prop:rate:char} that for well chosen $m$ and small enough $\nu$, $\widehat \Phi_{n,m,\rho}$ converges to $\Phi_X$ in $L^2([-\nu,\nu])$ at almost parametric rate--when $m$ is well chosen--, uniformly over $\rho$ for $\rho$ in the compact set $[1,\rho_0]$. 
The well chosen $m$ will be set to $m=\left \lceil 2\rho_0\frac{\log(n)}{\log \log(n)} \right \rceil$, where $\rho_0$ is the {\it{a priori}} upper bound on $\rho$. For the sake of simplicity, we denote $\widehat \Phi_{n,\rho}$ the estimator $\widehat \Phi_{n,m,\rho}$ in which $m=\left \lceil 2\rho_0\frac{\log(n)}{\log \log(n)} \right \rceil$, which is a valid choice of $m$ to get this uniform almost parametric convergence rate, see Proposition \ref{prop:rate:char} below.

To define our estimator of the density of $X$, we truncate the polynomial expansion further. Let us introduce the truncation operator $T_m$ as follows. 
If $\phi$ is a multivariate analytic function defined in a neighborhood of $0$ in $\Cbb^D$ written as $ \phi : x \mapsto \sum_{(i_1, \ldots, i_d) \in \Nbb^d} c_i \prod_{a=1}^d x_a^{i_a}$, define 
\begin{equation*}
T_m \phi : x \mapsto \sum_{(i_1, \ldots, i_D) \in \Nbb^D :i_{1}+\ldots+i_{d} \leq m} c_i \prod_{a=1}^D x_a^{i_a}.
\end{equation*}
We finally define the estimator of the density of the signal as follows. Fix some integer $m_{n,\rho}$ and positive real number $h_{n,\rho}$. Then for all $t \in \Rbb^d$,
\begin{equation}
\label{eq_inversion_Fourier}
    \widehat{f}_{n,\rho}(t) = \frac{1}{(2\pi)^d} \int_{[-h_{n,\rho},h_{n,\rho}]^d} \exp(-it^\top u)  \ T_{m_{n,\rho}} \widehat \Phi_{n,\rho}(u) du.
\end{equation}
We prove in Theorem~\ref{thm:density} that a good choice of $m_{n,\rho}$ and $h_{n,\rho}$ allows to control the integrated quadratic risk over regularity classes of densities, and construct an estimator that is adaptive in $\rho$ in Theorem~\ref{thm:adaptivity}. The rates are shown to be minimax optimal for compactly supported signals in Theorem~\ref{thm:lowerbound}.

\section{Rates of convergence}


The first step to control the quadratic risk of the estimated density is to control the quadratic risk of the estimator of the characteristic function over some small set in $\Rbb^d$.
The constants in the proposition below depend on the signal through $\rho$ and $S$, and on the noise through its second moment and the quantity
\begin{equation}
\label{eq:cnu}
c_{\nu}=\inf\{|\Phi_{\varepsilon^{(i)}}(t)|,\;t\in [-\nu,\nu]^d,\;i=1,2\},
\end{equation}
provided it is positive, which holds for any noise distribution for small enough $\nu$ by continuity of the characteristic function.
For any $\nu>0$, $c(\nu ) >0$, $E>0$, define $\Qcal^{(2d)} (\nu,c({\nu}),E)$ the set of distributions $\Qbb = \otimes_{j=1}^2 \Qbb_j$ on $\Rbb^{2d}$ such that $c_{\nu}\geq c(\nu)$ and $\int_{\Rbb^{2d}} \|x\|^2 d\Qbb(x)\leq E$.


\begin{proposition}
\label{prop:rate:char}
For all $\nu\in(0,\nuest]$ and $\rho_0 \geq 1$, $S,c(\nu),E>0$ and $\delta, \delta',\delta'' \in (0,1)$ with $\delta' > \delta$, there exist positive constants $c$ and $n_0$ such that the following holds. For any $\rho \in [1, \rho_0]$, for all $\Phi_{X} \in \Upsilon_{\rho,S,d}$ and $\Qbb \in \Qcal^{(2d)} (\nu,c({\nu}),E)$,
for all $n\geq n_0$ and $x \in [1, n^{1-\delta'}]$, with probability at least $1 - 2e^{-x}$,
\begin{equation*}
\sup_{\rho' \in [\rho, \rho_0]}
    \int_{[-\frac{\nu}{2},\frac{\nu}{2}]^d} |\widehat \Phi_{n,\rho'}(t) - \Phi_{X}(t)|^2 dt
        \leq c \left(\frac{x}{n^{1-\delta}}\right)^{1-\delta''}.
\end{equation*}
\end{proposition}

The proof of Proposition \ref{prop:rate:char} is adapted from the proof of Proposition 2 of \cite{InferSupport} and is detailed in Section \ref{proof:thm:char}. 


\subsection{Upper bound}

The aim of this section is to give an upper bound of the  maximum $L_2 (\Rbb^d)$-risk for the estimation of $f$. We shall denote $\| \cdot \|_{2}$ the norm in $L_2 (\Rbb^d)$.
For all $\rho \geq 1$, $\beta > 0$, $S > 0$, $c_{\beta}>0$, we denote $\Psi(\rho,S,\beta,c_{\beta})$ the set of distributions $\Pbb_X$ with a density $f$ on $\Rbb^d$ such that the characteristic function $\Phi_X$ is in $\Upsilon_{S,\rho,d}$ and satisfies 
\begin{equation*}
    \int_{\Rbb^d} |\Phi_X(u)|^2 (1 + \|u\|^2)^{\beta} du \leq c_{\beta}.
\end{equation*}

\begin{theorem}
\label{thm:density}
For $c_h \leq \exp(-(5d+3)/2)$, define for any $\rho\geq 1$ 
\begin{equation*}
    m_{n,\rho} = \Bigg \lfloor \frac{\rho}{4} \frac{ \log(n)}{\log  \log(n)} \Bigg \rfloor, \  \  \  h_{n,\rho} = c_h \frac{m_{n,\rho}^{1/\rho}}{S}.
\end{equation*}
Then for any 
$\nu\in(0,\nu_\text{est}]$, $c(\nu)>0$, $E>0$, $S>0$, $\beta > 0$ and $c_{\beta} > 0$, there exists $n_0$ and $C > 0$ such that for all $n \geq n_0$,
\begin{equation*}
\sup_{\rho \in[1,\rho_{0}]}
\underset{\Qbb\in \Qcal^{(2d)} (\nu,c({\nu}),E)}
        {\sup_{\Pbb_{X} \in \Psi(\rho,S,\beta,c_{\beta})}}
    \left ( \frac{\log(n)}{\log \log(n)} \right )^{\frac{2\beta}{\rho}}
    \Ebb_{(\Pbb_{r(X)} * \Qbb)^{\otimes n}}[\| \widehat{f}_{n,\rho} - f \|^2_{2}]
        \leq C.
\end{equation*}
\end{theorem}
The proof of Theorem \ref{thm:density} is detailed in Section~\ref{proof:thm:density}.

Let us compare this upper bound with the previous results in the literature. 
The best rates of convergence without the symmetry assumption on the noise are obtained in \cite{MR3372551}, who improve upon the results in \cite{MR1625869} regarding both the assumptions on the noise and signal distributions and the rates of convergence. Though the earlier work \cite{MR2396948} has weaker assumptions than \cite{MR3372551}, the author only proves consistency of his estimator but does not provide rates of convergence.

The rates in \cite{MR3372551} depend on the tail of the characteristic function of the signal and on the tail of the characteristic function of the noise, as is usual in the deconvolution literature. The authors prove polynomial rates of convergence when the noise is ordinary smooth (that is with a characteristic function decreasing as a power function), up to a power of $\log n$ factor when the signal is not ordinary smooth but supersmooth. For supersmooth errors the rate is a power of $\log n$.
Our assumptions are not an exact generalization of the ones in \cite{MR3372551}, as they do not need assumptions on the tails of the distribution of the signal, and instead assume that both the characteristic function of the signal and the characteristic function of the noise do not vanish anywhere and have some controlled behaviour near infinity. In contrast, the class upon which our upper bound applies includes both ordinary smooth and supersmooth noises, and allows the characteristic functions to vanish, even on open sets.
We propose in Section \ref{sec:adaptnoise} a method to choose between our estimator and the one proposed in \cite{MR3372551} to get the best of both worlds.

Although we were not able to prove it, we believe it should be possible to adapt our estimation procedure to the tails of the characteristic functions by selecting the parameter $\nuest$ based on the observations to obtain improved rates on the classes where the characteristic functions do not vanish and are ordinary smooth. Our simulations  indicate that this idea seems relevant, see Section \ref{sec:simu}. We discuss this point in Section \ref{sec:discussion}.

\subsection{Adaptivity in \texorpdfstring{$\rho$}{the tail parameter}}

The construction of the estimator above assumes the tail parameter $\rho$ known. Unfortunately, this tail parameter is typically unknown in practice. We now propose a data-driven model selection procedure to choose $\rho$ and we prove that the resulting estimator has a rate of convergence corresponding to the smallest $\rho$ such that $\Phi_X \in \Upsilon_{\rho,S}$ for some $S>0$.
Our strategy is based on Goldenshluger and Lepski's methodology \cite{MR2543590} and on the adaptivity in $\rho$ in \cite{LCGL2021}.
As usual, the idea is to perform a bias-variance trade off. Although we have an upper bound for the variance term, the bias is not easily accessible. 
The variance bound is
\begin{equation*}
    \sigma_n(\rho) = c_{\sigma} \left( \frac{\log \log(n)}{\log(n)} \right)^{\frac{2 \beta}{\rho}}
\end{equation*}
for all $\rho \in [1, \rho_0]$ and for some constant $c_{\sigma} > 0$. The proxy of the bias is defined for all $\rho \in [1,\rho_0]$ as
\begin{equation*}
    A_n(\rho) = 0 \vee \sup_{\rho' \in [\rho, \rho_0]} \lbrace \|\widehat{f}_{n,\rho'} - \widehat{f}_{n,\rho}\|_2 - \sigma_n(\rho') \rbrace.
\end{equation*}
Finally, the estimator of $\rho$ is
\begin{equation*}
    \widehat{\rho}_n \in \argmin \lbrace A_n(\rho) + \sigma_n(\rho), \ \rho \in [1, \rho_0] \rbrace,
\end{equation*}
and the estimator of the density of the signal is $\widehat{f}_{n,\widehat{\rho}_n}$. The following theorem states that this estimator is rate adaptive.

\begin{theorem}
\label{thm:adaptivity}
For any 
$\nu \in (0, \nuest]$, $\beta > 0$, $c_{\beta >0}$ $c(\nu) > 0$, $E>0$, $S>0$, there exists $c_{\sigma} > 0$ and $C > 0$ such that

\begin{equation*}
\limsup_{n \rightarrow + \infty} \sup_{\rho \in[1,\rho_{0}]}
\underset{\Qbb\in \Qcal^{(2d)} (\nu,c({\nu}),E)}
        {\sup_{\Pbb_{X} \in \Psi(\rho,S,\beta,c_{\beta})}}
    \left ( \frac{\log(n)}{\log \log(n)} \right )^{\frac{2\beta}{\rho}}
    \Ebb_{(\Pbb_{r(X)} * \Qbb)^{\otimes n}}[\| \widehat{f}_{n,\widehat{\rho}_n} - f \|^2_{2}]
        \leq C.
\end{equation*}
\end{theorem}
The proof follows the same lines as that of Theorem 3.5 of \cite{LCGL2021}, taking $\rho = 1/\kappa$ and $\Hcal$ as the set of multivariate analytic functions $g : \Cbb^{2d} \longrightarrow \Cbb$   such that there exists $G : \Cbb^{d} \longrightarrow \Cbb$ such that for all $(t_1,t_2) \in \Cbb^d \times \Cbb^d$, $g(t_1,t_2) = G(t_1+t_2)$.

\subsection{Lower bound}

In this section, we provide a lower bound of the minimax risk in the case $\rho = 1$, that is for compactly supported signals. The lower bound in Theorem \ref{thm:lowerbound} matches the rate given in Theorem \ref{thm:density} for our estimator, proving that our estimator, together with its adaptive version, achieve the minimax adaptive rate in this case.

\begin{theorem}
\label{thm:lowerbound}
For all $S>0$, $\beta \geq 1/2$, $c_{\beta}>0$, and $\nu > 0$, there exists a constant $c>0$ such that

\begin{equation*}
\inf_{\widehat{f}}
\underset{\Qbb\in \Qcal^{(2d)} (\nu,c({\nu}),E)}
        {\sup_{\Pbb_{X} \in \Psi(1,S,\beta,c_{\beta}) }}
    \Ebb_{(\Pbb_{r(X)} * \Qbb)^{\otimes n}}[\| \widehat{f} - f \|^2_{2}]
        \geq c \left ( \frac{\log \log(n)}{\log(n)} \right )^{2\beta},
\end{equation*}
where the infimum is taken over all the estimators $\widehat{f}$, that is all measurable functions of $Y_1, \ldots, Y_n$.
\end{theorem}
The proof of Theorem \ref{thm:lowerbound} is detailed in Section \ref{proof:thm:lowerbound}.
It is based on Le Cam’s method, also known as the two-points method (see \cite{MR0856411}), one of the most widespread technique to derive lower bounds, and adapts ideas from \cite{LCGL2021} to the repeated measurements setting.

\section{Adaptativity to unknown noise regularity}
\label{sec:adaptnoise}

In this section, we propose a method to choose between our estimator and the one developed in \cite{MR3372551}.
We thus consider the same setting as \cite{MR3372551} and assume that $d=1$, that is, observations are in one dimension. We also assume that both coordinates of the noise have the same distribution, with density 
$f_{\varepsilon}$. 

Let us first recall results from \cite{MR3372551}. For positive constants $C_1$, $C_2$, $c$, $\beta$ and $\eta$, we denote by $\Fcal^u(C_1,C_2,c,\beta,\eta)$ the class of square integrable probability densities $f$ whose characteristic function $\Phi$ satisfies

\begin{equation}
\label{eq:condFC}
    \forall u,v \in \Rbb^{+}, \quad u \geq v \Rightarrow |\Phi(u)| \leq C_2 |\Phi(v)|
\end{equation}
and
\begin{equation*}
    \forall u \in \Rbb, \ |\Phi(u)| \leq \left( 1 + C_1 |u|^2 \right)^{-\beta} e^{-c|u|^{\eta}}.
\end{equation*}
We denote $\Fcal^{l}(C_1,C_2,c,\beta,\eta)$ the class of square integrable probability densities for which the condition \eqref{eq:condFC} holds and
\begin{equation*}
    \forall u \in \Rbb, \ |\Phi(u)| \geq \left( 1 + C_1 |u|^2 \right)^{-\beta} e^{-c|u|^{\eta}}.
\end{equation*}
For $C_3 >0$, we denote $\Fcal(C_3,p)$ the class of pairs $(f_X,f_{\varepsilon})$ of square integrable probability densities for which the following conditions are met:
\begin{equation*}
    \left( \| \Phi_X''\Phi_{\varepsilon} + \Ebb[\varepsilon^2] \Phi_X \Phi_{\varepsilon}\|_1 + \|\Phi_X'\Phi_{\varepsilon} \|_2^2 \right)^{p} + \Ebb[|X+\varepsilon|^{2p}] \leq C_3,
\end{equation*}
and $(\log(\Phi_{X}\Phi_{\varepsilon})')$ is square integrable with
\begin{equation*}
    \|\log(\Phi_{X}\Phi_{\varepsilon})' \|_2^{2p} \leq C_3.
\end{equation*}
Finally, we use the notation
\begin{equation*}
    \Fcal^{u,l}(X,\varepsilon,p) = \left[ \Fcal^{u}(C_{1,X},C_{2,X},c_X,\beta_X,\eta_X) \times \Fcal^{l}(C_{1,\varepsilon},C_{2,\varepsilon},c_{\varepsilon},\beta_{\varepsilon},\eta_{\varepsilon}) \right] \cap \Fcal(C_3,p).
\end{equation*}
In the following, we say that the noise is ordinary smooth when $f_{\varepsilon} \in \Fcal^{u}(C_{1,\varepsilon},C_{2,\varepsilon},0,\beta_{\varepsilon},\eta_{\varepsilon})$ (that is with $c_{\varepsilon} = 0$) and the signal is supersmooth if its density $f$ is in $ \Fcal^{u}(C_{1,X}, C_{2,X}, c_x, \beta_X, \eta_X)$ with $c_X > 0$.
The authors of \cite{MR3372551} prove the following theorem.
\begin{theorem} (Case III in \cite{MR3372551})
\label{thm:Comte}
Let $p \geq 2$, there exists $\widehat{f}^{CK}$ such that, in the case of supersmooth signal and ordinary smooth errors,
\begin{equation*}
    \sup_{(f,f_{\varepsilon}) \in \Fcal^{u,l}(X,\varepsilon,p)} \Ebb_{(f,f_{\varepsilon})}[\|\widehat{f}^{CK}-f \|_2^2] = O \left( n^{-\frac{p}{p+1}} (\log(n))^{\frac{p \gamma_1 + \gamma_2}{(p+1)\eta_X}} \right)
\end{equation*}
where $\gamma_1 = - 4 \beta_X+1-\eta_X$ and $\gamma_2 = \max  \lbrace \max \lbrace p(4 \beta_X + 4 \beta_{\varepsilon}+1 - \eta_X),0\rbrace + 1 - \eta_X, 0  \rbrace$.
\end{theorem}

We build a new estimator by combining the estimator $\widehat{f}^{CK}$ and $\widehat{f}_{n,\rho}$ 
as follows:
\begin{equation*}
    \widehat{f}^{A} = \begin{cases}
    \widehat{f}^{CK} \ \ \text{if} \ \| \widehat{f}_{n,\rho} - \widehat{f}^{CK} \|^2_{2} \leq C_{adapt} \left ( \frac{\log \log(n)}{\log(n)} \right )^{\frac{2\beta}{\rho}} \\
    \widehat{f}_{n,\rho} \ \  \text{if} \ \| \widehat{f}_{n,\rho} - \widehat{f}^{CK} \|^2_{2} > C_{adapt} \left ( \frac{\log \log(n)}{\log(n)} \right )^{\frac{2 \beta}{\rho}} \\
    \end{cases}
\end{equation*}
for some constant $C_{adapt}$.
The idea is that either $(f,f_{\varepsilon}) \in \Fcal^{u,l}(X,\varepsilon,p)$, in which case both $\widehat{f}^{CK}$ and $\widehat{f}_{n,\rho}$ will be close to $f$ and thus to each other, so that $\widehat{f}^A = \widehat{f}^{CK}$, which has better rates of convergence in this case, or $(f,f_{\varepsilon}) \notin \Fcal^{u,l}(X,\varepsilon,p)$, in which case we have no control over $\widehat{f}^{CK}$, and $\widehat{f}^A$ cannot perform worse than $\widehat{f}_{n,\rho}$.
A good choice is $C_{adapt}\geq 4 C_{\rho}$, for $C_{\rho}$ given in the following proposition.

\begin{proposition}
\label{prop:deviation}
For all $\rho\geq 1$, 
$\nu \in (0, \nuest]$, $\beta > 0$, $c_{\beta} >0$, $c(\nu) > 0$, $E>0$, $S>0$, there exists $C_{\rho} > 0$ such that
\begin{equation*}
\underset{\Qbb\in \Qcal^{(2d)} (\nu,c({\nu}),E)}
        {\sup_{\Pbb_{X} \in \Psi(\rho,S,\beta,c_{\beta})}}
        (\Pbb_{r(X)} * \Qbb)^{\otimes n}
        \left[ \| \widehat{f}_{n,\rho} - f_X \|^2_2 \geq C_{\rho} \left( \frac{\log\log(n)}{\log(n)} \right)^{2\beta/\rho} \right] =O\left( \frac{1}{n}\right).
\end{equation*}
\end{proposition}
The proof of Proposition \ref{prop:deviation} is detailed in section \ref{proof:thm:adaptnoise}.
%

We denote by $\Fcal^{u,l}(X,\varepsilon,p,\rho,S,\nu,c(\nu),E)$ the set of pairs of densities $(f_X,f_{\varepsilon})$ such that $(f_X, f_{\varepsilon}) \in \Fcal^{u,l}(X,\varepsilon,p)$, the probability distribution with density $f$ is in $\Psi(\rho,S,\beta_X, c_{\beta})$, and the probability distribution with two i.i.d. coordinates with density $f_{\varepsilon}$ is in $ \Qcal^{(2)}(\nu,c(\nu),E)$.
The following theorem proves that the estimator $\widehat{f}^A$ inherits the best rate of convergence.
\begin{theorem}
\label{thm:adaptnoise}
For all $p \geq 2$, $\rho\geq 1$, 
$\nu \in (0, \nuest]$, $\beta > 0$, $c_{\beta >0}$ $c(\nu) > 0$, $E>0$, $S>0$, for all $C_{adapt} \geq 4 C_{\rho}$,
\begin{equation}
\label{eq:adapt1}
    \sup_{(f,f_{\varepsilon}) \in \Fcal^{u,l}(X,\varepsilon,p,\rho,S,\nu,c(\nu),E) 
    } \Ebb_{(f,f_{\varepsilon})}\left[\|\widehat{f}^{A}-f \|_2^2\right] = O \left( n^{-\frac{p}{p+1}} (\log(n))^{\frac{p \gamma_1 + \gamma_3}{(p+1)\eta_X}} \right)
\end{equation}
where where $\gamma_1 = - 4 \beta_X+1-\eta_X$ and $\gamma_2 = \max  \lbrace \max \lbrace p(4 \beta_X + 4 \beta_{\varepsilon}+1 - \eta_X),0\rbrace + 1 - \eta_X, 0  \rbrace$
while
\begin{equation}
\label{eq:adapt2}
\underset{\Qbb\in \Qcal^{(2d)} (\nu,c({\nu}),E)}
        {\sup_{\Pbb_{X} \in \Psi(\rho,S,\beta,c_{\beta})}}
        \Ebb_{(\Pbb_{r(X)} * \Qbb)^{\otimes n}}
        \left[ \| \widehat{f}^A - f \|^2_2 \right] =O\left( \left( \frac{\log\log(n)}{\log(n)} \right)^{2\beta/\rho}\right).
\end{equation}
\end{theorem}
\begin{proof}
We first consider the case where $\Qbb\in \Qcal^{(2d)} (\nu,c({\nu}),E)$ and $\Pbb_{X} \in \Psi(\rho,S,\beta,c_{\beta})$.
The difference $\widehat{f}^A - f$ can be written as
\begin{equation*}
    \widehat{f}^A - f = (\widehat{f}^{CK} - \widehat{f}_{n,\rho}) 1|_{\widehat{f}^A = \widehat{f}^{CK}} + (\widehat{f}_{n,\rho} - f),
\end{equation*}
so that
\begin{align*}
\| \widehat{f}^A - f \|_2^2
    & \leq 2 \| \widehat{f}_{n,\rho} - f \|_2^2 + 2  \| (\widehat{f}^{CK} - \widehat{f}_{n,\rho}) 1|_{\widehat{f}^A = \widehat{f}^{CK}} \|_2^2 \\
    & \leq 2 \| \widehat{f}_{n,\rho} - f \|_2^2 + C_{adapt} \left( \frac{\log\log(n)}{\log(n)} \right)^{2\beta/\rho}
\end{align*}
and~\eqref{eq:adapt2} follows from Theorem \ref{thm:density}.
On the other hand, 
\begin{equation*}
\| \widehat{f}^A - f \|_2^2 = \| \widehat{f}^{CK} - f \|_2^2 1|_{\widehat{f}^A = \widehat{f}^{CK}} +  \| \widehat{f}_{n,\rho} - f\|_2^2 1|_{\widehat{f}^A = \widehat{f}_{n,\rho}}.
\end{equation*}
Since
\begin{equation*}
\|\widehat{f}_{n,\rho} - \widehat{f}^{CK} \|_2^2 \leq 2 \| \widehat{f}_{n,\rho} - f \|^2_2 + 2 \| \widehat{f}^{CK}-f\|^2_2,
\end{equation*}
the event $\widehat{f}^A = \widehat{f}_{n,\rho}$ can only be achieved if either we have $\| \widehat{f}_{n,\rho} - f \|^2_2 \geq C_{\rho} \left( \frac{\log\log(n)}{\log(n)} \right)^{2\beta/\rho}$ or if we have simultaneously $\| \widehat{f}^{CK} - f \|^2_2 \geq C_{\rho} \left( \frac{\log\log(n)}{\log(n)} \right)^{2 \beta/\rho}$ and $\| \widehat{f}_{n,\rho} - f \|^2_2 \leq C_{\rho} \left( \frac{\log\log(n)}{\log(n)} \right)^{2 \beta/\rho}$. Under this second possibility, $\| \widehat{f}_{n,\rho} - f \|^2_2 \leq \| \widehat{f}^{CK} - f \|^2_2$, so that

\begin{align*}
\| \widehat{f}^A - f \|^2_2
    &\leq 2 \| \widehat{f}^{CK} - f \|^2_2 +
\| \widehat{f}_{n,\rho} - f \|^2_2 1|_{\| \widehat{f}_{n,\rho} - f \|^2_2 \geq C_{\rho} \left( \frac{\log\log(n)}{\log(n)} \right)^{2 \beta/\rho}}
\\
    &\leq 2\| \widehat{f}^{CK} - f \|^2_2 + \text{diam}(\Upsilon_{\rho,S})^2 1|_{ \| \widehat{f}_{n,\rho} - f \|^2_2 \geq C_\rho \left( \frac{\log\log(n)}{\log(n)} \right)^{2 \beta/\rho} }
\end{align*}
and \eqref{eq:adapt1} follows from Theorem \ref{thm:Comte} and Proposition \ref{prop:deviation}.
\end{proof}

Note that Proposition \ref{prop:deviation} does not provide an explicit formula for $C_{\rho}$ and thus for $C_{adapt}$. Theorem \ref{thm:adaptnoise} proves that for a large enough $C_{adapt}$, the estimator $\widehat{f}^A$ adapts to the best of the two estimators. To use the methodology in practice would require a method to choose $C_{adapt}$ based on the observations, for instance one based on the cross-validation approach discussed in Section~\ref{sec:datadriven}. 

\section{Simulations}
\label{sec:simu}

The aim of this section is to assess the performance of our method on synthetic datasets.

Although good theoretical values for the parameters $m$ and $h$ are given in Theorem~\ref{thm:density}, in practice, other values may produce much better results. As such, a practical implementation of our method consists of two steps: constructing an estimator for several possible parameters $(m,\nu_{\text{est}}, h)$, and then selecting the ``best" parameters $(\hat{m},\hat{\nu}_{\text{est}},\hat{h})$, in the hope that the resulting estimator performs comparably to the best parameter.
Devising a data-driven selection procedure for the parameters $(m,\nu_{\text{est}}, h)$ is outside the scope of this numerical study, although we discuss ways to do it in Section~\ref{sec:datadriven}. For the rest of this section, we select the parameters $(m,\nu_{\text{est}}, h)$ ourselves, to get an idea of how the method performs when the best, or at least good, parameters are selected.

This section is divided into three parts: the first part (Sections~\ref{sec_simus_procedure} and~\ref{sec:examples}) explains how the estimator of the characteristic function and the target density of the signal are computed, introduces six synthetic datasets on which our method is assessed and the graphs of the resulting estimators. The second part (Section~\ref{sec_simus_comparaison}) compares the empirical risk of our estimator with the empirical risk of the estimator defined in \cite{MR3372551} on four other synthetic datasets covering several smoothness scenarios. 
Finally, we discuss how to select the parameters $(m,\nu_{\text{est}},h)$ in Section~\ref{sec:datadriven}.

All simulation codes are available in Python at \textbf{https://sites.google.com/view/jcapitaominiconi}

\subsection{Procedure}
\label{sec_simus_procedure}

We consider real-valued signals, that is $d=1$. For each dataset, we may consider several set of parameters $(m,\nu_{\text{est}},h)$. For each of them, the estimators are computed as follows.

For any integer $m>0$, $\phi \in \Upsilon_{\rho,S}$ and $t \in \Cbb$, $T_m \phi(t) = 1 + \sum_{k=1}^{m} \phi_k t^k$, where the coefficients satisfy $\phi_k \in \Rbb$ when $k$ is even and $i\phi_k \in \Rbb$ when $k$ is odd, due to the fact that for all $\phi \in \Upsilon_{\rho,S}$ and $t\in\Rbb$, $\overline{\phi(t)}=\phi(-t)$. 

For any polynomial $T_m \phi$ with coefficients as above, the integral $M_n(T_m\phi)$ is approximated by a Riemann sum over a regular grid with 8000$\times$8000 points. We minimize $M_n(T_m\phi)$ as a function of the coefficients $(\phi_k)_{1\leq k \leq m}$ with the function \textit{optimize.minimize} from the package \textit{scipy}, starting at the projection of the characteristic function of the signal $T_m \Phi_X$, with otherwise default parameters (for Python version 3.8.10 and \textit{scipy} version 1.10.1). Choosing this initialization is not doable in practical scenarios, but for a proof of concept of our method, it alleviates most known issues of approximate minimization algorithms such as getting stuck in suboptimal minima.
The estimator $\max(0,\widehat{f}_n)$ is computed using~\eqref{eq_inversion_Fourier} over a regular grid with 8000 points, whose bounds depend on the distribution of the signal: $[-5;5]$ for Gaussian signals in Section~\ref{sec:examples} and $[-3;3]$ in Section~\ref{sec_simus_comparaison}, $[-1;2]$ for $\beta(2,2)$ signal, $[-5;10]$ for Gamma signal and $[-5;5]$ for bigamma signal. 

In Section~\ref{sec_simus_comparaison}, we then compute the $L_2$ loss $\|\max(0,\widehat{f}_n)-f\|_2^2$, where $f$ is the target density, through a Riemann sum on the aforementioned regular grid.
Note that taking $\max(0,\widehat{f}_n)$ as our estimator instead of $\widehat{f}_n$ has almost no computational cost and can only improve the quadratic loss. Indeed, the loss  $\|\max(0,\widehat{f}_n)-f\|_2^2$ is upper bounded by the loss $\|\widehat{f}_n-f\|_2^2$  since $\max(0,\widehat{f}_n)$ is the projection of $\widehat{f}_n$ on the convex set of nonnegative functions. Note that the projection of $\widehat{f}_n$ on the set of probability densities, which is convex, would reduce the quadratic loss even further, but computing it is more involved and the result depends on the tails of $\widehat{f}_n$. 



\subsection{Some examples}
\label{sec:examples}


In the following, we write $\beta(2,2)$ the beta distribution with density $f_{\beta}(t) = c_{b} t(1-t)$ for $t \in (0,1)$ where $c_{b}$ is a constant such that $f_{\beta}$ is a density. We write $\mathcal{L}(0)$ the Laplace distribution with mean equal to $0$ and scale 1, that is, the distribution with density $f_{\mathcal{L}(0)}(t) = \frac{1}{2} \exp{(-|t|)}$. The uniform distribution on $(-1,3)$ is denoted by $\mathcal{U}(-1,3)$ and the Dirac distribution with point mass at $-1$ by $\delta_{(-1)}$. 

The six synthetic dataset we consider are composed of $n=500$ i.i.d. observations of model \eqref{eq:model}, one for each of the following settings:
\begin{itemize}
\item[(I)] $X \sim \mathcal{N}(0,1)$, for $i \in \lbrace 1,2 \rbrace$, $\varepsilon^{(i)} \sim \mathcal{N}(0,1)$, 
\item[(II)] $X \sim \mathcal{N}(0,1)$, for $i \in \lbrace 1,2 \rbrace$, $\varepsilon^{(i)} \sim \mathcal{L}(0)$, 
\item[(III)] $X \sim \mathcal{N}(0,1)$, for $i \in \lbrace 1,2 \rbrace$, $\varepsilon^{(i)} \sim \frac{1}{2} \delta_{(-1)} + \frac{1}{2} \mathcal{U}(-1,3)$, 
\item[(IV)] $X \sim \beta(2,2)$, for $i \in \lbrace 1,2 \rbrace$, $\varepsilon^{(i)} \sim \mathcal{N}(0,1)$, 
\item[(V)] $X \sim \beta(2,2)$, for $i \in \lbrace 1,2 \rbrace$, $\varepsilon^{(i)} \sim \mathcal{L}(0)$, 
\item[(VI)] $X \sim \beta(2,2)$, for $i \in \lbrace 1,2 \rbrace$, $\varepsilon^{(i)} \sim \frac{1}{2} \delta_{(-1)} + \frac{1}{2} \mathcal{U}(-1,3)$. 
\end{itemize}
Note that in the settings  (III) and (VI), the noise is not symmetric and its characteristic function vanishes on infinitely many points on the real line, which which is outside of the scope of \cite{MR2396811} and \cite{MR3372551}. 

We computed the estimator with $m = 15$, $\nu_{\text{est}} = 2$ and three different values of $h$. In the case of a Gaussian signal, we choose $h \in \lbrace 1,2,3 \rbrace$, and for the a signal with beta distribution, we choose $h \in \lbrace 5,7,9 \rbrace$. Results for (I) are displayed in 
Figure \ref{fig:GaussGauss}, for (II) in Figure \ref{fig:GaussLaplace}, for (III) in Figure \ref{fig:GaussUnif}, for (IV) in Figure \ref{fig:BetaGauss}, for (V) in Figure \ref{fig:BetaLaplace}, for (VI) in Figure \ref{fig:BetaUnif}.
Since Figures \ref{fig:GaussGauss}, \ref{fig:GaussLaplace}, and \ref{fig:GaussUnif}, as well as Figures \ref{fig:BetaGauss}, \ref{fig:BetaLaplace}, and \ref{fig:BetaUnif} are similar, we only show in this section Figures \ref{fig:GaussUnif} and \ref{fig:BetaUnif}, Figures \ref{fig:GaussLaplace}, \ref{fig:GaussGauss}, \ref{fig:BetaLaplace} and \ref{fig:BetaGauss} can be found in Appendix~\ref{section:App:graphs}.

We note several things: on the one hand, the noise seems to only have a very minor influence on the result of the estimation procedure. Indeed, Figures \ref{fig:GaussGauss}, \ref{fig:GaussLaplace}, and \ref{fig:GaussUnif}, as well as Figures \ref{fig:BetaGauss}, \ref{fig:BetaLaplace}, and \ref{fig:BetaUnif} are very similar even though the noise distributions are very different. On the other hand, the choice of $h$ has a significant impact on the estimated densities. 
Finally, even with a small amount of data ($n = 500$), our estimator manages to recover the density of the signal accurately, provided that $h$ and $\nu_{\text{est}}$ are chosen correctly.

\begin{figure}[ht]
\begin{subfigure}{0.57\textwidth}
\includegraphics[width=\linewidth]{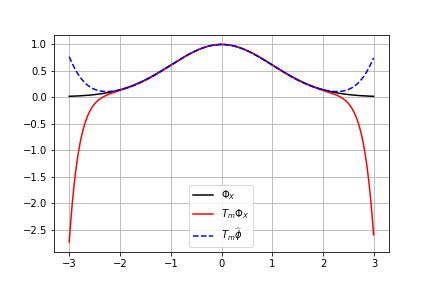}
\caption{Characteristic functions estimation} 
\end{subfigure}\hspace*{\fill}
\begin{subfigure}{0.57\textwidth}
\includegraphics[width=\linewidth]{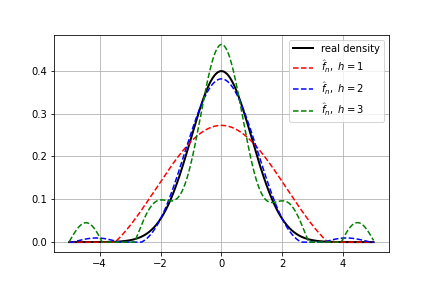}
\caption{Density estimation} 
\end{subfigure}

\caption{$X \sim \mathcal{N}(0,1)$, for $i \in \lbrace 1,2 \rbrace$, $\varepsilon^{(i)} \sim \frac{1}{2} \delta_{(-1)} + \frac{1}{2} \mathcal{U}(-1,3)$} \label{fig:GaussUnif}

\end{figure}

\begin{figure}[ht]
\begin{subfigure}{0.57\textwidth}
\includegraphics[width=\linewidth]{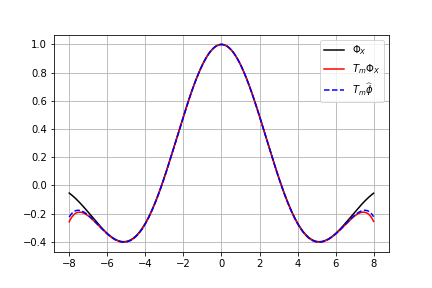}
\caption{Characteristic functions estimation estimation} 
\end{subfigure}\hspace*{\fill}
\begin{subfigure}{0.57\textwidth}
\includegraphics[width=\linewidth]{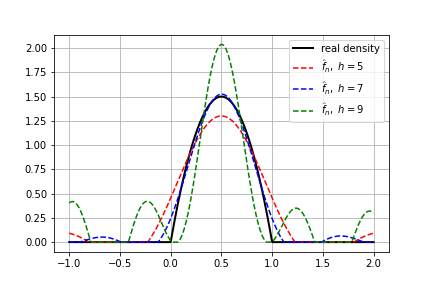}
\caption{Density estimation} 
\end{subfigure}

\caption{$X \sim \beta(2,2)$, for $i \in \lbrace 1,2 \rbrace$, $\varepsilon^{(i)} \sim \frac{1}{2} \delta_{(-1)} + \frac{1}{2} \mathcal{U}(-1,3)$} \label{fig:BetaUnif}

\end{figure}

\newpage

\subsection{Comparison with the estimator of \texorpdfstring{\cite{MR3372551}}{Comte and Kappus, 2015}}
\label{sec_simus_comparaison}

In the following, we compare our estimator with the density estimator presented in \cite{MR3372551}. We denote $\Gamma(\alpha, \beta)$ the gamma distribution with density $f(x) = \frac{\beta^{\alpha}}{\Gamma(\alpha)} x^{\alpha -1} e^{-\beta x} 1|_{(0,\infty)}(x)$. We also denote $b\Gamma(\alpha, \beta, \gamma, \delta)$ the bilateral Gamma distribution with parameters $\alpha$, $\beta$, $\gamma$, $\delta$ which is the distribution of $U-V$ where $U$ and $V$ are independent random variables such that $U$ has distribution $\Gamma(\alpha, \beta)$ and $V$ has distribution $\Gamma(\gamma, \delta)$.
We use the same target densities and errors than in \cite{MR3372551}, that is:
\begin{enumerate}
    \item $X$ has  distribution $\Gamma(4, 2)$ and $\varepsilon$ has distribution $b\Gamma(2, 2, 3, 3)$.
    \item $X$ has distribution  $b\Gamma(1, 1, 2, 2)$ and 
    $\varepsilon + 2$ has distribution $\sim \Gamma(4, 2)$ 
    (the location shift ensure that $\mathbb{E}[\varepsilon] = 0$.)
    \item $X$ has a standard normal distribution, $X \sim \mathcal{N}(0, 1)$ and $\varepsilon \sim b\Gamma(2, 2, 3, 3)$.
    \item $X$ has a standard normal distribution and $\varepsilon$ is the mixture of two normal distributions with parameters $(-2, 1)$ and $(2, 2)$ and equal weights $1/2$.
    We use the notation $\varepsilon \sim m\mathcal{N}(-2, 1, 2, 2)$.
\end{enumerate}

These four scenarios cover three of the possible cases for the smoothness of the signal and noise distributions: the first and second ones have ordinary smooth signal and noise, the third one supersmooth signal and ordinary smooth noise, and the last one supersmooth signal and noise.

To compare our estimator with the one of \cite{MR3372551}, we sample 500 i.i.d. repetitions of $n=1000$ observations, producing estimators $\widehat{f}_n^{(i)}$, $1 \leq i \leq 500$, and compute the empirical risk
\begin{equation}
\label{eq_empirical_risk}
\widehat{r} = \frac{1}{500} \sum_{i=1}^{500} \|\max(0,\widehat{f}_n^{(i)})-f\|_2^2.
\end{equation}
Table~\ref{tab:value} shows the comparison to the empirical risk of the oracle estimator of \cite{MR3372551}, denoted $\widehat{r}_{CK}$ in what follows, reported in their Table 4 (column $\widehat{r}^\text{or})$.

Computing our estimator requires choosing the parameters $(m,\nu_{\text{est}}, h)$. This is done as follows.
We consider the following range of parameters: $m\in\{3;4;5;\dots;15\}$, $h\in\{0.25;0.5;0.75;1;1.25;1.5;1.75;2\}$ and $\nu_{\text{est}}\in\{0.33;0.5;1;1.5;2;2.5;3;3.5;4;4.5\}$. For each possible set of parameters and each scenario, we sample one set of $n=1000$ observations, compute our estimator $\widehat{f}_n$ with these parameters, and compute the loss $\|\max(0,\widehat{f}_n)-f\|_2^2$. 

The resulting losses are given in Tables \ref{tab:GaussMixtureGauss1} to \ref{tab:GaussianBiGamma3} in Appendix~\ref{section:App:losses}. Examination of the tables allows to find a value of the parameters $(m,\nu_{\text{est}}, h)$ for which the error is minimum. We call this value the oracle value of $(m,\nu_{\text{est}}, h)$. Of course, this oracle value depends on the observations: it may change when sampling a new set of observations.

Ideally, comparison with the oracle values in \cite{MR3372551} should be done by computing $\widehat{r}$ using the oracle value of the parameters $(m,\nu_{\text{est}}, h)$ for each repetition. Unfortunately, computing the loss for each of the 1040 possible sets of parameters $(m,\nu_{\text{est}}, h)$ and for each of the 500 repetitions is very time consuming.
Instead, for each scenario, once Tables \ref{tab:GaussMixtureGauss1} to \ref{tab:GaussianBiGamma3} have been computed, we retain only four sets of parameters, given by four of the smallest losses, to run the 500 repetitions, while keeping these parameters distinct enough to retain some variety and ideally robustness. The four sets of parameters for each dataset are indicated in the last column of Table~\ref{tab:value}. This produces estimators $\widehat{f}_n^{(i,j)}$, $1 \leq i \leq 500$, $j \in \{1,2,3,4\}$. Finally, for each $i$, $\widehat{f}_n^{(i)}$ is taken as the one with the lowest loss among these four. This provides an--hopefully close--upper bound of the empirical risk of the oracle estimator, for a considerably reduced computational cost.

\begin{table}[ht]
\centering
\begin{tabular}{|c|c|c|c|c|}
\hline
& $100 \widehat{r}$ & 
C.I. for $100 \Ebb[\widehat{r}]$ & $100 \widehat{r}_{CK}$ & $(m,\nu_{\text{est}},h)$ \\
\hline
\hline
$X \sim \Gamma(4,2)$, & $0.863$ &  
$(0.856, 0.873)$ & $0.34$ & $(15,1,2)$, $(13,1,2)$, \\
$\varepsilon \sim b\Gamma(2,2,3,3)$ & & & & $(14,1,2)$, $(11,1,2)$ \\
\hline
$X \sim b\Gamma(1,1,2,2)$, & $0.990$ &  
$(0.978, 1.002)$ & $0.27$ &$(15,1,2)$, $(14,1,2)$, \\
$\varepsilon \sim \Gamma(4,2)-2$ & & & & $(13,1,2)$, $(15,1,1.75)$ \\
\hline
$X \sim \mathcal{N}(0,1)$, & $0.061$ & 
$(0.060, 0.062)$ & $0.19$ & $(14,3,2)$, $ (13,2.5,2)$, \\
$\varepsilon \sim b\Gamma(2,2,3,3)$ & & & & $(13,1.5,2)$, $(14,3.5,2)$ \\
\hline
$X \sim \mathcal{N}(0,1)$, & $0.064$ & 
$(0.064, 0.065)$ & $1.18$ & $(13,4,2)$, $ (15,3.5,2)$, \\
$\varepsilon \sim m\mathcal{N}(-2,1,2,2)$ & & & & $(14,3.5,2)$, $(14,4,2)$ \\
\hline

\end{tabular}
\caption{Comparison of the empirical risks of our estimator, defined in~\eqref{eq_empirical_risk}, and of the oracle estimator of~\cite{MR3372551}, with 95\% confidence interval for $100\Ebb[\widehat{r}]$, in each scenario. The last column contains the four sets of parameters used to compute our estimator.}
\label{tab:value}
\end{table}



Our procedure outperforms the estimator of \cite{MR3372551} for the two last lines of Table \ref{tab:value}, corresponding to a Gaussian signal.
For the first two lines, in the case of $\Gamma(4,2)$ signal or $b\Gamma(1,1,2,2)$ signal, even though our theory does not apply because the signals have sub-exponential tails, it should be noted that our procedure still works in practice. We illustrate this fact with the graphs of the different densities and their estimates calculated on the dataset used to compute Tables \ref{tab:GaussMixtureGauss1} to \ref{tab:GaussianBiGamma3} (thus ensuring the best selection of the parameters $(m,\nu_{\text{est}}, h)$), presented in Figures \ref{fig:aPlot}, \ref{fig:bPlot}, \ref{fig:cPlot} and \ref{fig:dPlot}.

The question of how to select the set of parameters $(m,\nu_{\text{est}}, h)$ in a data-driven manner is considered in Section \ref{sec:datadriven}.

\begin{figure}[ht]
\begin{subfigure}{0.57\textwidth}
\includegraphics[width=\linewidth]{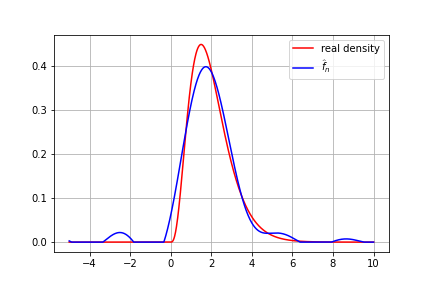}
\caption{\centering
$X \sim \Gamma(4,2)$, $\forall i \in \{1,2\}$,  $\varepsilon^{(i)} \sim b\Gamma(2,2,3,3)$ and $(m,\nu_{\text{est}},h) = (15,1,2)$.} \label{fig:aPlot}
\end{subfigure}\hspace*{\fill}
\begin{subfigure}{0.57\textwidth}
\includegraphics[width=\linewidth]{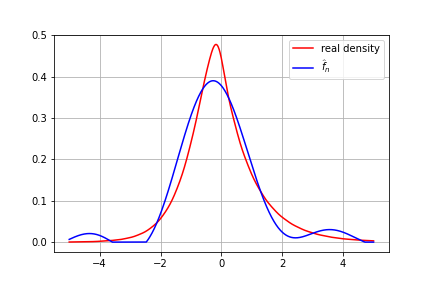}
\caption{\centering
    $X \sim b\Gamma(1,1,2,2)$, $\forall i \in \{1,2\}$,  $\varepsilon^{(i)} \sim \Gamma(4,2)-2$ and  $(m,\nu_{\text{est}},h) = (15,1,2)$.} \label{fig:bPlot}
\end{subfigure}

\caption{Plot of the target density (red) and the estimated density (blue).}
\end{figure}

\begin{figure}[ht]
\begin{subfigure}{0.57\textwidth}
\includegraphics[width=\linewidth]{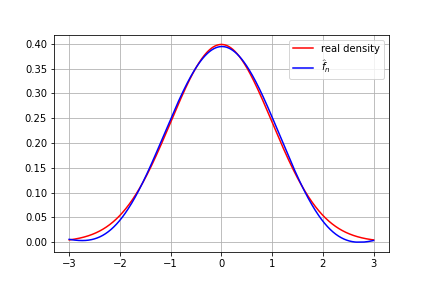}
\caption{\centering
$X \sim \mathcal{N}(0,1)$, $\forall i \in \{1,2\}$,  $\varepsilon^{(i)} \sim b\Gamma(2,2,3,3)$ and $(m,\nu_{\text{est}},h) = (14,3,2)$.} \label{fig:cPlot}
\end{subfigure}\hspace*{\fill}
\begin{subfigure}{0.57\textwidth}
\includegraphics[width=\linewidth]{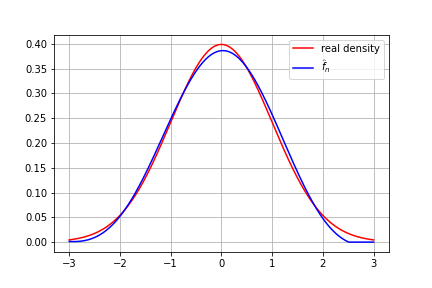}
\caption{\centering
$X \sim \mathcal{N}(0,1)$, $\forall i \in \{1,2\}$,  $\varepsilon^{(i)} \sim m\mathcal{N}(-2,1,2,2)$ and $(m,\nu_{\text{est}},h) = (13,4,2)$.} \label{fig:dPlot}
\end{subfigure}

\caption{Plot of the target density (red) and the estimated density (blue).}
\end{figure}

\newpage

\subsection{Data-driven selection of the parameters \texorpdfstring{$(m, \nu_{\text{est}},h)$}{}}
\label{sec:datadriven}

Theorem \ref{thm:density} states that $m$ has to be chosen of the order of $\log(n)/ \log(\log(n))$ when $n$ is large enough, with a constant depending on $\rho$.
In practice, and especially for small values of $n$, such a choice might not produce the best results, and thus we need a way to select a value of $m$ based on the observations themselves.
For instance, for $n=1000$, Theorem \ref{thm:density} gives $m \simeq \rho$, but the algorithm works better when we significantly increase the value of $m$. This is why we choose to consider values of $m$ between 3 and 15.

Considering the illustrations of Section \ref{sec:examples}, we also note that, given a fixed value of $\nu_{\text{est}}$, the value of $h$ that works best corresponds to the value of the point where $T_m \widehat{\phi}$ starts diverging from $\Phi_X$.
Finally, we believe that there exists an optimal $\nu_{\text{est}}$ to approximate the characteristic function $\Phi_X$, in the sense that there exists a $\nu_{\text{est}}^{\star}$ which minimizes the distance $\| \Phi_X - T_m \widehat{\phi} \|_{2,\nu_{\text{est}}}$, as can be seen in Figure \ref{fig:OptimalNu} in Appendix \ref{section:App:graphs}, where we took $m=4$.


Considering Tables \ref{tab:GaussMixtureGauss1} - \ref{tab:GaussianBiGamma3}, we note that for a Gaussian signal, the values of the loss remain relatively stable when varying the parameters 
$(m, \nu_{\text{est}}, h)$. We computed Tables \ref{tab:GaussianBiGammaMC1}-\ref{tab:GaussianBiGammaMC3} using another sample
in the case $X \sim \mathcal{N}(0,1)$, $\varepsilon \sim b\Gamma(2,2,3,3)$. This also illustrates the variability of the oracle set of parameters, as the set of parameters chosen based on Table \ref{tab:value} is not the best ones for this new sample, instead it is $(m, \nu_{\text{est}}, h) = (15,3.5,2)$.

In contrast, for a signal following a $\Gamma$ or $b\Gamma$ distribution, Tables \ref{tab:BiGammaGamma1} - \ref{tab:GammaBiGamma3} show that a careless choice of parameters can result in the loss exploding, for instance when taking large values of $m$.
However, we have observed when computing Table \ref{tab:value} that the oracle losses remain fairly stable across different samples for a $\Gamma$ or $b\Gamma$ signal, as can be seen from the relative size of the confidence intervals.


Thus, there is a need for a data driven choice of $(m,\nu_{\text{est}}, h)$. Here is a proposition. For each set of parameters $H=(m,\nu_{\text{est}}, h)$, compute an estimator of the Fourier transform of the distribution of $X$ given by
$\widehat{\Phi}_{H}$, obtained by minimizing $M_n (T_m \phi)$ with $M_n$ defined using $\nu_{\text{est}}$, and truncated to $0$ for $|t| \geq h$. This is what is used to compute $\widehat{f}_{H}$. Then, considering now the noise as the unknown signal, one can use the methodology proposed in \cite{MR2853732} to estimate the density of the noise. We denote $q$ the parameter $m$ in the procedure of \cite{MR2853732}, which also needs to be chosen based on the data. Let $\widehat{g}_{H,q,1}$ be the estimator of the density of the noise $\varepsilon^{(1)}$ using the first coordinate of the observations, and $\widehat{g}_{H,q,2}$ be the estimator of the density of the noise $\varepsilon^{(2)}$ using the second coordinate of the observations. 
This provides estimators $\widehat{p}_{H,q,1}=\widehat{f}_{H} * \widehat{g}_{H,q,1}$ of the density of the first coordinate of the observations
and $\widehat{p}_{H,q,2}=\widehat{f}_{H} * \widehat{g}_{H,q,2}$ of the density of the second coordinate of the observations. Here, $\widehat{f}_{H} * \widehat{g}_{H,q,i}$ denotes the convolution of the functions $\widehat{f}_{H}$ and $\widehat{g}_{H,q,i}$. Now, by partitioning the set of observations in three blocks with indices respectively in $E_1$, $E_2$ and $T$, one can compute the estimator $\widehat{f}_{H}$  using the observations with indices in $E_1$, compute the estimators $\widehat{g}_{H,q,1}$ and $\widehat{g}_{H,q,1}$ using the observations with indices in $E_2$, and then compute a cross-validation criterion
\begin{equation*}
CV(H,q)=\sum_{i\in T}\left(\log \widehat{p}_{H,q,1}(Y_{i}^{(1)})+\log \widehat{p}_{H,q,2}(Y_{i}^{(2)})\right).
\end{equation*}
Finally, we choose $H$ maximizing $H\mapsto \max \{CV(H,q), q\in \lbrace k/4 \pi, \ k \in \Nbb^{\star}, \ k \leq n^{1/4} \rbrace \}$.



\section{Discussion and perspectives}
\label{sec:discussion}

In this paper, we considered the repeated measurements model with two repetitions of the unknown signal, and we proposed a new estimation procedure for the distribution of the signal. Our estimator is consistent without any assumption on the noise distribution provided the noise has independent components, and for signals having a Laplace transform at $\lambda$ with growth at most $\exp(b \|\lambda\|^\rho)$ for some $\rho >0$. We provided theoretical results about minimax rates and adaptation procedures 
and presented simulation experiments.

In the case where more repetitions are available, it is possible to extend our procedure to take the additional information into accounts. We have investigated the possibility of using the sum of elementary criteria built using only groups of two repetitions as criterion to estimate the characteristic function of the signal. With this approach, taking $p > 2$ repetitions does not improve the rates of convergence of the estimators, only the constants. What could change the rates would be to take $p_n \rightarrow +\infty$ together with $n$; depending on the relationship between $n$ and $p_n$, we expect a transition to occur between our regime (which has logarithmic rates in $n$) and the limit where $p_n$ is considerably larger than $n$ ($p_n \gg n^\alpha$ for some large enough $\alpha > 0$), where the signal can be recovered by simple average and the problem devolves into standard density estimation (which has polynomial rates in $n$).

One important result of our findings is the surprisingly good behaviour of our estimator in simulations. These findings open up new avenues of research, based on the following questions:
\begin{itemize}
\item Is it possible to extend the identifiability result to nonparametric classes of distributions having heavier tails? 

\item Is it possible to adapt the choice of $(m,\nu_{\text{est}}, h)$ on classes of distributions involving ordinary smooth distributions for the noise and/or for the signal, when the decay of the characteristic function is known? In particular, slower decay should allow one to take larger values of $\nu_{\text{est}}$ and $h$. Can this adaptation be done using data-driven choices of the hyperparameters $(m,\nu_{\text{est}}, h)$?

\item How to design a computationally efficient implementation of the procedure? In particular, a data driven choice of the parameters such as the one in Section~\ref{sec:datadriven} requires the computation of the estimator for each set of parameters to be done.
\end{itemize}

\section{Proofs}
\label{sec:proofs}

\subsection{Proof of Theorem \ref{thm:id}}
\label{proof:thm:id}

In the following, we will write $\phi_i$ (resp. $\tilde{\phi_i}$) the characteristic function of $\varepsilon^(i)$ (resp $\tilde{\varepsilon}(i)$) for $i \in \lbrace 1,2 \rbrace$.
Since the distribution of $Y$ is the same under $\Pbb_{r(X)} * \Qbb$ and $\Pbb_{r(X')} * \tilde{\Qbb}$, for any $(t_1, t_2) \in \Rbb^d \times \Rbb^d$,
\begin{equation}
\label{eq:Y}
    \Phi_X(t_1 + t_2) \phi_1(t_1) \phi_2(t_{2}) = \Phi_{X'}(t_1 + t_2)  \tilde{\phi}_1(t_1) \tilde{\phi}_2(t_{2}),
\end{equation}
and 
by taking $t_2=0$, and $t_1=0$ in \eqref{eq:Y}, also
\begin{equation}
\label{eq:Y(1)}
    \Phi_X(t_1) \phi_1(t_1) = \Phi_{X'}(t_1) \tilde{\phi}_1(t_1),
\end{equation}
and
\begin{equation}
\label{eq:Y(2)}
    \Phi_X(t_{2}) \phi_2(t_2) = \Phi_{X'}(t_2) \tilde{\phi}_2(t_2).
\end{equation}
There exists a neighborhood $V$ of $0$ in $\Rbb^d \times \Rbb^d$ such that for all $(t_1, t_2) \in \Rbb^d \times \Rbb^d$, $\phi_1(t_1) \neq 0$, $\tilde{\phi}_1(t_1) \neq 0$, $\phi_2(t_{2}) \neq 0$ and $\tilde{\phi}_2(t_{2}) \neq 0$, so that, equations \eqref{eq:Y},  \eqref{eq:Y(1)} and \eqref{eq:Y(2)} imply that for any $(t_1, t_2) \in V$,
\begin{eqnarray*}
    & & \Phi_X(t_1 + t_{2}) \Phi_{X'}(t_1) \Phi_{X'}(t_{2}) = \Phi_{X'}(t_1 + t_2) \Phi_{X}(t_1) \Phi_{X}(t_2).
\end{eqnarray*}
Since $(z_1,z_2) \in \Cbb^d \times \Cbb^d \mapsto \Phi_{X}(z_1+z_2) \Phi_{X'}(z_1) \Phi_{X'}(z_2) - \Phi_{X'}(z_1+z_2) \Phi_{X}(z_1) \Phi_{X}(z_2) $ is a multivariate analytic function 
which is zero on a purely real neighborhood of $0$, then it is the null function on the whole multivariate complex space, see Lemma 25 in \cite{LCGL2021}, so that, for any $z_1, z_2$ in $\Cbb^d$,
\begin{equation}
\label{eq:id}
    \Phi_{X}(z_1+z_2) \Phi_{X'}(z_1) \Phi_{X'}(z_2) = \Phi_{X'}(z_1+z_2) \Phi_{X}(z_1) \Phi_{X}(z_2).
\end{equation}
Taking $z_2 = -z_1$, since $\Phi_{X}(0) = 1$ and $\Phi_{X}(-z_1) = \overline{\Phi_{X}(z_1)}$,
we get that, for all $z \in \Cbb^d$, 
\begin{equation}
\label{eq:module}
        |\Phi_{X'}(z)| = |\Phi_{X}(z)|.
\end{equation}
We set $R(z)=|\Phi_{X'}(z)| = |\Phi_{X}(z)|$ and define $\Theta(z) \in (-\pi, \pi]$ and  $\tilde{\Theta}(z) \in (-\pi,\pi]$ such that 
$\Phi_{X}(z) = R(z) \exp{(i \Theta(z))}$ and $\Phi_{X'}(z) = R(z) \exp{(i \tilde{\Theta}(z))}$. 
The functions $R$, $\theta$ and $\tilde{\Theta}$ are continuous on the open set where $R(z)\neq 0$, which includes a neighborhood of $0$ since $R(0)=1$. Note that  $\Theta(0) = 0$ and $\tilde{\Theta}(0) = 0$. Thus there exists $\delta$ such that $0 < \delta < \pi/6$, such that there there exist a neighborhood $A_{\delta}$ of $0$ in $\Cbb^d$, such that for all $(z_1,z_2)\in A_{\delta}^2$ and all $z\in A_{\delta}$, $R(z\neq 0$, $R(z_1+z_2) \neq 0$, and also 
$\Theta(z) \in (-\delta, \delta)$, $\tilde{\Theta}(z) \in (-\delta, \delta)$, $\Theta(z_1+z_2) \in(-\delta, \delta)$, and $\tilde{\Theta}(z_1+z_2) \in(-\delta, \delta)$.
Using equations \eqref{eq:id} and \eqref{eq:module}, we get that for all $z_1 \in A_{\delta}$ and $z_2 \in A_{\delta}$, 
\begin{equation*}
    \exp \lbrace i ( \Theta(z_1+z_2) + \tilde{\Theta}(z_1) + \tilde{\Theta}(z_2) ) \rbrace =  \exp \lbrace i ( \tilde{\Theta}(z_1+z_2) + \Theta(z_1) + \Theta(z_2) ) \rbrace, 
\end{equation*}
which gives
\begin{equation}
\label{eq:angle}
    \exp \lbrace i ( \Theta(z_1+z_2) + \tilde{\Theta}(z_1) + \tilde{\Theta}(z_2) ) -i ( \tilde{\Theta}(z_1+z_2) + \Theta(z_1) + \Theta(z_2) )\rbrace =  1.
\end{equation}
But since for all 
$z_1 \in A_{\delta}$ and $z_2 \in A_{\delta}$,
\begin{equation*}
    -6 \varepsilon < \Theta(z_1+z_2) - \tilde{\Theta}(z_1+z_2) + \tilde{\Theta}(z_1) - \Theta(z_1) + \tilde{\Theta}(z_2) - \Theta(z_2) < 6 \varepsilon,
\end{equation*}
\eqref{eq:angle} implies that for all $z_1 \in A_{\delta}$ and $z_2 \in A_{\delta}$,
\begin{equation*}
    \Theta(z_1+z_2) - \tilde{\Theta}(z_1+z_2) = \Theta(z_1) - \tilde{\Theta}(z_1) + \Theta(z_2) - \tilde{\Theta}(z_2).
\end{equation*}
Now using Theorem 2 in \cite{MR0208210} we get that there exists $\alpha \in \Rbb^d$ such that for all $z \in A_{\delta} \cap \Rbb^d$,  $\Theta(z) = \tilde{\Theta}(z) + \alpha^\top z$. Thus, for all $z \in A_{\delta} \cap \Rbb^d$, $\Phi_{X}(z) - \Phi_{X'}(z) \exp{(i \alpha^\top z)} = 0$. Since the function $z \in \Cbb^d \mapsto \Phi_{X}(z) - \Phi_{X'}(z) \exp{(i \alpha^\top z)}$ is a multivariate analytic function of $d$ variables which is zero on a purely real neighborhood of $0$, then it is the null function on the whole multivariate complex space, that is, 
\begin{equation*}
    \forall z \in \Cbb^d,\;\Phi_{X}(z) = \Phi_{X'}(z) \exp{(i \alpha^\top z)} = \Phi_{X'+\alpha}(z)
\end{equation*}
which ends the proof.

\subsection{Proof of Proposition \ref{prop:rate:char}}
\label{proof:thm:char}

First, note that all arguments of the proof of Proposition 6 in \cite{InferSupport} go through with unbounded $\rho_0$ as soon as identifiability holds. We choose $\Hcal$ as the set of multivariate analytic functions $f : \Cbb^{2d} \longrightarrow \Cbb$   such that there exists $F : \Cbb^{d} \longrightarrow \Cbb$ such that  forall $(t_1,t_2) \in \Cbb^d \times \Cbb^d$, $f(t_1,t_2) = F(t_1+t_2)$, for which identifiability holds without any bound on $\rho$ as our Theorem \ref{thm:id} proves. Also, the restrictions to $\Rbb^{2d}$ of functions in $\Hcal$  is a closed subset of $L_2([-\nu,\nu]^{2d})$. 
We then apply Proposition 6 in \cite{InferSupport} with this set $\Hcal$ after noticing that
\begin{equation*}
    \int_{[-\frac{\nu}{2},\frac{\nu}{2}]^d} |\widehat \Phi_{n,\rho'}(t) - \Phi_X(t)|^2 dt \leq \frac{1}{(2\nu)^d} \int_{[-\nu,\nu]^d \times [-\nu,\nu]^d} |\widehat \Phi_{n,\rho'}(t_1+t_2) - \Phi_X(t_1+t_2)|^2 dt_1dt_2.
\end{equation*}

\subsection{Proof of Theorem \ref{thm:density}}
\label{proof:thm:density}

In the following, for any positive real number $\alpha$, we write $\|f\|^2_{2,\alpha} = \int_{-\alpha}^{\alpha} |f(u)|^2 du$.
We can easily show that

\begin{align*}
\| \widehat{f}_{n,\rho} - f \|^2_2 & \leq \frac{1}{(4 \pi^2)^d} \|T_{m_{n,\rho}} \widehat \Phi_{n,\rho} - \Phi_X \|^2_{2,h_{n,\rho}} + \frac{1}{(4\pi^2)^d} \int_{\Rbb^d \setminus [-h_{n,\rho},h_{n,\rho}]^d} |\Phi_X(u)|^2 du\\
& \leq \frac{1}{(4 \pi^2)^d} \|T_{m_{n,\rho}} \widehat \Phi_{n,\rho} - \Phi_X \|^2_{2,h_{n,\rho}} + \frac{1}{(4 \pi^2)^d}\frac{c_{\beta}}{(1 + h_{n,\rho}^2)^{\beta}}\\
& \leq 2 \max\left\{ \frac{1}{(4 \pi^2)^d} \|T_{m_{n,\rho}} \widehat \Phi_{n,\rho} - \Phi_X \|^2_{2,h_{n,\rho}} 
, \frac{1}{(4 \pi^2)^d}\frac{c_{\beta}}{(1 + h_{n,\rho}^2)^{\beta}}\right\}.
\end{align*}
The end of the proof follows the proof of Theorem 3.2 of \cite{LCGL2021}, after noticing that all arguments go through with unbounded $\rho_0$.

\subsection{Proof of Theorem \ref{thm:lowerbound}}
\label{proof:thm:lowerbound}

Before to give the proof of Theoreme \ref{thm:lowerbound}, let us set the framework. \\ Assume that the coordinate of $\varepsilon$ are independent identically distributed with density
\begin{equation*}
    g : x \mapsto c_g \frac{1 + \cos(cx)}{(\pi^2 -(cx)^2)^2}
\end{equation*}
for some $c>0$, where $c_g$ is such that $g$ is a probability density with characteristic function
\begin{equation*}
    \Fcal[g] : t \mapsto \left [ (1- \Bigg | \frac{t}{c} \Bigg |) \cos(\pi \frac{t}{c}) + \frac{1}{\pi} \sin(\pi \Bigg | \frac{t}{c} \Bigg | )  \right ] 1_{[-c,c]}(t).
\end{equation*}
With an adequate choice of $c$, $Q \in \Qcal^{(2d)}(\nu,c(\nu),E)$.
Let $f_0$ and $f_n$ be the two different probability densities of $X$ and $\Qbb$ the distribution of the noise $\varepsilon$. We write $\Pbb_0$ and $\Pbb_n$ the distribution of $r(X)$ for $X$ having probability density $f_0$ and $f_n$.
Then, the minimax risk is lower bounded by
\begin{equation}
\label{borneinf}
    \frac{1}{4} \| f_0 - f_n \|^2_{L_{2}(\Rbb^d)} \left [ 1 - \frac{1}{2} \| (\Pbb_0 * \Qbb)^{\otimes n} - (\Pbb_n * \Qbb)^{\otimes n} \|_{L_1(\Rbb^{2d})} \right ]
\end{equation}
Let us now provide the chosen densities $f_0$ and $f_n$.
Let $u : x \in \Rbb \mapsto c_u \exp{(-1/(1-x^2))}1_{(-1,1)}(x)$  with $c_u$ constant such that $u$ is a probability density. 
For all $b > 0$ and $x \in \Rbb$, define $u_b(x) = b u(bx)$. For $x \in \Rbb$ and a sequence $(K_n)_n$, 
define $\zeta_0 = u_b(x)$ and $\zeta_n(x) = \zeta_0(x) + \alpha_n (P_{K_{n}} * u_b)(x)$ where for $k \geq 0$ and $P_k$ is the Legendre polynomial of order $k$ in $(-1,1)$. Note that $\zeta_n$ is nonnegative as soon as $\alpha_n \leq 1/\| P_{K_{n}} \|_{\infty}$. 
Moreover, $\zeta_0$ has integral equal to 1 and since $P_{K_{n}}$ is orthogonal to $P_0$ (which is a constant function) in $L_2(\Rbb)$ the integral of $P_{K_{n}} * u_b$ is zero, then $\zeta_n$ is a probability density. 
now define, for all $x \in \Rbb^d$,
\begin{equation*}
f_0(x) = \prod_{j=1}^d \zeta_0(x_j) \ \text{and} \ f_n(x) = \zeta_n(x_1) \prod_{j=2}^d \zeta_0(x_j).
\end{equation*}
Let us recall some useful properties that can be found in \cite{LCGL2021} and \cite{MR2319471}.\\
First, $f_0$ and $f_n$ are probability densities with characteristic function in $\Upsilon_{\rho,S,d}$ (see Lemma 14 in \cite{LCGL2021}).  Since for $i \in \lbrace 0, 1 \rbrace$, $t=(t_1,t_2) \in \Rbb^{d} \times \Rbb^{d}$, $\Fcal[\Pbb_i](t) = \Fcal[f_i](t_1+t_2)$, $\Pbb_0$ and $\Pbb_n$ are probability distribution with characteristic function in $\Upsilon_{\rho,S,2d}$.
Then,
noticing that $\sup_{(t_1,t_2) \in \Rbb^d \times \Rbb^d} |\Fcal[u](t_1+t_2)|^2 (1 + \|t_1 + t_2\|^2)^{\beta}$ is finite,  the same proof as that of Theorem 2 in \cite{MR2319471} gives that, for all $\beta > 0$, $c_{\beta}>0$, there exists $x_0>0$ and $C_0>0$ such that $\Fcal[\Pbb_0]$ and $\Fcal[\Pbb_n]$ belong to $\Psi(1,T,\beta,c_{\beta})$ as soon as the two following assumptions are met
\begin{equation}
\label{eq:cond1}
\alpha_n \leq C_0 b_n^{-\beta} K_n^{1/2} \ \text{and} \ \alpha_n \leq 1/\| P_{K_{n}} \|_{\infty}.
\end{equation}

Finally, as it is used in \cite{LCGL2021}, we have
\begin{equation*}
     1 - \frac{1}{2} \| (\Pbb_0 * \Qbb)^{\otimes n} - (\Pbb_n * \Qbb)^{\otimes n} \|_{L_1(\Rbb^{2d})} \geq \left ( 1-\frac{1}{2} \|(\Pbb_0 * \Qbb) - (\Pbb_n * \Qbb) \|_{L_1(\Rbb^{2d})} \right )^n
\end{equation*}
Thus, using \eqref{borneinf}, if $(\alpha_n)_n$, $(b_n)_n$ and $(K_n)_n$ are chosen such that 
\begin{equation}
\label{eq:diff1/n}
    \int_{\Rbb^{2d}} |(\Pbb_0 * \Qbb)(x) - (\Pbb_n * \Qbb)(x)|dx = O\left( \frac{1}{n} \right)
\end{equation}
with \eqref{borneinf}, the minimax risk based on $n$ observations is lower bounded by $c \|f_0 - f_n\|_{L_2(\Rbb^d)}$ for some constant $c>0$. We prove in Lemma \ref{lemma:diffDistr}  below that a good choice of $K_n$ makes \eqref{eq:diff1/n} hold true. 
Then we have
\begin{equation}
\label{eq:distThm}
    \|f_0-f_n\|_{L_{2}(\Rbb^d)}^2 = \alpha_n^2 \|\zeta_0\|_{L_{2}(\Rbb)}^{2(d-1)} \|\zeta_n - \zeta_0\|_{L_{2}(\Rbb)}^2,
\end{equation}
so that using $\|\zeta_0\|_{L_{2}(\Rbb)}^{2(d-1)} = 1$ and the end of proof of Theorem 2 in \cite{MR2319471}, $\|\zeta_n - \zeta_0\|_{L_{2}(\Rbb)}^2 \geq K_n^{-2 \beta}$ by choosing $b_n = c_b K_n$ and $\alpha_n = C_0 b_n^{-\beta} K_n^{1/2} $ for $c_b$ small enough. \eqref{eq:distThm} leads to to

\begin{equation*}
    \|f_0-f_n\|_{L_{2}(\Rbb^d)}^2 \geq K_n^{-2\beta} = C \left(\frac{\log\log(n)}{\log(n)} \right)^{-2\beta}
\end{equation*}
for some constant $C>0$. The condition $\alpha_n \leq 1/\|P_{K_{n}} \|_{\infty}$ corresponds to $\beta \geq 1/2$.
\begin{lemma}
\label{lemma:diffDistr}
Let
\begin{equation}
\label{eq:Kn}
K_n = C_0 \left( \frac{\log(n)}{\log \log(n)} \right)
\end{equation}
then, for $\alpha_n \leq 1$ and $b_n \geq 1$,
\begin{equation*}
    \int_{\Rbb^{2d}} |(\Pbb_0 * \Qbb)(x) - (\Pbb_n * \Qbb)(x)|dx = O\left( \frac{1}{n} \right)
\end{equation*}
\end{lemma}
\begin{proof}
Following the proof of Lemma 16 of \cite{LCGL2021}, we have for all $\eta \in \lbrace 0,1 \rbrace^{2d}$,
\begin{equation*}
    \int_{\Rbb^{2d}} |(\Pbb_0 * \Qbb)(x) - (\Pbb_n * \Qbb)(x)|dx \leq \left ( \sum_{0 \leq \eta' \leq \eta} \int_{\Rbb^{2d}} |\Bigg( \prod_{j=1}^{2d} \partial_j^{\eta'_{j}} \Bigg) ((\Fcal[\Pbb_0] - \Fcal[\Pbb_n]) \Fcal[\Qbb])(t)|^2 dt \right )^{1/2}
\end{equation*}
Since for $x = (x_1,x_2) \in \Rbb^{2d}$ with $x_1 \in \Rbb^d$ and $x_2 \in \Rbb^d$, we have $\Fcal[\Pbb_0](x) = \Fcal[f_0](x_1+x_2)$ and $\Fcal[\Pbb_n](x) = \Fcal[f_n](x_1+x_2)$,
\begin{multline*}
    \sum_{0 \leq \eta' \leq \eta} \int_{\Rbb^{2d}} |\Bigg( \prod_{j=1}^{2d} \partial_j^{\eta'_{j}} \Bigg) ((\Fcal[\Pbb_0] - \Fcal[\Pbb_n]) \Fcal[\Qbb])(t)|^2 dt \\
    = \sum_{0 \leq \eta' \leq \eta} \int_{\Rbb^{d} \times \Rbb^{d}} |\Bigg( \prod_{j=1}^{2d} \partial_j^{\eta'_{j}} \Bigg) (\Fcal[f_0] - \Fcal[f_1]) (t_1+t_2) \Fcal[\Qbb](t_1,t_2)|^2 d(t_1,t_2).
\end{multline*}
Since $\Fcal[g]$ and $\Fcal[g]'$ are supported on $[-c,c]$, we have for all $\eta \in \lbrace 0,1 \rbrace^{2d}$, 
\begin{align*}
    &\sum_{0 \leq \eta' \leq \eta} \int_{\Rbb^{d} \times \Rbb^{d}} |\Bigg( \prod_{j=1}^{2d} \partial_j^{\eta'_{j}} \Bigg) (\Fcal[f_0] - \Fcal[f_1]) (t_1+t_2) \Fcal[Q](t_1,t_2)|^2 d(t_1,t_2) \\
    &\leq c_d \sum_{0 \leq \eta' \leq \eta} \int_{[-c,c]^d \times [-c,c]^d} |\Bigg( \prod_{j=1}^{2d} \partial_j^{\eta'_{j}} \Bigg) (\Fcal[f_0] - \Fcal[f_1]) (t_1+t_2)|^2 dt_1 dt_2\\
    &= c_d \sum_{0 \leq \eta' \leq \eta} \int_{[-c,c]^d \times [-c,c]^d} |\Bigg( \prod_{j=1}^{2d} \partial_j^{\eta'_{j}} \Bigg) (\Fcal[\zeta_0] - \Fcal[\zeta_n])(x_1 + x_{d+1}) \prod_{j=2}^{d} \Fcal[\zeta_0](x_j + x_{j+d}) |^2 dx \\
    &= c_d \sum_{0 \leq \eta' \leq \eta} \int_{[-c,c]^d \times [-c,c]^d} |(\Fcal[\zeta_0] - \Fcal[\zeta_n])^{(\eta'_{1} + \eta'_{d+1})}(x_1 + x_{d+1})|^2 \prod_{j=2}^{d} |\Fcal[\zeta_0]^{(\eta'_{j} + \eta'_{j+d})}(x_j + x_{j+d})|^2 dx 
\end{align*}
Finally, there exists a constant $c_d$ such that
\begin{multline*}
     \int_{\Rbb^{2d}} |(\Pbb_0 * \Qbb)(x) - (\Pbb_n * \Qbb)(x)|dx  \leq c_d \Bigg ( \int_{[-c,c]^2} |(\Fcal[\zeta_0] - \Fcal[\zeta_n])(x+y)|^2 dxdy  \\
     + \int_{[-c,c]^2} |(\Fcal[\zeta_0] - \Fcal[\zeta_n])^{'}(x+y)|^2 dxdy + \int_{[-c,c]^2} |(\Fcal[\zeta_0] - \Fcal[\zeta_n])^{''}(x+y)|^2 dxdy \Bigg )^{1/2}.
\end{multline*}
By making the change of variables $u = x+y$ and $v = y$, we get that there exists a constant $c_d'$ such that
\begin{align*}
     \int_{\Rbb^{2d}} |(\Pbb_0 * \Qbb)(x) - (\Pbb_n * \Qbb)(x)|dx  \leq c_d' \Bigg ( \int_{[-2c,2c]} |(\Fcal[\zeta_0] - \Fcal[\zeta_n])(u)|^2 du  \\
     + \int_{[-2c,2c]} |(\Fcal[\zeta_0] - \Fcal[\zeta_n])^{'}(u)^2 du + \int_{[-2c,2c]} |(\Fcal[\zeta_0] - \Fcal[\zeta_n])^{''}(u)|^2 du \Bigg )^{1/2}.
\end{align*}
Using that for all $u \in \Rbb$, $\Fcal[\zeta_0] - \Fcal[\zeta_n])(u) = \alpha_n \Fcal[P_{K_n}1_{(-1,1)}](u) \Fcal[u_{b_{n}}](u) = \alpha_n \Fcal[P_{K_n}1_{(-1,1)}](u) \Fcal[u](\frac{u}{b_n})$, we have that for $b_n$ large enough, there exists a constant $c_d''$ such that
\begin{align*}
     \int_{\Rbb^{2d}} |(\Pbb_0 * \Qbb)(x) - (\Pbb_n * \Qbb)(x)|dx  \leq c_d'' \alpha_n \Bigg ( \int_{[-2c,2c]} |\Fcal[P_{K_{n}}](u)|^2 du  \\
     + \int_{[-2c,2c]} |\Fcal[P_{K_{n}}]^{'}(u)^2 du + \int_{[-2c,2c]} |\Fcal[P_{K_{n}}]^{''}(u)|^2 du \Bigg )^{1/2}
\end{align*}

Following \cite{MR2319471} proof of Theorem 2, we have the following bounds, for constants $c>0$ and $C>0$
\begin{equation*}
    |\Fcal[P_{K_{n}}](u)| \leq c \left(\frac{C u}{K_n} \right)^{K_{n}} \ , \ |\Fcal[P_{K_{n}}]'(u)| \leq c \left(\frac{C u}{K_n-1} \right)^{K_{n}-1},
\end{equation*}
and
\begin{equation*}
    |\Fcal[P_{K_{n}}]''(u)| \leq c \left(\frac{C u}{K_n-2} \right)^{K_{n}-2}.
\end{equation*}
Finally, we get that there exists $c'$ such that for $K_n$ large enough,
\begin{equation*}
    \int_{\Rbb^{2d}} |(\Pbb_0 * Q)(x) - (\Pbb_n * Q)(x)|dx  \leq c' \alpha_n \left( \frac{C}{K_n-2} \right)^{K_{n}-2}
\end{equation*}

Then \eqref{eq:diff1/n} holds if $\alpha_n \leq 1$ and $K_n$ is chosen, for some large constant $C'$, as

\begin{equation*}
    K_n = C' \left( \frac{\log(n)}{\log\log(n)} \right) .
\end{equation*}
\end{proof}

\subsection{Proof of Proposition \ref{prop:deviation}}
\label{proof:thm:adaptnoise}

Following the proof of Theorem 3.2 of \cite{LCGL2021}, after noticing that all arguments go through with unbounded $\rho_0$, there exists constants $C > 0$ and $n_0$ such that for all $n \geq n_0$,
\begin{equation*}
    \| \widehat{f}_{n,\rho} - f^{\star} \|^2_2  \leq C \max \Bigg \lbrace \exp(-m_{n,\rho}), (2m_{n,\rho}/\rho)^{2m_{n,\rho}/\rho} \exp(-m_{n,\rho}) \| \widehat \Phi_{n,\rho} - \Phi \|^2_{2,\frac{\nu}{2}}, m_{n,\rho}^{-2\beta/\rho} \Bigg \rbrace.
\end{equation*}
By Proposition~\ref{prop:rate:char}, taking and $\delta, \delta''$ such that $(1-\delta)(1-\delta'') > 1/2$ and $x = \log(n)$, we obtain that up to changing the constant $C$, with probability at least $1-2 n^{-1}$,

\begin{equation}
\label{eq:fvitesse}
    \| \widehat{f} - f^{\star} \|^2_2  \leq C \max \Bigg \lbrace \exp(-m_{n,\rho}), (2m_{n,\rho}/\rho)^{2m_{n,\rho}/\rho} \exp(-m_{n,\rho}) \left ( \frac{\log(n)}{n^{1-\delta''}} \right )^{1-\delta}, m_{n,\rho}^{-2\beta/\rho} \Bigg \rbrace.
\end{equation}
By definition, $m_{n,\rho} \leq \frac{\rho}{4} \frac{\log(n)}{\log\log(n)}$, so that on the event where \eqref{eq:fvitesse} is true, there exists $n_0$ such that for $n \geq n_0$,

\begin{equation}
\label{eq:densityproba}
        \| \widehat{f} - f^{\star} \|^2_2  \leq C \max \Bigg \lbrace \exp(-m_{n,\rho}) \left [ 1 \vee n^{1/2} \left ( \frac{\log(n)}{n^{1-\delta''}} \right )^{1-\delta} \right ], m_{n,\rho}^{-2\beta/\rho} \Bigg \rbrace.
\end{equation}
Therefore, given the choice of $m_{n,\rho}$,

\begin{equation*}
\underset{\Qbb\in \Qcal^{(2d)} (\nu,c({\nu}),E)}
        {\sup_{\Pbb_{X} \in \Psi(\rho,S,\beta,c_{\beta}) }}
    \Pbb_{(\Pbb_{r(X)} * \Qbb)^{\otimes n}}\left[ m_{n,\rho}^{2\beta/\rho} \| \widehat{f}_n - f \|^2_{2} > C \right]
        \leq \frac{2}{n},
\end{equation*}
and the Proposition follows.

\newpage

\appendix
\section{Appendix}
\label{section:App}

\subsection{Graphs of the estimators of the characteristic function and density of the signal}
\label{section:App:graphs}

\begin{figure}[ht]
\begin{subfigure}{0.57\textwidth}
\includegraphics[width=\linewidth]{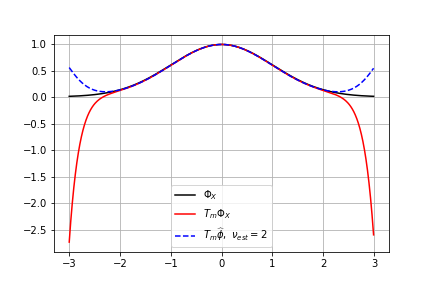}
\caption{Characteristic functions estimation} 
\end{subfigure}\hspace*{\fill}
\begin{subfigure}{0.57\textwidth}
\includegraphics[width=\linewidth]{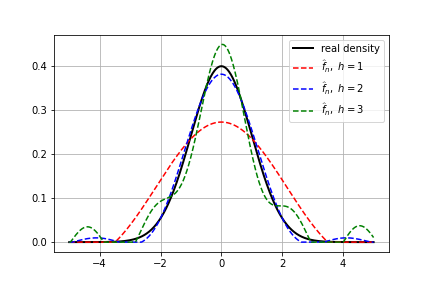}
\caption{Density estimation} 
\end{subfigure}

\caption{$X \sim \mathcal{N}(0,1)$, for $i \in \lbrace 1,2 \rbrace$, $\varepsilon^{(i)} \sim \mathcal{L}(0)$} \label{fig:GaussLaplace}

\end{figure}

\begin{figure}[ht]
\begin{subfigure}{0.57\textwidth}
\includegraphics[width=\linewidth]{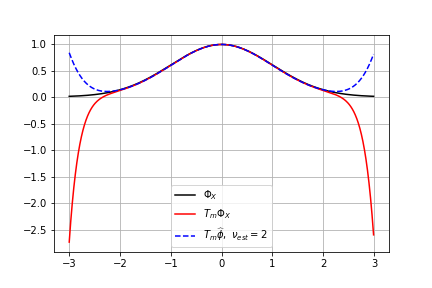}
\caption{Characteristic functions estimation} 
\end{subfigure}\hspace*{\fill}
\begin{subfigure}{0.57\textwidth}
\includegraphics[width=\linewidth]{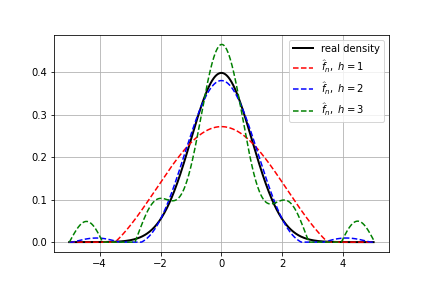}
\caption{Density estimation} 
\end{subfigure}

\caption{$X \sim \mathcal{N}(0,1)$, for $i \in \lbrace 1,2 \rbrace$, $\varepsilon^{(i)} \sim \mathcal{N}(0,1)$} \label{fig:GaussGauss}

\end{figure}

\begin{figure}[ht]
\begin{subfigure}{0.57\textwidth}
\includegraphics[width=\linewidth]{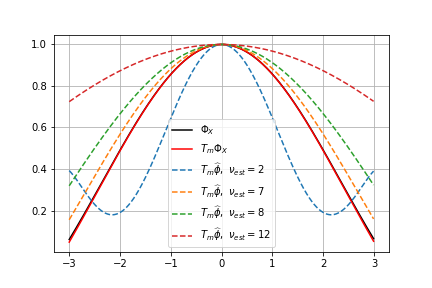}
\caption{$X \sim \beta(2,2)$, for $i \in \lbrace 1,2 \rbrace$, $\varepsilon^{(i)} \sim \mathcal{L}(0)$} 
\end{subfigure}\hspace*{\fill}
\begin{subfigure}{0.57\textwidth}
\includegraphics[width=\linewidth]{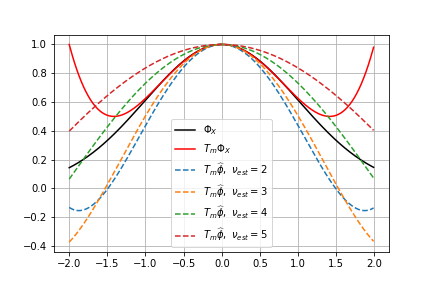}
\caption{$X \sim \mathcal{N}(0,1)$, for $i \in \lbrace 1,2 \rbrace$, $\varepsilon^{(i)} \sim \mathcal{L}(0)$} 
\end{subfigure}

\caption{Characteristic function estimation for different values of $\nu_{\text{est}}$.} \label{fig:OptimalNu}
\end{figure}

\newpage

\begin{figure}[ht]
\begin{subfigure}{0.57\textwidth}
\includegraphics[width=\linewidth]{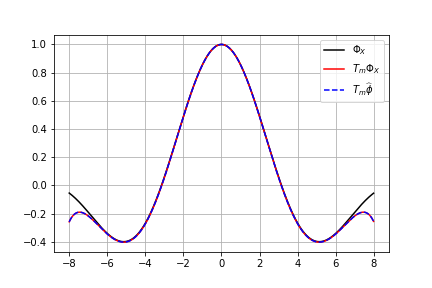}
\caption{Characteristic functions estimation} 
\end{subfigure}\hspace*{\fill}
\begin{subfigure}{0.57\textwidth}
\includegraphics[width=\linewidth]{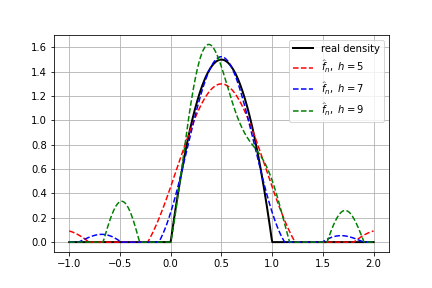}
\caption{Density estimation} 
\end{subfigure}

\caption{$X \sim \beta(2,2)$, for $i \in \lbrace 1,2 \rbrace$, $\varepsilon^{(i)} \sim \mathcal{N}(0,1)$} \label{fig:BetaGauss}

\end{figure}

\begin{figure}[ht]
\begin{subfigure}{0.57\textwidth}
\includegraphics[width=\linewidth]{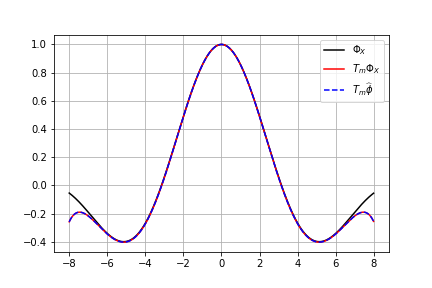}
\caption{Characteristic functions estimation estimation} 
\end{subfigure}\hspace*{\fill}
\begin{subfigure}{0.57\textwidth}
\includegraphics[width=\linewidth]{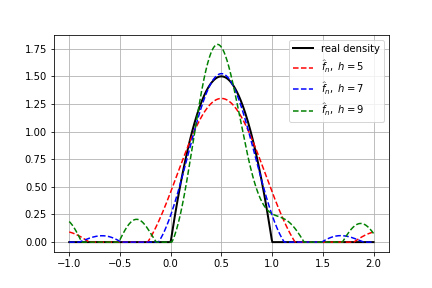}
\caption{Density estimation} 
\end{subfigure}

\caption{$X \sim \beta(2,2)$, for $i \in \lbrace 1,2 \rbrace$, $\varepsilon^{(i)} \sim \mathcal{L}(0)$} \label{fig:BetaLaplace}

\end{figure}

\newpage

\subsection{Loss tables} 
\label{section:App:losses}

The following tables \ref{tab:GaussMixtureGauss1}-\ref{tab:GaussianBiGamma3} are used to find the oracle triplet $(m,\nu_{\text{est}},h)$. They contain the loss $\| \max(\widehat{f}_n, 0) - f\|_2^2$, and are used to compute the oracle triplet $(m,\nu_{\text{est}},h)$. The tables are read as follows: the values of $m$ are on the y-axis, and the values of $\nu_{\text{est}}$ are on the x-axis. For each pair $(m, \nu_{\text{est}})$, there is a sub table that contains the values of the $L_2$ norm for $h \in \{0.25, 0.5, 0.75, 1, 1.25, 1.5, 1.75, 2\}$ as follows.

\begin{center}
\begin{tabular}{ |c||c| }
\hline
  & $\nu_{\text{est}}$ \\
\hline
\hline
$m$ & $h = 0.25$ \\ 
\cline{2-2}
& $h = 0.5$ \\ 
\cline{2-2}
& $h = 0.75$ \\ 
\cline{2-2}
& $h = 1$ \\
\cline{2-2}
& $h = 1.25$ \\
\cline{2-2}
& $h = 1.5$ \\
\cline{2-2}
& $h = 1.75$ \\
\cline{2-2}
& $h = 2$\\
\hline
\end{tabular}
\end{center}


\newpage

\begin{landscape}
\begin{table}
\begin{tabular}
{|c||c|c|c|c|c|c|c|c|c|c|}
\hline
 \multicolumn{11}{|c|}{$\nu_{\text{est}}$} \\
 \hline
 \multicolumn{1}{|c|}{} & $0.33$ & $0.5$ & $1$ & $1.5$ & $2$ & $2.5$ & $3$ & $3.5$ & $4$ & $4.5$\\
\hline
\hline

{}{}{$ 3 $} & $ 1.62 \cdot 10^{-1} $ & $ 1.62 \cdot 10^{-1} $ & $ 1.62 \cdot 10^{-1} $ & $ 1.62 \cdot 10^{-1} $ & $ 1.62 \cdot 10^{-1} $ & $ 1.61 \cdot 10^{-1} $ & $ 1.61 \cdot 10^{-1} $ & $ 1.61 \cdot 10^{-1} $ & $ 1.61 \cdot 10^{-1} $ & $ 1.61 \cdot 10^{-1} $ \\
\cline{2-11}
& $ 1.02 \cdot 10^{-1} $ & $ 1.02 \cdot 10^{-1} $ & $ 1.02 \cdot 10^{-1} $ & $ 1.01 \cdot 10^{-1} $ & $ 1.00 \cdot 10^{-1} $ & $ 9.94 \cdot 10^{-2} $ & $ 9.84 \cdot 10^{-2} $ & $ 9.80 \cdot 10^{-2} $ & $ 9.79 \cdot 10^{-2} $ & $ 9.78 \cdot 10^{-2} $ \\
\cline{2-11}
& $ 6.63 \cdot 10^{-2} $ & $ 6.64 \cdot 10^{-2} $ & $ 6.62 \cdot 10^{-2} $ & $ 6.41 \cdot 10^{-2} $ & $ 6.34 \cdot 10^{-2} $ & $ 6.10 \cdot 10^{-2} $ & $ 5.90 \cdot 10^{-2} $ & $ 5.81 \cdot 10^{-2} $ & $ 5.79 \cdot 10^{-2} $ & $ 5.77 \cdot 10^{-2} $ \\
\cline{2-11}
& $ 4.34 \cdot 10^{-2} $ & $ 4.40 \cdot 10^{-2} $ & $ 4.33 \cdot 10^{-2} $ & $ 3.83 \cdot 10^{-2} $ & $ 3.65 \cdot 10^{-2} $ & $ 3.13 \cdot 10^{-2} $ & $ 2.78 \cdot 10^{-2} $ & $ 2.64 \cdot 10^{-2} $ & $ 2.61 \cdot 10^{-2} $ & $ 2.58 \cdot 10^{-2} $ \\
\cline{2-11}
& $ 3.79 \cdot 10^{-2} $ & $ 3.80 \cdot 10^{-2} $ & $ 3.66 \cdot 10^{-2} $ & $ 2.61 \cdot 10^{-2} $ & $ 2.25 \cdot 10^{-2} $ & $ 1.39 \cdot 10^{-2} $ & $ 1.04 \cdot 10^{-2} $ & $ 9.63 \cdot 10^{-3} $ & $ 9.67 \cdot 10^{-3} $ & $ 9.79 \cdot 10^{-3} $ \\
\cline{2-11}
& $ 5.82 \cdot 10^{-2} $ & $ 5.36 \cdot 10^{-2} $ & $ 5.02 \cdot 10^{-2} $ & $ 2.81 \cdot 10^{-2} $ & $ 2.13 \cdot 10^{-2} $ & $ 7.56 \cdot 10^{-3} $ & $ 6.22 \cdot 10^{-3} $ & $ 8.05 \cdot 10^{-3} $ & $ 8.98 \cdot 10^{-3} $ & $ 1.03 \cdot 10^{-2} $ \\
\hline

\end{tabular}
\caption{$X \sim \mathcal{N}(0,1), \ \forall i \in \lbrace 1,2 \rbrace, \  \varepsilon^{(i)} \sim m \mathcal{N}(2,2,3,3), \ m \in \lbrace 3,4,5 \rbrace$}
\label{tab:GaussMixtureGauss1}
\end{table}
\end{landscape}

\begin{landscape}
\begin{table}
\begin{tabular}
{|c||c|c|c|c|c|c|c|c|c|c|}
\hline
 \multicolumn{11}{|c|}{$\nu_{\text{est}}$} \\
 \hline
 \multicolumn{1}{|c|}{} & $0.33$ & $0.5$ & $1$ & $1.5$ & $2$ & $2.5$ & $3$ & $3.5$ & $4$ & $4.5$\\
\hline
\hline
{}{}{$ 6 $} & $ 1.62 \cdot 10^{-1} $ & $ 1.62 \cdot 10^{-1} $ & $ 1.62 \cdot 10^{-1} $ & $ 1.62 \cdot 10^{-1} $ & $ 1.62 \cdot 10^{-1} $ & $ 1.62 \cdot 10^{-1} $ & $ 1.62 \cdot 10^{-1} $ & $ 1.62 \cdot 10^{-1} $ & $ 1.61 \cdot 10^{-1} $ & $ 1.61 \cdot 10^{-1} $ \\
\cline{2-11}
& $ 1.02 \cdot 10^{-1} $ & $ 1.02 \cdot 10^{-1} $ & $ 1.02 \cdot 10^{-1} $ & $ 1.02 \cdot 10^{-1} $ & $ 1.03 \cdot 10^{-1} $ & $ 1.02 \cdot 10^{-1} $ & $ 1.02 \cdot 10^{-1} $ & $ 1.01 \cdot 10^{-1} $ & $ 1.00 \cdot 10^{-1} $ & $ 9.96 \cdot 10^{-2} $ \\
\cline{2-11}
& $ 6.63 \cdot 10^{-2} $ & $ 6.65 \cdot 10^{-2} $ & $ 6.65 \cdot 10^{-2} $ & $ 6.76 \cdot 10^{-2} $ & $ 6.92 \cdot 10^{-2} $ & $ 6.79 \cdot 10^{-2} $ & $ 6.69 \cdot 10^{-2} $ & $ 6.53 \cdot 10^{-2} $ & $ 6.25 \cdot 10^{-2} $ & $ 6.12 \cdot 10^{-2} $ \\
\cline{2-11}
& $ 4.35 \cdot 10^{-2} $ & $ 4.19 \cdot 10^{-2} $ & $ 4.17 \cdot 10^{-2} $ & $ 4.47 \cdot 10^{-2} $ & $ 4.90 \cdot 10^{-2} $ & $ 4.59 \cdot 10^{-2} $ & $ 4.39 \cdot 10^{-2} $ & $ 3.98 \cdot 10^{-2} $ & $ 3.39 \cdot 10^{-2} $ & $ 3.15 \cdot 10^{-2} $ \\
\cline{2-11}
& $ 4.02 \cdot 10^{-2} $ & $ 2.81 \cdot 10^{-2} $ & $ 2.94 \cdot 10^{-2} $ & $ 3.51 \cdot 10^{-2} $ & $ 4.24 \cdot 10^{-2} $ & $ 3.75 \cdot 10^{-2} $ & $ 3.43 \cdot 10^{-2} $ & $ 2.66 \cdot 10^{-2} $ & $ 1.68 \cdot 10^{-2} $ & $ 1.35 \cdot 10^{-2} $ \\
\cline{2-11}
& $ 7.56 \cdot 10^{-2} $ & $ 2.03 \cdot 10^{-2} $ & $ 2.87 \cdot 10^{-2} $ & $ 3.69 \cdot 10^{-2} $ & $ 4.69 \cdot 10^{-2} $ & $ 4.11 \cdot 10^{-2} $ & $ 3.70 \cdot 10^{-2} $ & $ 2.50 \cdot 10^{-2} $ & $ 9.87 \cdot 10^{-3} $ & $ 6.10 \cdot 10^{-3} $ \\
\hline

\end{tabular}
\caption{$X \sim \mathcal{N}(0,1), \ \forall i \in \lbrace 1,2 \rbrace, \  \varepsilon^{(i)} \sim m \mathcal{N}(2,2,3,3), \ m \in \lbrace 6,7,8 \rbrace$}
\label{tab:GaussMixtureGauss2}
\end{table}
\end{landscape}

\begin{landscape}
\begin{table}
\begin{tabular}
{|c||c|c|c|c|c|c|c|c|c|c|}
\hline
 \multicolumn{11}{|c|}{$\nu_{\text{est}}$} \\
 \hline
 \multicolumn{1}{|c|}{} & $0.33$ & $0.5$ & $1$ & $1.5$ & $2$ & $2.5$ & $3$ & $3.5$ & $4$ & $4.5$\\
\hline
\hline

{}{}{$ 9 $} & $ 1.62 \cdot 10^{-1} $ & $ 1.62 \cdot 10^{-1} $ & $ 1.62 \cdot 10^{-1} $ & $ 1.62 \cdot 10^{-1} $ & $ 1.62 \cdot 10^{-1} $ & $ 1.62 \cdot 10^{-1} $ & $ 1.62 \cdot 10^{-1} $ & $ 1.62 \cdot 10^{-1} $ & $ 1.62 \cdot 10^{-1} $ & $ 1.62 \cdot 10^{-1} $ \\
\cline{2-11}
& $ 1.02 \cdot 10^{-1} $ & $ 1.02 \cdot 10^{-1} $ & $ 1.02 \cdot 10^{-1} $ & $ 1.02 \cdot 10^{-1} $ & $ 1.02 \cdot 10^{-1} $ & $ 1.02 \cdot 10^{-1} $ & $ 1.02 \cdot 10^{-1} $ & $ 1.02 \cdot 10^{-1} $ & $ 1.01 \cdot 10^{-1} $ & $ 1.00 \cdot 10^{-1} $ \\
\cline{2-11}
& $ 6.63 \cdot 10^{-2} $ & $ 6.67 \cdot 10^{-2} $ & $ 6.72 \cdot 10^{-2} $ & $ 6.66 \cdot 10^{-2} $ & $ 6.67 \cdot 10^{-2} $ & $ 6.81 \cdot 10^{-2} $ & $ 6.84 \cdot 10^{-2} $ & $ 6.58 \cdot 10^{-2} $ & $ 6.51 \cdot 10^{-2} $ & $ 6.32 \cdot 10^{-2} $ \\
\cline{2-11}
& $ 4.36 \cdot 10^{-2} $ & $ 4.21 \cdot 10^{-2} $ & $ 4.35 \cdot 10^{-2} $ & $ 4.21 \cdot 10^{-2} $ & $ 4.21 \cdot 10^{-2} $ & $ 4.62 \cdot 10^{-2} $ & $ 4.71 \cdot 10^{-2} $ & $ 4.08 \cdot 10^{-2} $ & $ 3.90 \cdot 10^{-2} $ & $ 3.51 \cdot 10^{-2} $ \\
\cline{2-11}
& $ 4.19 \cdot 10^{-2} $ & $ 3.20 \cdot 10^{-2} $ & $ 3.28 \cdot 10^{-2} $ & $ 3.00 \cdot 10^{-2} $ & $ 2.97 \cdot 10^{-2} $ & $ 3.73 \cdot 10^{-2} $ & $ 3.93 \cdot 10^{-2} $ & $ 2.84 \cdot 10^{-2} $ & $ 2.51 \cdot 10^{-2} $ & $ 1.77 \cdot 10^{-2} $ \\
\cline{2-11}
& $ 9.07 \cdot 10^{-2} $ & $ 5.41 \cdot 10^{-2} $ & $ 3.59 \cdot 10^{-2} $ & $ 2.92 \cdot 10^{-2} $ & $ 2.82 \cdot 10^{-2} $ & $ 3.95 \cdot 10^{-2} $ & $ 4.33 \cdot 10^{-2} $ & $ 2.81 \cdot 10^{-2} $ & $ 2.29 \cdot 10^{-2} $ & $ 9.42 \cdot 10^{-3} $ \\
\hline

\end{tabular}
\caption{$X \sim \mathcal{N}(0,1), \ \forall i \in \lbrace 1,2 \rbrace, \  \varepsilon^{(i)} \sim m \mathcal{N}(2,2,3,3), \ m \in \lbrace 9,10,11 \rbrace$}
\label{tab:GaussMixtureGauss2bis}
\end{table}
\end{landscape}

\begin{landscape}
\begin{table}
\begin{tabular}
{|c||c|c|c|c|c|c|c|c|c|c|}
\hline
 \multicolumn{11}{|c|}{$\nu_{\text{est}}$} \\
 \hline
 \multicolumn{1}{|c|}{} & $0.33$ & $0.5$ & $1$ & $1.5$ & $2$ & $2.5$ & $3$ & $3.5$ & $4$ & $4.5$\\
\hline
\hline

{}{}{$ 12 $} & $ 1.62 \cdot 10^{-1} $ & $ 1.62 \cdot 10^{-1} $ & $ 1.62 \cdot 10^{-1} $ & $ 1.62 \cdot 10^{-1} $ & $ 1.62 \cdot 10^{-1} $ & $ 1.62 \cdot 10^{-1} $ & $ 1.62 \cdot 10^{-1} $ & $ 1.55 \cdot 10^{-1} $ & $ 1.53 \cdot 10^{-1} $ & $ 1.62 \cdot 10^{-1} $ \\
\cline{2-11}
& $ 1.02 \cdot 10^{-1} $ & $ 1.02 \cdot 10^{-1} $ & $ 1.02 \cdot 10^{-1} $ & $ 1.02 \cdot 10^{-1} $ & $ 1.02 \cdot 10^{-1} $ & $ 1.00 \cdot 10^{-1} $ & $ 1.00 \cdot 10^{-1} $ & $ 9.20 \cdot 10^{-2} $ & $ 8.31 \cdot 10^{-2} $ & $ 1.01 \cdot 10^{-1} $ \\
\cline{2-11}
& $ 6.63 \cdot 10^{-2} $ & $ 6.65 \cdot 10^{-2} $ & $ 6.70 \cdot 10^{-2} $ & $ 6.67 \cdot 10^{-2} $ & $ 6.64 \cdot 10^{-2} $ & $ 6.27 \cdot 10^{-2} $ & $ 6.27 \cdot 10^{-2} $ & $ 7.47 \cdot 10^{-2} $ & $ 9.24 \cdot 10^{-2} $ & $ 6.70 \cdot 10^{-2} $ \\
\cline{2-11}
& $ 4.37 \cdot 10^{-2} $ & $ 4.18 \cdot 10^{-2} $ & $ 4.28 \cdot 10^{-2} $ & $ 4.25 \cdot 10^{-2} $ & $ 4.19 \cdot 10^{-2} $ & $ 3.33 \cdot 10^{-2} $ & $ 3.34 \cdot 10^{-2} $ & $ 8.72 \cdot 10^{-2} $ & $ 2.35 \cdot 10^{-1} $ & $ 1.10 \cdot 10^{-1} $ \\
\cline{2-11}
& $ 4.30 \cdot 10^{-2} $ & $ 3.64 \cdot 10^{-2} $ & $ 3.22 \cdot 10^{-2} $ & $ 3.06 \cdot 10^{-2} $ & $ 2.99 \cdot 10^{-2} $ & $ 1.46 \cdot 10^{-2} $ & $ 1.47 \cdot 10^{-2} $ & $ 3.35 \cdot 10^{-1} $ & $ 5.55 \cdot 10^{-1} $ & $ 4.85 \cdot 10^{-1} $ \\
\cline{2-11}
& $ 1.06 \cdot 10^{-1} $ & $ 2.53 \cdot 10^{-1} $ & $ 3.71 \cdot 10^{-2} $ & $ 2.95 \cdot 10^{-2} $ & $ 2.87 \cdot 10^{-2} $ & $ 5.67 \cdot 10^{-3} $ & $ 5.91 \cdot 10^{-3} $ & $ 1.01 $ & $ 9.56 \cdot 10^{-1} $ & $ 2.24 $ \\
\hline

\end{tabular}
\caption{$X \sim \mathcal{N}(0,1), \ \forall i \in \lbrace 1,2 \rbrace, \  \varepsilon^{(i)} \sim m \mathcal{N}(2,2,3,3), \ m \in \lbrace 12,13,14 \rbrace$}
\label{tab:GaussMixtureGauss3bis}
\end{table}
\end{landscape}
\newpage

\begin{landscape}
\begin{table}
\begin{tabular}
{|c||c|c|c|c|c|c|c|c|c|c|}
\hline
 \multicolumn{11}{|c|}{$\nu_{\text{est}}$} \\
 \hline
 \multicolumn{1}{|c|}{} & $0.33$ & $0.5$ & $1$ & $1.5$ & $2$ & $2.5$ & $3$ & $3.5$ & $4$ & $4.5$\\
\hline
\hline

{}{}{$ 15 $} & $ 1.62 \cdot 10^{-1} $ & $ 1.62 \cdot 10^{-1} $ & $ 1.62 \cdot 10^{-1} $ & $ 1.62 \cdot 10^{-1} $ & $ 1.62 \cdot 10^{-1} $ & $ 1.62 \cdot 10^{-1} $ & $ 1.62 \cdot 10^{-1} $ & $ 1.62 \cdot 10^{-1} $ & $ 1.62 \cdot 10^{-1} $ & $ 1.62 \cdot 10^{-1} $ \\
\cline{2-11}
& $ 1.02 \cdot 10^{-1} $ & $ 1.02 \cdot 10^{-1} $ & $ 1.02 \cdot 10^{-1} $ & $ 1.02 \cdot 10^{-1} $ & $ 1.00 \cdot 10^{-1} $ & $ 1.00 \cdot 10^{-1} $ & $ 1.00 \cdot 10^{-1} $ & $ 1.00 \cdot 10^{-1} $ & $ 1.00 \cdot 10^{-1} $ & $ 1.00 \cdot 10^{-1} $ \\
\cline{2-11}
& $ 6.63 \cdot 10^{-2} $ & $ 6.65 \cdot 10^{-2} $ & $ 6.66 \cdot 10^{-2} $ & $ 6.63 \cdot 10^{-2} $ & $ 6.27 \cdot 10^{-2} $ & $ 6.27 \cdot 10^{-2} $ & $ 6.27 \cdot 10^{-2} $ & $ 6.27 \cdot 10^{-2} $ & $ 6.27 \cdot 10^{-2} $ & $ 6.27 \cdot 10^{-2} $ \\
\cline{2-11}
& $ 4.37 \cdot 10^{-2} $ & $ 4.18 \cdot 10^{-2} $ & $ 4.23 \cdot 10^{-2} $ & $ 4.22 \cdot 10^{-2} $ & $ 3.33 \cdot 10^{-2} $ & $ 3.33 \cdot 10^{-2} $ & $ 3.33 \cdot 10^{-2} $ & $ 3.33 \cdot 10^{-2} $ & $ 3.33 \cdot 10^{-2} $ & $ 3.33 \cdot 10^{-2} $ \\
\cline{2-11}
& $ 4.33 \cdot 10^{-2} $ & $ 4.30 \cdot 10^{-2} $ & $ 3.09 \cdot 10^{-2} $ & $ 3.04 \cdot 10^{-2} $ & $ 1.46 \cdot 10^{-2} $ & $ 1.46 \cdot 10^{-2} $ & $ 1.46 \cdot 10^{-2} $ & $ 1.46 \cdot 10^{-2} $ & $ 1.46 \cdot 10^{-2} $ & $ 1.46 \cdot 10^{-2} $ \\
\cline{2-11}
& $ 1.13 \cdot 10^{-1} $ & $ 7.15 \cdot 10^{-1} $ & $ 3.48 \cdot 10^{-2} $ & $ 2.79 \cdot 10^{-2} $ & $ 5.66 \cdot 10^{-3} $ & $ 5.66 \cdot 10^{-3} $ & $ 5.66 \cdot 10^{-3} $ & $ 5.65 \cdot 10^{-3} $ & $ 5.66 \cdot 10^{-3} $ & $ 5.66 \cdot 10^{-3} $ \\
\hline

\end{tabular}
\caption{$X \sim \mathcal{N}(0,1), \ \forall i \in \lbrace 1,2 \rbrace, \  \varepsilon^{(i)} \sim m \mathcal{N}(2,2,3,3), \ m = 15$}
\label{tab:GaussMixtureGauss3}
\end{table}
\end{landscape}

\begin{landscape}
\begin{table}
\begin{tabular}
{|c||c|c|c|c|c|c|c|c|c|c|}
\hline
 \multicolumn{11}{|c|}{$\nu_{\text{est}}$} \\
 \hline
 \multicolumn{1}{|c|}{} & $0.33$ & $0.5$ & $1$ & $1.5$ & $2$ & $2.5$ & $3$ & $3.5$ & $4$ & $4.5$\\
\hline
\hline
{}{}{$ 3 $} & $ 1.76 \cdot 10^{-1} $ & $ 1.75 \cdot 10^{-1} $ & $ 1.75 \cdot 10^{-1} $ & $ 1.75 \cdot 10^{-1} $ & $ 1.75 \cdot 10^{-1} $ & $ 1.75 \cdot 10^{-1} $ & $ 1.75 \cdot 10^{-1} $ & $ 1.75 \cdot 10^{-1} $ & $ 1.75 \cdot 10^{-1} $ & $ 1.75 \cdot 10^{-1} $ \\
\cline{2-11}
& $ 1.23 \cdot 10^{-1} $ & $ 1.23 \cdot 10^{-1} $ & $ 1.21 \cdot 10^{-1} $ & $ 1.21 \cdot 10^{-1} $ & $ 1.21 \cdot 10^{-1} $ & $ 1.21 \cdot 10^{-1} $ & $ 1.21 \cdot 10^{-1} $ & $ 1.21 \cdot 10^{-1} $ & $ 1.21 \cdot 10^{-1} $ & $ 1.21 \cdot 10^{-1} $ \\
\cline{2-11}
& $ 7.59 \cdot 10^{-2} $ & $ 7.68 \cdot 10^{-2} $ & $ 7.28 \cdot 10^{-2} $ & $ 7.22 \cdot 10^{-2} $ & $ 7.21 \cdot 10^{-2} $ & $ 7.21 \cdot 10^{-2} $ & $ 7.20 \cdot 10^{-2} $ & $ 7.14 \cdot 10^{-2} $ & $ 7.13 \cdot 10^{-2} $ & $ 7.08 \cdot 10^{-2} $ \\
\cline{2-11}
& $ 4.72 \cdot 10^{-2} $ & $ 4.93 \cdot 10^{-2} $ & $ 4.24 \cdot 10^{-2} $ & $ 4.16 \cdot 10^{-2} $ & $ 4.15 \cdot 10^{-2} $ & $ 4.15 \cdot 10^{-2} $ & $ 4.14 \cdot 10^{-2} $ & $ 4.06 \cdot 10^{-2} $ & $ 4.05 \cdot 10^{-2} $ & $ 4.01 \cdot 10^{-2} $ \\
\cline{2-11}
& $ 3.28 \cdot 10^{-2} $ & $ 3.73 \cdot 10^{-2} $ & $ 2.67 \cdot 10^{-2} $ & $ 2.60 \cdot 10^{-2} $ & $ 2.60 \cdot 10^{-2} $ & $ 2.59 \cdot 10^{-2} $ & $ 2.57 \cdot 10^{-2} $ & $ 2.53 \cdot 10^{-2} $ & $ 2.53 \cdot 10^{-2} $ & $ 2.59 \cdot 10^{-2} $ \\
\cline{2-11}
& $ 3.26 \cdot 10^{-2} $ & $ 3.74 \cdot 10^{-2} $ & $ 2.\cdot 10 \cdot 10^{-2} $ & $ 2.07 \cdot 10^{-2} $ & $ 2.09 \cdot 10^{-2} $ & $ 2.07 \cdot 10^{-2} $ & $ 2.05 \cdot 10^{-2} $ & $ 2.\cdot 10 \cdot 10^{-2} $ & $ 2.15 \cdot 10^{-2} $ & $ 2.44 \cdot 10^{-2} $ \\
\hline
\end{tabular}
\caption{$X \sim b\Gamma(1,1,2,2), \forall i \in \lbrace 1,2 \rbrace, \  \varepsilon^{(i)} \sim \Gamma(4,2)-2 , \ m \in \lbrace 3,4,5 \rbrace$}
\label{tab:BiGammaGamma1}
\end{table}
\end{landscape}

\begin{landscape}
\begin{table}
\begin{tabular}
{|c||c|c|c|c|c|c|c|c|c|c|}
\hline
 \multicolumn{11}{|c|}{$\nu_{\text{est}}$} \\
 \hline
 \multicolumn{1}{|c|}{} & $0.33$ & $0.5$ & $1$ & $1.5$ & $2$ & $2.5$ & $3$ & $3.5$ & $4$ & $4.5$\\
\hline
\hline
{}{}{$ 6 $} & $ 1.76 \cdot 10^{-1} $ & $ 1.76 \cdot 10^{-1} $ & $ 1.76 \cdot 10^{-1} $ & $ 1.76 \cdot 10^{-1} $ & $ 1.76 \cdot 10^{-1} $ & $ 1.76 \cdot 10^{-1} $ & $ 1.75 \cdot 10^{-1} $ & $ 1.75 \cdot 10^{-1} $ & $ 1.75 \cdot 10^{-1} $ & $ 1.75 \cdot 10^{-1} $ \\
\cline{2-11}
& $ 1.23 \cdot 10^{-1} $ & $ 1.23 \cdot 10^{-1} $ & $ 1.25 \cdot 10^{-1} $ & $ 1.24 \cdot 10^{-1} $ & $ 1.24 \cdot 10^{-1} $ & $ 1.24 \cdot 10^{-1} $ & $ 1.23 \cdot 10^{-1} $ & $ 1.23 \cdot 10^{-1} $ & $ 1.22 \cdot 10^{-1} $ & $ 1.22 \cdot 10^{-1} $ \\
\cline{2-11}
& $ 7.67 \cdot 10^{-2} $ & $ 7.66 \cdot 10^{-2} $ & $ 8.55 \cdot 10^{-2} $ & $ 8.22 \cdot 10^{-2} $ & $ 8.23 \cdot 10^{-2} $ & $ 8.08 \cdot 10^{-2} $ & $ 7.85 \cdot 10^{-2} $ & $ 7.77 \cdot 10^{-2} $ & $ 7.48 \cdot 10^{-2} $ & $ 7.38 \cdot 10^{-2} $ \\
\cline{2-11}
& $ 4.91 \cdot 10^{-2} $ & $ 4.63 \cdot 10^{-2} $ & $ 6.36 \cdot 10^{-2} $ & $ 5.83 \cdot 10^{-2} $ & $ 5.88 \cdot 10^{-2} $ & $ 5.63 \cdot 10^{-2} $ & $ 5.20 \cdot 10^{-2} $ & $ 5.05 \cdot 10^{-2} $ & $ 4.53 \cdot 10^{-2} $ & $ 4.37 \cdot 10^{-2} $ \\
\cline{2-11}
& $ 5.07 \cdot 10^{-2} $ & $ 2.92 \cdot 10^{-2} $ & $ 5.55 \cdot 10^{-2} $ & $ 4.87 \cdot 10^{-2} $ & $ 5.03 \cdot 10^{-2} $ & $ 4.66 \cdot 10^{-2} $ & $ 4.00 \cdot 10^{-2} $ & $ 3.77 \cdot 10^{-2} $ & $ 3.01 \cdot 10^{-2} $ & $ 2.80 \cdot 10^{-2} $ \\
\cline{2-11}
& $ 2.23 \cdot 10^{-1} $ & $ 2.51 \cdot 10^{-2} $ & $ 5.52 \cdot 10^{-2} $ & $ 4.82 \cdot 10^{-2} $ & $ 5.16 \cdot 10^{-2} $ & $ 4.68 \cdot 10^{-2} $ & $ 3.77 \cdot 10^{-2} $ & $ 3.47 \cdot 10^{-2} $ & $ 2.44 \cdot 10^{-2} $ & $ 2.19 \cdot 10^{-2} $ \\
\hline
\end{tabular}
\caption{$X \sim b\Gamma(1,1,2,2), \forall i \in \lbrace 1,2 \rbrace, \  \varepsilon^{(i)} \sim \Gamma(4,2)-2 , \ m \in \lbrace 6,7,8 \rbrace$}
\label{tab:BiGammaGamma2}
\end{table}
\end{landscape}

\begin{landscape}
\begin{table}
\begin{tabular}
{|c||c|c|c|c|c|c|c|c|c|c|}
\hline
 \multicolumn{11}{|c|}{$\nu_{\text{est}}$} \\
 \hline
 \multicolumn{1}{|c|}{} & $0.33$ & $0.5$ & $1$ & $1.5$ & $2$ & $2.5$ & $3$ & $3.5$ & $4$ & $4.5$\\
\hline
\hline
{}{}{$ 9 $} & $ 1.76 \cdot 10^{-1} $ & $ 1.76 \cdot 10^{-1} $ & $ 1.76 \cdot 10^{-1} $ & $ 1.76 \cdot 10^{-1} $ & $ 1.76 \cdot 10^{-1} $ & $ 1.76 \cdot 10^{-1} $ & $ 1.76 \cdot 10^{-1} $ & $ 1.76 \cdot 10^{-1} $ & $ 2.78 \cdot 10^{-1} $ & $ 1.83 \cdot 10^{1} $ \\
\cline{2-11}
& $ 1.23 \cdot 10^{-1} $ & $ 1.23 \cdot 10^{-1} $ & $ 1.28 \cdot 10^{-1} $ & $ 1.25 \cdot 10^{-1} $ & $ 1.25 \cdot 10^{-1} $ & $ 1.25 \cdot 10^{-1} $ & $ 1.25 \cdot 10^{-1} $ & $ 1.24 \cdot 10^{-1} $ & $ 1.76 $ & $ 5.39 \cdot 10^{2} $ \\
\cline{2-11}
& $ 7.65 \cdot 10^{-2} $ & $ 7.68 \cdot 10^{-2} $ & $ 9.55 \cdot 10^{-2} $ & $ 8.57 \cdot 10^{-2} $ & $ 8.33 \cdot 10^{-2} $ & $ 8.44 \cdot 10^{-2} $ & $ 8.40 \cdot 10^{-2} $ & $ 8.23 \cdot 10^{-2} $ & $ 4.74 $ & $ 2.00 \cdot 10^{3} $ \\
\cline{2-11}
& $ 4.37 \cdot 10^{-2} $ & $ 4.66 \cdot 10^{-2} $ & $ 8.16 \cdot 10^{-2} $ & $ 6.40 \cdot 10^{-2} $ & $ 6.01 \cdot 10^{-2} $ & $ 6.25 \cdot 10^{-2} $ & $ 6.20 \cdot 10^{-2} $ & $ 5.90 \cdot 10^{-2} $ & $ 7.22 $ & $ 3.18 \cdot 10^{3} $ \\
\cline{2-11}
& $ 3.64 \cdot 10^{-2} $ & $ 2.75 \cdot 10^{-2} $ & $ 8.17 \cdot 10^{-2} $ & $ 5.62 \cdot 10^{-2} $ & $ 5.11 \cdot 10^{-2} $ & $ 5.55 \cdot 10^{-2} $ & $ 5.50 \cdot 10^{-2} $ & $ 5.05 \cdot 10^{-2} $ & $ 1.47 \cdot 10^{1} $ & $ 3.97 \cdot 10^{3} $ \\
\cline{2-11}
& $ 1.01 $ & $ 6.01 \cdot 10^{-2} $ & $ 8.88 \cdot 10^{-2} $ & $ 5.65 \cdot 10^{-2} $ & $ 5.08 \cdot 10^{-2} $ & $ 5.76 \cdot 10^{-2} $ & $ 5.76 \cdot 10^{-2} $ & $ 5.19 \cdot 10^{-2} $ & $ 3.71 \cdot 10^{1} $ & $ 6.32 \cdot 10^{3} $ \\
\hline
\end{tabular}
\caption{$X \sim b\Gamma(1,1,2,2), \forall i \in \lbrace 1,2 \rbrace, \  \varepsilon^{(i)} \sim \Gamma(4,2)-2 , \ m \in \lbrace 9,10,11 \rbrace$}
\label{tab:BiGammaGamma2bis}
\end{table}
\end{landscape}

\begin{landscape}
\begin{table}
\begin{tabular}
{|c||c|c|c|c|c|c|c|c|c|c|}
\hline
 \multicolumn{11}{|c|}{$\nu_{\text{est}}$} \\
 \hline
 \multicolumn{1}{|c|}{} & $0.33$ & $0.5$ & $1$ & $1.5$ & $2$ & $2.5$ & $3$ & $3.5$ & $4$ & $4.5$\\
\hline
\hline
{}{}{$ 12 $} & $ 1.76 \cdot 10^{-1} $ & $ 1.76 \cdot 10^{-1} $ & $ 1.76 \cdot 10^{-1} $ & $ 1.76 \cdot 10^{-1} $ & $ 1.68 \cdot 10^{-1} $ & $ 1.79 \cdot 10^{-1} $ & $ 2.16 \cdot 10^{1} $ & $ 2.10 \cdot 10^{2} $ & $ 3.37 \cdot 10^{2} $ & $ 2.84 $ \\
\cline{2-11}
& $ 1.23 \cdot 10^{-1} $ & $ 1.23 \cdot 10^{-1} $ & $ 1.28 \cdot 10^{-1} $ & $ 1.25 \cdot 10^{-1} $ & $ 1.26 \cdot 10^{-1} $ & $ 1.89 \cdot 10^{-1} $ & $ 3.21 \cdot 10^{2} $ & $ 3.60 \cdot 10^{3} $ & $ 3.18 \cdot 10^{4} $ & $ 8.49 \cdot 10^{1} $ \\
\cline{2-11}
& $ 7.66 \cdot 10^{-2} $ & $ 7.66 \cdot 10^{-2} $ & $ 9.57 \cdot 10^{-2} $ & $ 8.25 \cdot 10^{-2} $ & $ 2.94 \cdot 10^{-1} $ & $ 3.77 \cdot 10^{-1} $ & $ 3.30 \cdot 10^{2} $ & $ 3.35 \cdot 10^{3} $ & $ 3.13 \cdot 10^{5} $ & $ 3.56 \cdot 10^{2} $ \\
\cline{2-11}
& $ 4.36 \cdot 10^{-2} $ & $ 4.63 \cdot 10^{-2} $ & $ 8.17 \cdot 10^{-2} $ & $ 5.78 \cdot 10^{-2} $ & $ 1.11 $ & $ 7.30 \cdot 10^{-1} $ & $ 1.12 \cdot 10^{3} $ & $ 1.28 \cdot 10^{4} $ & $ 1.32 \cdot 10^{6} $ & $ 8.55 \cdot 10^{2} $ \\
\cline{2-11}
& $ 5.04 \cdot 10^{-2} $ & $ 2.16 \cdot 10^{-2} $ & $ 8.16 \cdot 10^{-2} $ & $ 4.54 \cdot 10^{-2} $ & $ 3.35 $ & $ 1.21 $ & $ 5.57 \cdot 10^{3} $ & $ 1.31 \cdot 10^{5} $ & $ 4.24 \cdot 10^{6} $ & $ 4.31 \cdot 10^{3} $ \\
\cline{2-11}
& $ 7.67 $ & $ 8.76 \cdot 10^{-1} $ & $ 8.87 \cdot 10^{-2} $ & $ 4.08 \cdot 10^{-2} $ & $ 8.92 $ & $ 1.78 $ & $ 8.96 \cdot 10^{3} $ & $ 4.40 \cdot 10^{5} $ & $ 9.22 \cdot 10^{6} $ & $ 3.97 \cdot 10^{4} $ \\
\hline
\end{tabular}
\caption{$X \sim b\Gamma(1,1,2,2), \forall i \in \lbrace 1,2 \rbrace, \  \varepsilon^{(i)} \sim \Gamma(4,2)-2 , \ m \in \lbrace 12,13,14 \rbrace$}
\label{tab:BiGammaGamma2bisbis}
\end{table}
\end{landscape}

\begin{landscape}
\begin{table}
\begin{tabular}
{|c||c|c|c|c|c|c|c|c|c|c|}
\hline
 \multicolumn{11}{|c|}{$\nu_{\text{est}}$} \\
 \hline
 \multicolumn{1}{|c|}{} & $0.33$ & $0.5$ & $1$ & $1.5$ & $2$ & $2.5$ & $3$ & $3.5$ & $4$ & $4.5$\\
\hline
\hline
{}{}{$ 15 $} & $ 1.76 \cdot 10^{-1} $ & $ 1.76 \cdot 10^{-1} $ & $ 1.76 \cdot 10^{-1} $ & $ 1.86 \cdot 10^{-1} $ & $ 2.78 \cdot 10^{-1} $ & $ 2.03 \cdot 10^{3} $ & $ 8.54 \cdot 10^{3} $ & $ 3.37 \cdot 10^{5} $ & $ 2.48 \cdot 10^{6} $ & $ 1.27 \cdot 10^{5} $ \\
\cline{2-11}
& $ 1.23 \cdot 10^{-1} $ & $ 1.23 \cdot 10^{-1} $ & $ 1.23 \cdot 10^{-1} $ & $ 5.07 \cdot 10^{-1} $ & $ 3.42 \cdot 10^{-1} $ & $ 7.80 \cdot 10^{4} $ & $ 2.77 \cdot 10^{5} $ & $ 1.29 \cdot 10^{7} $ & $ 6.11 \cdot 10^{7} $ & $ 4.12 \cdot 10^{6} $ \\
\cline{2-11}
& $ 7.66 \cdot 10^{-2} $ & $ 7.64 \cdot 10^{-2} $ & $ 7.68 \cdot 10^{-2} $ & $ 2.01 $ & $ 4.01 \cdot 10^{-1} $ & $ 5.36 \cdot 10^{5} $ & $ 1.25 \cdot 10^{6} $ & $ 7.25 \cdot 10^{7} $ & $ 1.44 \cdot 10^{8} $ & $ 1.92 \cdot 10^{7} $ \\
\cline{2-11}
& $ 4.76 \cdot 10^{-2} $ & $ 4.60 \cdot 10^{-2} $ & $ 4.66 \cdot 10^{-2} $ & $ 7.71 $ & $ 6.68 $ & $ 1.65 \cdot 10^{6} $ & $ 3.01 \cdot 10^{6} $ & $ 2.18 \cdot 10^{8} $ & $ 1.18 \cdot 10^{8} $ & $ 5.85 \cdot 10^{7} $ \\
\cline{2-11}
& $ 2.69 \cdot 10^{-1} $ & $ 4.37 \cdot 10^{-2} $ & $ 2.86 \cdot 10^{-2} $ & $ 1.95 \cdot 10^{1} $ & $ 8.75 \cdot 10^{1} $ & $ 2.64 \cdot 10^{6} $ & $ 6.33 \cdot 10^{6} $ & $ 5.38 \cdot 10^{8} $ & $ 2.30 \cdot 10^{8} $ & $ 1.91 \cdot 10^{8} $ \\
\cline{2-11}
& $ 4.86 \cdot 10^{1} $ & $ 8.44 $ & $ 1.85 \cdot 10^{-2} $ & $ 7.06 \cdot 10^{1} $ & $ 5.66 \cdot 10^{2} $ & $ 2.77 \cdot 10^{6} $ & $ 1.07 \cdot 10^{7} $ & $ 8.37 \cdot 10^{8} $ & $ 2.58 \cdot 10^{8} $ & $ 8.01 \cdot 10^{8} $ \\
\hline
\end{tabular}
\caption{$X \sim b\Gamma(1,1,2,2), \forall i \in \lbrace 1,2 \rbrace, \  \varepsilon^{(i)} \sim \Gamma(4,2)-2 , \ m = 15$}
\label{tab:BiGammaGamma3}
\end{table}
\end{landscape}
\newpage

\begin{landscape}
\begin{table}
\begin{tabular}
{|c||c|c|c|c|c|c|c|c|c|c|}
\hline
 \multicolumn{11}{|c|}{$\nu_{\text{est}}$} \\
 \hline
 \multicolumn{1}{|c|}{} & $0.33$ & $0.5$ & $1$ & $1.5$ & $2$ & $2.5$ & $3$ & $3.5$ & $4$ & $4.5$\\
\hline
\hline
{}{}{$ 3 $} & $ 2.24 \cdot 10^{-1} $ & $ 2.25 \cdot 10^{-1} $ & $ 2.25 \cdot 10^{-1} $ & $ 2.25 \cdot 10^{-1} $ & $ 2.25 \cdot 10^{-1} $ & $ 2.25 \cdot 10^{-1} $ & $ 2.24 \cdot 10^{-1} $ & $ 2.24 \cdot 10^{-1} $ & $ 2.24 \cdot 10^{-1} $ & $ 2.24 \cdot 10^{-1} $ \\
\cline{2-11}
& $ 1.99 \cdot 10^{-1} $ & $ 2.07 \cdot 10^{-1} $ & $ 2.04 \cdot 10^{-1} $ & $ 2.04 \cdot 10^{-1} $ & $ 2.03 \cdot 10^{-1} $ & $ 2.03 \cdot 10^{-1} $ & $ 2.02 \cdot 10^{-1} $ & $ 2.00 \cdot 10^{-1} $ & $ 1.99 \cdot 10^{-1} $ & $ 1.98 \cdot 10^{-1} $ \\
\cline{2-11}
& $ 2.21 \cdot 10^{-1} $ & $ 2.43 \cdot 10^{-1} $ & $ 2.38 \cdot 10^{-1} $ & $ 2.37 \cdot 10^{-1} $ & $ 2.36 \cdot 10^{-1} $ & $ 2.36 \cdot 10^{-1} $ & $ 2.32 \cdot 10^{-1} $ & $ 2.28 \cdot 10^{-1} $ & $ 2.26 \cdot 10^{-1} $ & $ 2.24 \cdot 10^{-1} $ \\
\cline{2-11}
& $ 2.75 \cdot 10^{-1} $ & $ 3.11 \cdot 10^{-1} $ & $ 3.08 \cdot 10^{-1} $ & $ 3.08 \cdot 10^{-1} $ & $ 3.08 \cdot 10^{-1} $ & $ 3.07 \cdot 10^{-1} $ & $ 3.02 \cdot 10^{-1} $ & $ 2.97 \cdot 10^{-1} $ & $ 2.94 \cdot 10^{-1} $ & $ 2.93 \cdot 10^{-1} $ \\
\cline{2-11}
& $ 3.33 \cdot 10^{-1} $ & $ 3.79 \cdot 10^{-1} $ & $ 3.84 \cdot 10^{-1} $ & $ 3.87 \cdot 10^{-1} $ & $ 3.87 \cdot 10^{-1} $ & $ 3.87 \cdot 10^{-1} $ & $ 3.81 \cdot 10^{-1} $ & $ 3.78 \cdot 10^{-1} $ & $ 3.77 \cdot 10^{-1} $ & $ 3.78 \cdot 10^{-1} $ \\
\cline{2-11}
& $ 3.72 \cdot 10^{-1} $ & $ 4.28 \cdot 10^{-1} $ & $ 4.47 \cdot 10^{-1} $ & $ 4.56 \cdot 10^{-1} $ & $ 4.56 \cdot 10^{-1} $ & $ 4.56 \cdot 10^{-1} $ & $ 4.54 \cdot 10^{-1} $ & $ 4.54 \cdot 10^{-1} $ & $ 4.56 \cdot 10^{-1} $ & $ 4.64 \cdot 10^{-1} $ \\
\hline
\end{tabular}
\caption{$X \sim \Gamma(4,2), \ \forall i \in \lbrace 1,2 \rbrace, \  \varepsilon^{(i)} \sim b\Gamma(2,2,3,3), \ m \in \lbrace 3,4,5 \rbrace$}
\label{tab:GammaBiGamma1}
\end{table}
\end{landscape}
\newpage

\begin{landscape}
\begin{table}
\begin{tabular}
{|c||c|c|c|c|c|c|c|c|c|c|}
\hline
 \multicolumn{11}{|c|}{$\nu_{\text{est}}$} \\
 \hline
 \multicolumn{1}{|c|}{} & $0.33$ & $0.5$ & $1$ & $1.5$ & $2$ & $2.5$ & $3$ & $3.5$ & $4$ & $4.5$\\
\hline
\hline
{}{}{$ 6 $} & $ 2.22 \cdot 10^{-1} $ & $ 2.22 \cdot 10^{-1} $ & $ 2.26 \cdot 10^{-1} $ & $ 2.26 \cdot 10^{-1} $ & $ 2.26 \cdot 10^{-1} $ & $ 2.25 \cdot 10^{-1} $ & $ 2.25 \cdot 10^{-1} $ & $ 2.25 \cdot 10^{-1} $ & $ 2.25 \cdot 10^{-1} $ & $ 2.25 \cdot 10^{-1} $ \\
\cline{2-11}
& $ 1.54 \cdot 10^{-1} $ & $ 1.53 \cdot 10^{-1} $ & $ 2.23 \cdot 10^{-1} $ & $ 2.20 \cdot 10^{-1} $ & $ 2.15 \cdot 10^{-1} $ & $ 2.11 \cdot 10^{-1} $ & $ 2.09 \cdot 10^{-1} $ & $ 2.07 \cdot 10^{-1} $ & $ 2.07 \cdot 10^{-1} $ & $ 2.05 \cdot 10^{-1} $ \\
\cline{2-11}
& $ 9.77 \cdot 10^{-2} $ & $ 9.81 \cdot 10^{-2} $ & $ 2.82 \cdot 10^{-1} $ & $ 2.76 \cdot 10^{-1} $ & $ 2.64 \cdot 10^{-1} $ & $ 2.54 \cdot 10^{-1} $ & $ 2.47 \cdot 10^{-1} $ & $ 2.42 \cdot 10^{-1} $ & $ 2.44 \cdot 10^{-1} $ & $ 2.39 \cdot 10^{-1} $ \\
\cline{2-11}
& $ 5.79 \cdot 10^{-2} $ & $ 6.16 \cdot 10^{-2} $ & $ 3.65 \cdot 10^{-1} $ & $ 3.58 \cdot 10^{-1} $ & $ 3.42 \cdot 10^{-1} $ & $ 3.28 \cdot 10^{-1} $ & $ 3.18 \cdot 10^{-1} $ & $ 3.11 \cdot 10^{-1} $ & $ 3.13 \cdot 10^{-1} $ & $ 3.08 \cdot 10^{-1} $ \\
\cline{2-11}
& $ 4.72 \cdot 10^{-2} $ & $ 3.43 \cdot 10^{-2} $ & $ 4.39 \cdot 10^{-1} $ & $ 4.33 \cdot 10^{-1} $ & $ 4.15 \cdot 10^{-1} $ & $ 4.00 \cdot 10^{-1} $ & $ 3.89 \cdot 10^{-1} $ & $ 3.81 \cdot 10^{-1} $ & $ 3.83 \cdot 10^{-1} $ & $ 3.81 \cdot 10^{-1} $ \\
\cline{2-11}
& $ 5.07 \cdot 10^{-1} $ & $ 5.00 \cdot 10^{-2} $ & $ 4.90 \cdot 10^{-1} $ & $ 4.86 \cdot 10^{-1} $ & $ 4.71 \cdot 10^{-1} $ & $ 4.56 \cdot 10^{-1} $ & $ 4.45 \cdot 10^{-1} $ & $ 4.37 \cdot 10^{-1} $ & $ 4.39 \cdot 10^{-1} $ & $ 4.42 \cdot 10^{-1} $ \\
\hline
\end{tabular}
\caption{$X \sim \Gamma(4,2), \forall i \in \lbrace 1,2 \rbrace, \  \varepsilon^{(i)} \sim b\Gamma(2,2,3,3), \ m \in \lbrace 6,7,8 \rbrace$}
\label{tab:GammaBiGamma2}
\end{table}
\end{landscape}
\newpage

\begin{landscape}
\begin{table}
\begin{tabular}
{|c||c|c|c|c|c|c|c|c|c|c|}
\hline
 \multicolumn{11}{|c|}{$\nu_{\text{est}}$} \\
 \hline
 \multicolumn{1}{|c|}{} & $0.33$ & $0.5$ & $1$ & $1.5$ & $2$ & $2.5$ & $3$ & $3.5$ & $4$ & $4.5$\\
\hline
\hline
{}{}{$ 9 $} & $ 2.22 \cdot 10^{-1} $ & $ 2.22 \cdot 10^{-1} $ & $ 2.27 \cdot 10^{-1} $ & $ 2.27 \cdot 10^{-1} $ & $ 2.27 \cdot 10^{-1} $ & $ 2.26 \cdot 10^{-1} $ & $ 2.26 \cdot 10^{-1} $ & $ 2.26 \cdot 10^{-1} $ & $ 2.26 \cdot 10^{-1} $ & $ 9.83 \cdot 10^{-1} $ \\
\cline{2-11}
& $ 1.54 \cdot 10^{-1} $ & $ 1.54 \cdot 10^{-1} $ & $ 2.25 \cdot 10^{-1} $ & $ 2.27 \cdot 10^{-1} $ & $ 2.26 \cdot 10^{-1} $ & $ 2.22 \cdot 10^{-1} $ & $ 2.16 \cdot 10^{-1} $ & $ 2.15 \cdot 10^{-1} $ & $ 2.16 \cdot 10^{-1} $ & $ 1.33 \cdot 10^{1} $ \\
\cline{2-11}
& $ 9.77 \cdot 10^{-2} $ & $ 9.80 \cdot 10^{-2} $ & $ 2.87 \cdot 10^{-1} $ & $ 2.92 \cdot 10^{-1} $ & $ 2.88 \cdot 10^{-1} $ & $ 2.79 \cdot 10^{-1} $ & $ 2.66 \cdot 10^{-1} $ & $ 2.62 \cdot 10^{-1} $ & $ 2.65 \cdot 10^{-1} $ & $ 5.30 \cdot 10^{1} $ \\
\cline{2-11}
& $ 5.92 \cdot 10^{-2} $ & $ 6.08 \cdot 10^{-2} $ & $ 3.72 \cdot 10^{-1} $ & $ 3.79 \cdot 10^{-1} $ & $ 3.75 \cdot 10^{-1} $ & $ 3.63 \cdot 10^{-1} $ & $ 3.44 \cdot 10^{-1} $ & $ 3.38 \cdot 10^{-1} $ & $ 3.44 \cdot 10^{-1} $ & $ 1.73 \cdot 10^{2} $ \\
\cline{2-11}
& $ 5.29 \cdot 10^{-2} $ & $ 3.71 \cdot 10^{-2} $ & $ 4.46 \cdot 10^{-1} $ & $ 4.53 \cdot 10^{-1} $ & $ 4.50 \cdot 10^{-1} $ & $ 4.38 \cdot 10^{-1} $ & $ 4.18 \cdot 10^{-1} $ & $ 4.11 \cdot 10^{-1} $ & $ 4.20 \cdot 10^{-1} $ & $ 4.16 \cdot 10^{2} $ \\
\cline{2-11}
& $ 1.03 $ & $ 1.45 \cdot 10^{-1} $ & $ 4.96 \cdot 10^{-1} $ & $ 5.03 \cdot 10^{-1} $ & $ 5.01 \cdot 10^{-1} $ & $ 4.91 \cdot 10^{-1} $ & $ 4.72 \cdot 10^{-1} $ & $ 4.66 \cdot 10^{-1} $ & $ 4.76 \cdot 10^{-1} $ & $ 7.74 \cdot 10^{2} $ \\
\hline
\end{tabular}
\caption{$X \sim \Gamma(4,2), \forall i \in \lbrace 1,2 \rbrace, \  \varepsilon^{(i)} \sim b\Gamma(2,2,3,3), \ m \in \lbrace 9,10,11 \rbrace$}
\label{tab:GammaBiGamma2bis}
\end{table}
\end{landscape}

\begin{landscape}
\begin{table}
\begin{tabular}
{|c||c|c|c|c|c|c|c|c|c|c|}
\hline
 \multicolumn{11}{|c|}{$\nu_{\text{est}}$} \\
 \hline
 \multicolumn{1}{|c|}{} & $0.33$ & $0.5$ & $1$ & $1.5$ & $2$ & $2.5$ & $3$ & $3.5$ & $4$ & $4.5$\\
\hline
\hline
{}{}{$ 12 $} & $ 2.22 \cdot 10^{-1} $ & $ 2.22 \cdot 10^{-1} $ & $ 2.22 \cdot 10^{-1} $ & $ 2.22 \cdot 10^{-1} $ & $ 2.22 \cdot 10^{-1} $ & $ 2.19 \cdot 10^{-1} $ & $ 2.25 \cdot 10^{-1} $ & $ 3.89 \cdot 10^{3} $ & $ 4.88 \cdot 10^{5} $ & $ 2.17 \cdot 10^{-1} $ \\
\cline{2-11}
& $ 1.54 \cdot 10^{-1} $ & $ 1.54 \cdot 10^{-1} $ & $ 1.64 \cdot 10^{-1} $ & $ 1.53 \cdot 10^{-1} $ & $ 1.64 \cdot 10^{-1} $ & $ 1.76 \cdot 10^{-1} $ & $ 2.80 \cdot 10^{-1} $ & $ 4.75 \cdot 10^{4} $ & $ 5.08 \cdot 10^{6} $ & $ 1.52 \cdot 10^{-1} $ \\
\cline{2-11}
& $ 9.77 \cdot 10^{-2} $ & $ 9.77 \cdot 10^{-2} $ & $ 1.28 \cdot 10^{-1} $ & $ 9.89 \cdot 10^{-2} $ & $ 1.30 \cdot 10^{-1} $ & $ 2.78 \cdot 10^{-1} $ & $ 6.00 \cdot 10^{-1} $ & $ 9.46 \cdot 10^{4} $ & $ 1.52 \cdot 10^{7} $ & $ 7.73 \cdot 10^{-1} $ \\
\cline{2-11}
& $ 6.10 \cdot 10^{-2} $ & $ 6.04 \cdot 10^{-2} $ & $ 1.17 \cdot 10^{-1} $ & $ 6.60 \cdot 10^{-2} $ & $ 1.25 \cdot 10^{-1} $ & $ 7.40 \cdot 10^{-1} $ & $ 1.18 $ & $ 1.09 \cdot 10^{5} $ & $ 3.12 \cdot 10^{7} $ & $ 1.46 \cdot 10^{1} $ \\
\cline{2-11}
& $ 5.17 \cdot 10^{-2} $ & $ 3.90 \cdot 10^{-2} $ & $ 1.21 \cdot 10^{-1} $ & $ 4.99 \cdot 10^{-2} $ & $ 1.34 \cdot 10^{-1} $ & $ 1.67 $ & $ 1.88 $ & $ 2.92 \cdot 10^{5} $ & $ 4.85 \cdot 10^{7} $ & $ 1.87 \cdot 10^{2} $ \\
\cline{2-11}
& $ 6.37 \cdot 10^{-1} $ & $ 1.14 \cdot 10^{-1} $ & $ 1.33 \cdot 10^{-1} $ & $ 4.11 \cdot 10^{-2} $ & $ 1.50 \cdot 10^{-1} $ & $ 3.24 $ & $ 2.63 $ & $ 1.35 \cdot 10^{6} $ & $ 6.76 \cdot 10^{7} $ & $ 1.99 \cdot 10^{3} $ \\
\hline
\end{tabular}
\caption{$X \sim \Gamma(4,2), \forall i \in \lbrace 1,2 \rbrace, \  \varepsilon^{(i)} \sim b\Gamma(2,2,3,3), \ m \in \lbrace 12,13,14 \rbrace$}
\label{tab:GammaBiGamma2bisbis}
\end{table}
\end{landscape}

\begin{landscape}
\begin{table}
\begin{tabular}
{|c||c|c|c|c|c|c|c|c|c|c|}
\hline
 \multicolumn{11}{|c|}{$\nu_{\text{est}}$} \\
 \hline
 \multicolumn{1}{|c|}{} & $0.33$ & $0.5$ & $1$ & $1.5$ & $2$ & $2.5$ & $3$ & $3.5$ & $4$ & $4.5$\\
\hline
\hline
{}{}{$ 15 $} & $ 2.22 \cdot 10^{-1} $ & $ 2.22 \cdot 10^{-1} $ & $ 2.22 \cdot 10^{-1} $ & $ 2.22 \cdot 10^{-1} $ & $ 2.38 \cdot 10^{-1} $ & $ 1.17 \cdot 10^{2} $ & $ 4.03 \cdot 10^{5} $ & $ 1.70 \cdot 10^{4} $ & $ 1.34 \cdot 10^{6} $ & $ 2.22 \cdot 10^{-1} $ \\
\cline{2-11}
& $ 1.54 \cdot 10^{-1} $ & $ 1.54 \cdot 10^{-1} $ & $ 1.53 \cdot 10^{-1} $ & $ 1.50 \cdot 10^{-1} $ & $ 6.17 \cdot 10^{-1} $ & $ 9.84 \cdot 10^{2} $ & $ 3.19 \cdot 10^{6} $ & $ 2.71 \cdot 10^{5} $ & $ 7.96 \cdot 10^{7} $ & $ 1.53 \cdot 10^{-1} $ \\
\cline{2-11}
& $ 9.77 \cdot 10^{-2} $ & $ 9.77 \cdot 10^{-2} $ & $ 9.73 \cdot 10^{-2} $ & $ 9.05 \cdot 10^{-2} $ & $ 3.28 $ & $ 1.27 \cdot 10^{3} $ & $ 5.65 \cdot 10^{6} $ & $ 1.77 \cdot 10^{6} $ & $ 6.19 \cdot 10^{8} $ & $ 9.75 \cdot 10^{-2} $ \\
\cline{2-11}
& $ 6.05 \cdot 10^{-2} $ & $ 6.04 \cdot 10^{-2} $ & $ 6.00 \cdot 10^{-2} $ & $ 5.56 \cdot 10^{-2} $ & $ 1.73 \cdot 10^{1} $ & $ 4.24 \cdot 10^{3} $ & $ 1.34 \cdot 10^{7} $ & $ 1.78 \cdot 10^{7} $ & $ 1.90 \cdot 10^{9} $ & $ 6.18 \cdot 10^{-2} $ \\
\cline{2-11}
& $ 3.87 \cdot 10^{-2} $ & $ 3.74 \cdot 10^{-2} $ & $ 3.70 \cdot 10^{-2} $ & $ 4.68 \cdot 10^{-2} $ & $ 6.97 \cdot 10^{1} $ & $ 6.43 \cdot 10^{3} $ & $ 3.78 \cdot 10^{7} $ & $ 1.86 \cdot 10^{8} $ & $ 6.20 \cdot 10^{9} $ & $ 6.73 \cdot 10^{-2} $ \\
\cline{2-11}
& $ 1.97 \cdot 10^{-1} $ & $ 6.21 \cdot 10^{-2} $ & $ 2.25 \cdot 10^{-2} $ & $ 6.02 \cdot 10^{-2} $ & $ 2.83 \cdot 10^{2} $ & $ 8.50 \cdot 10^{3} $ & $ 8.43 \cdot 10^{7} $ & $ 1.19 \cdot 10^{9} $ & $ 3.24 \cdot 10^{10} $ & $ 9.46 \cdot 10^{-1} $ \\
\hline
\end{tabular}
\caption{$X \sim \Gamma(4,2), \forall i \in \lbrace 1,2 \rbrace, \  \varepsilon^{(i)} \sim b\Gamma(2,2,3,3), \ m = 15$}
\label{tab:GammaBiGamma3}
\end{table}
\end{landscape}
\newpage

\begin{landscape}
\begin{table}
\begin{tabular}
{|c||c|c|c|c|c|c|c|c|c|c|}
\hline
 \multicolumn{11}{|c|}{$\nu_{\text{est}}$} \\
 \hline
 \multicolumn{1}{|c|}{} & $0.33$ & $0.5$ & $1$ & $1.5$ & $2$ & $2.5$ & $3$ & $3.5$ & $4$ & $4.5$\\
\hline
\hline
{}{}{$ 3 $} & $ 1.61 \cdot 10^{-1} $ & $ 1.61 \cdot 10^{-1} $ & $ 1.61 \cdot 10^{-1} $ & $ 1.61 \cdot 10^{-1} $ & $ 1.61 \cdot 10^{-1} $ & $ 1.61 \cdot 10^{-1} $ & $ 1.61 \cdot 10^{-1} $ & $ 1.61 \cdot 10^{-1} $ & $ 1.61 \cdot 10^{-1} $ & $ 1.61 \cdot 10^{-1} $ \\
\cline{2-11}
& $ 9.99 \cdot 10^{-2} $ & $ 9.95 \cdot 10^{-2} $ & $ 9.88 \cdot 10^{-2} $ & $ 9.86 \cdot 10^{-2} $ & $ 9.85 \cdot 10^{-2} $ & $ 9.85 \cdot 10^{-2} $ & $ 9.85 \cdot 10^{-2} $ & $ 9.83 \cdot 10^{-2} $ & $ 9.81 \cdot 10^{-2} $ & $ 9.80 \cdot 10^{-2} $ \\
\cline{2-11}
& $ 6.19 \cdot 10^{-2} $ & $ 6.12 \cdot 10^{-2} $ & $ 5.97 \cdot 10^{-2} $ & $ 5.92 \cdot 10^{-2} $ & $ 5.92 \cdot 10^{-2} $ & $ 5.92 \cdot 10^{-2} $ & $ 5.92 \cdot 10^{-2} $ & $ 5.87 \cdot 10^{-2} $ & $ 5.83 \cdot 10^{-2} $ & $ 5.81 \cdot 10^{-2} $ \\
\cline{2-11}
& $ 3.28 \cdot 10^{-2} $ & $ 3.15 \cdot 10^{-2} $ & $ 2.87 \cdot 10^{-2} $ & $ 2.80 \cdot 10^{-2} $ & $ 2.79 \cdot 10^{-2} $ & $ 2.79 \cdot 10^{-2} $ & $ 2.79 \cdot 10^{-2} $ & $ 2.72 \cdot 10^{-2} $ & $ 2.66 \cdot 10^{-2} $ & $ 2.63 \cdot 10^{-2} $ \\
\cline{2-11}
& $ 1.56 \cdot 10^{-2} $ & $ 1.37 \cdot 10^{-2} $ & $ 1.07 \cdot 10^{-2} $ & $ 1.01 \cdot 10^{-2} $ & $ 1.00 \cdot 10^{-2} $ & $ 1.00 \cdot 10^{-2} $ & $ 1.00 \cdot 10^{-2} $ & $ 9.73 \cdot 10^{-3} $ & $ 9.61 \cdot 10^{-3} $ & $ 9.62 \cdot 10^{-3} $ \\
\cline{2-11}
& $ 9.15 \cdot 10^{-3} $ & $ 6.53 \cdot 10^{-3} $ & $ 4.45 \cdot 10^{-3} $ & $ 4.80 \cdot 10^{-3} $ & $ 4.85 \cdot 10^{-3} $ & $ 4.85 \cdot 10^{-3} $ & $ 4.87 \cdot 10^{-3} $ & $ 5.79 \cdot 10^{-3} $ & $ 7.29 \cdot 10^{-3} $ & $ 8.20 \cdot 10^{-3} $ \\
\hline

\end{tabular}
\caption{$X \sim \mathcal{N}(0,1), \forall i \in \lbrace 1,2 \rbrace, \  \varepsilon^{(i)} \sim b\Gamma(2,2,3,3), \ m \in \lbrace 3,4,5 \rbrace$}
\label{tab:GaussianBiGamma}
\end{table}
\end{landscape}
\newpage

\begin{landscape}
\begin{table}
\begin{tabular}
{|c||c|c|c|c|c|c|c|c|c|c|}
\hline
 \multicolumn{11}{|c|}{$\nu_{\text{est}}$} \\
 \hline
 \multicolumn{1}{|c|}{} & $0.33$ & $0.5$ & $1$ & $1.5$ & $2$ & $2.5$ & $3$ & $3.5$ & $4$ & $4.5$\\
\hline
\hline
{}{}{$ 6 $} & $ 1.62 \cdot 10^{-1} $ & $ 1.61 \cdot 10^{-1} $ & $ 1.61 \cdot 10^{-1} $ & $ 1.61 \cdot 10^{-1} $ & $ 1.61 \cdot 10^{-1} $ & $ 1.61 \cdot 10^{-1} $ & $ 1.61 \cdot 10^{-1} $ & $ 1.61 \cdot 10^{-1} $ & $ 1.61 \cdot 10^{-1} $ & $ 1.61 \cdot 10^{-1} $ \\
\cline{2-11}
& $ 1.00 \cdot 10^{-1} $ & $ 1.00 \cdot 10^{-1} $ & $ 1.00 \cdot 10^{-1} $ & $ 9.99 \cdot 10^{-2} $ & $ 9.98 \cdot 10^{-2} $ & $ 9.98 \cdot 10^{-2} $ & $ 9.95 \cdot 10^{-2} $ & $ 9.93 \cdot 10^{-2} $ & $ 9.92 \cdot 10^{-2} $ & $ 9.91 \cdot 10^{-2} $ \\
\cline{2-11}
& $ 6.21 \cdot 10^{-2} $ & $ 6.21 \cdot 10^{-2} $ & $ 6.20 \cdot 10^{-2} $ & $ 6.18 \cdot 10^{-2} $ & $ 6.16 \cdot 10^{-2} $ & $ 6.16 \cdot 10^{-2} $ & $ 6.11 \cdot 10^{-2} $ & $ 6.06 \cdot 10^{-2} $ & $ 6.03 \cdot 10^{-2} $ & $ 6.02 \cdot 10^{-2} $ \\
\cline{2-11}
& $ 3.19 \cdot 10^{-2} $ & $ 3.23 \cdot 10^{-2} $ & $ 3.22 \cdot 10^{-2} $ & $ 3.19 \cdot 10^{-2} $ & $ 3.17 \cdot 10^{-2} $ & $ 3.17 \cdot 10^{-2} $ & $ 3.10 \cdot 10^{-2} $ & $ 3.02 \cdot 10^{-2} $ & $ 2.98 \cdot 10^{-2} $ & $ 2.95 \cdot 10^{-2} $ \\
\cline{2-11}
& $ 1.24 \cdot 10^{-2} $ & $ 1.37 \cdot 10^{-2} $ & $ 1.33 \cdot 10^{-2} $ & $ 1.31 \cdot 10^{-2} $ & $ 1.31 \cdot 10^{-2} $ & $ 1.31 \cdot 10^{-2} $ & $ 1.25 \cdot 10^{-2} $ & $ 1.18 \cdot 10^{-2} $ & $ 1.14 \cdot 10^{-2} $ & $ 1.11 \cdot 10^{-2} $ \\
\cline{2-11}
& $ 3.49 \cdot 10^{-3} $ & $ 5.87 \cdot 10^{-3} $ & $ 4.59 \cdot 10^{-3} $ & $ 4.61 \cdot 10^{-3} $ & $ 4.83 \cdot 10^{-3} $ & $ 4.87 \cdot 10^{-3} $ & $ 4.86 \cdot 10^{-3} $ & $ 4.37 \cdot 10^{-3} $ & $ 4.46 \cdot 10^{-3} $ & $ 4.12 \cdot 10^{-3} $ \\
\hline
\end{tabular}
\caption{$X \sim \mathcal{N}(0,1), \forall i \in \lbrace 1,2 \rbrace, \  \varepsilon^{(i)} \sim b\Gamma(2,2,3,3), \ m \in \lbrace 6,7,8 \rbrace$}
\label{tab:GaussianBiGamma2}
\end{table}
\end{landscape}
\newpage

\begin{landscape}
\begin{table}
\begin{tabular}
{|c||c|c|c|c|c|c|c|c|c|c|}
\hline
 \multicolumn{11}{|c|}{$\nu_{\text{est}}$} \\
 \hline
 \multicolumn{1}{|c|}{} & $0.33$ & $0.5$ & $1$ & $1.5$ & $2$ & $2.5$ & $3$ & $3.5$ & $4$ & $4.5$\\
\hline
\hline
{}{}{$ 9 $} & $ 1.62 \cdot 10^{-1} $ & $ 1.61 \cdot 10^{-1} $ & $ 1.61 \cdot 10^{-1} $ & $ 1.61 \cdot 10^{-1} $ & $ 1.61 \cdot 10^{-1} $ & $ 1.61 \cdot 10^{-1} $ & $ 1.61 \cdot 10^{-1} $ & $ 1.62 \cdot 10^{-1} $ & $ 1.61 \cdot 10^{-1} $ & $ 1.61 \cdot 10^{-1} $ \\
\cline{2-11}
& $ 1.00 \cdot 10^{-1} $ & $ 1.00 \cdot 10^{-1} $ & $ 1.00 \cdot 10^{-1} $ & $ 1.00 \cdot 10^{-1} $ & $ 1.00 \cdot 10^{-1} $ & $ 1.00 \cdot 10^{-1} $ & $ 9.99 \cdot 10^{-2} $ & $ 1.01 \cdot 10^{-1} $ & $ 9.94 \cdot 10^{-2} $ & $ 9.91 \cdot 10^{-2} $ \\
\cline{2-11}
& $ 6.21 \cdot 10^{-2} $ & $ 6.20 \cdot 10^{-2} $ & $ 6.20 \cdot 10^{-2} $ & $ 6.20 \cdot 10^{-2} $ & $ 6.19 \cdot 10^{-2} $ & $ 6.19 \cdot 10^{-2} $ & $ 6.18 \cdot 10^{-2} $ & $ 6.66 \cdot 10^{-2} $ & $ 6.08 \cdot 10^{-2} $ & $ 6.04 \cdot 10^{-2} $ \\
\cline{2-11}
& $ 3.19 \cdot 10^{-2} $ & $ 3.22 \cdot 10^{-2} $ & $ 3.22 \cdot 10^{-2} $ & $ 3.22 \cdot 10^{-2} $ & $ 3.21 \cdot 10^{-2} $ & $ 3.20 \cdot 10^{-2} $ & $ 3.21 \cdot 10^{-2} $ & $ 4.60 \cdot 10^{-2} $ & $ 3.04 \cdot 10^{-2} $ & $ 2.99 \cdot 10^{-2} $ \\
\cline{2-11}
& $ 1.21 \cdot 10^{-2} $ & $ 1.32 \cdot 10^{-2} $ & $ 1.33 \cdot 10^{-2} $ & $ 1.33 \cdot 10^{-2} $ & $ 1.33 \cdot 10^{-2} $ & $ 1.32 \cdot 10^{-2} $ & $ 1.35 \cdot 10^{-2} $ & $ 3.79 \cdot 10^{-2} $ & $ 1.19 \cdot 10^{-2} $ & $ 1.18 \cdot 10^{-2} $ \\
\cline{2-11}
& $ 2.94 \cdot 10^{-3} $ & $ 3.91 \cdot 10^{-3} $ & $ 4.60 \cdot 10^{-3} $ & $ 4.60 \cdot 10^{-3} $ & $ 4.58 \cdot 10^{-3} $ & $ 4.57 \cdot 10^{-3} $ & $ 5.18 \cdot 10^{-3} $ & $ 3.79 \cdot 10^{-2} $ & $ 4.34 \cdot 10^{-3} $ & $ 4.98 \cdot 10^{-3} $ \\
\hline

\end{tabular}
\caption{$X \sim \mathcal{N}(0,1), \forall i \in \lbrace 1,2 \rbrace, \  \varepsilon^{(i)} \sim b\Gamma(2,2,3,3), \ m \in \lbrace 9,10,11 \rbrace$}
\label{tab:GaussianBiGamma2bis}
\end{table}
\end{landscape}
\newpage

\begin{landscape}
\begin{table}
\begin{tabular}
{|c||c|c|c|c|c|c|c|c|c|c|}
\hline
 \multicolumn{11}{|c|}{$\nu_{\text{est}}$} \\
 \hline
 \multicolumn{1}{|c|}{} & $0.33$ & $0.5$ & $1$ & $1.5$ & $2$ & $2.5$ & $3$ & $3.5$ & $4$ & $4.5$\\
\hline
\hline
{}{}{$ 12 $} & $ 1.62 \cdot 10^{-1} $ & $ 1.61 \cdot 10^{-1} $ & $ 1.61 \cdot 10^{-1} $ & $ 1.61 \cdot 10^{-1} $ & $ 1.62 \cdot 10^{-1} $ & $ 1.62 \cdot 10^{-1} $ & $ 1.61 \cdot 10^{-1} $ & $ 1.62 \cdot 10^{-1} $ & $ 1.51 \cdot 10^{-1} $ & $ 1.65 \cdot 10^{-1} $ \\
\cline{2-11}
& $ 1.00 \cdot 10^{-1} $ & $ 1.00 \cdot 10^{-1} $ & $ 1.00 \cdot 10^{-1} $ & $ 1.00 \cdot 10^{-1} $ & $ 1.00 \cdot 10^{-1} $ & $ 1.00 \cdot 10^{-1} $ & $ 9.95 \cdot 10^{-2} $ & $ 1.00 \cdot 10^{-1} $ & $ 8.80 \cdot 10^{-2} $ & $ 1.23 \cdot 10^{-1} $ \\
\cline{2-11}
& $ 6.21 \cdot 10^{-2} $ & $ 6.20 \cdot 10^{-2} $ & $ 6.21 \cdot 10^{-2} $ & $ 6.21 \cdot 10^{-2} $ & $ 6.23 \cdot 10^{-2} $ & $ 6.23 \cdot 10^{-2} $ & $ 6.18 \cdot 10^{-2} $ & $ 6.27 \cdot 10^{-2} $ & $ 1.74 \cdot 10^{-1} $ & $ 1.93 \cdot 10^{-1} $ \\
\cline{2-11}
& $ 3.19 \cdot 10^{-2} $ & $ 3.23 \cdot 10^{-2} $ & $ 3.23 \cdot 10^{-2} $ & $ 3.24 \cdot 10^{-2} $ & $ 3.27 \cdot 10^{-2} $ & $ 3.26 \cdot 10^{-2} $ & $ 3.40 \cdot 10^{-2} $ & $ 3.33 \cdot 10^{-2} $ & $ 4.91 \cdot 10^{-1} $ & $ 4.36 \cdot 10^{-1} $ \\
\cline{2-11}
& $ 1.21 \cdot 10^{-2} $ & $ 1.28 \cdot 10^{-2} $ & $ 1.34 \cdot 10^{-2} $ & $ 1.34 \cdot 10^{-2} $ & $ 1.38 \cdot 10^{-2} $ & $ 1.38 \cdot 10^{-2} $ & $ 1.78 \cdot 10^{-2} $ & $ 1.47 \cdot 10^{-2} $ & $ 1.04 $ & $ 2.15 $ \\
\cline{2-11}
& $ 2.74 \cdot 10^{-3} $ & $ 1.94 \cdot 10^{-3} $ & $ 4.64 \cdot 10^{-3} $ & $ 4.66 \cdot 10^{-3} $ & $ 4.87 \cdot 10^{-3} $ & $ 4.90 \cdot 10^{-3} $ & $ 9.86 \cdot 10^{-3} $ & $ 5.82 \cdot 10^{-3} $ & $ 1.63 $ & $ 1.08 \cdot 10^{1} $ \\
\hline
\end{tabular}
\caption{$X \sim \mathcal{N}(0,1), \forall i \in \lbrace 1,2 \rbrace, \  \varepsilon^{(i)} \sim b\Gamma(2,2,3,3), \ m \in \lbrace 12,13,14 \rbrace$}
\label{tab:GaussianBiGamma2bisbis}
\end{table}
\end{landscape}
\newpage

\begin{landscape}
\begin{table}
\begin{tabular}
{|c||c|c|c|c|c|c|c|c|c|c|}
\hline
 \multicolumn{11}{|c|}{$\nu_{\text{est}}$} \\
 \hline
 \multicolumn{1}{|c|}{} & $0.33$ & $0.5$ & $1$ & $1.5$ & $2$ & $2.5$ & $3$ & $3.5$ & $4$ & $4.5$\\
\hline
\hline

{}{}{$ 15 $} & $ 1.62 \cdot 10^{-1} $ & $ 1.61 \cdot 10^{-1} $ & $ 1.61 \cdot 10^{-1} $ & $ 1.62 \cdot 10^{-1} $ & $ 1.62 \cdot 10^{-1} $ & $ 1.62 \cdot 10^{-1} $ & $ 1.62 \cdot 10^{-1} $ & $ 1.62 \cdot 10^{-1} $ & $ 1.62 \cdot 10^{-1} $ & $ 1.62 \cdot 10^{-1} $ \\
\cline{2-11}
& $ 1.00 \cdot 10^{-1} $ & $ 1.00 \cdot 10^{-1} $ & $ 1.00 \cdot 10^{-1} $ & $ 1.00 \cdot 10^{-1} $ & $ 1.00 \cdot 10^{-1} $ & $ 1.00 \cdot 10^{-1} $ & $ 1.00 \cdot 10^{-1} $ & $ 1.00 \cdot 10^{-1} $ & $ 1.00 \cdot 10^{-1} $ & $ 1.01 \cdot 10^{-1} $ \\
\cline{2-11}
& $ 6.21 \cdot 10^{-2} $ & $ 6.21 \cdot 10^{-2} $ & $ 6.20 \cdot 10^{-2} $ & $ 6.27 \cdot 10^{-2} $ & $ 6.27 \cdot 10^{-2} $ & $ 6.27 \cdot 10^{-2} $ & $ 6.27 \cdot 10^{-2} $ & $ 6.27 \cdot 10^{-2} $ & $ 6.27 \cdot 10^{-2} $ & $ 6.36 \cdot 10^{-2} $ \\
\cline{2-11}
& $ 3.19 \cdot 10^{-2} $ & $ 3.23 \cdot 10^{-2} $ & $ 3.22 \cdot 10^{-2} $ & $ 3.33 \cdot 10^{-2} $ & $ 3.33 \cdot 10^{-2} $ & $ 3.33 \cdot 10^{-2} $ & $ 3.33 \cdot 10^{-2} $ & $ 3.33 \cdot 10^{-2} $ & $ 3.33 \cdot 10^{-2} $ & $ 5.12 \cdot 10^{-2} $ \\
\cline{2-11}
& $ 1.20 \cdot 10^{-2} $ & $ 1.24 \cdot 10^{-2} $ & $ 1.33 \cdot 10^{-2} $ & $ 1.46 \cdot 10^{-2} $ & $ 1.46 \cdot 10^{-2} $ & $ 1.46 \cdot 10^{-2} $ & $ 1.46 \cdot 10^{-2} $ & $ 1.46 \cdot 10^{-2} $ & $ 1.46 \cdot 10^{-2} $ & $ 1.54 \cdot 10^{-1} $ \\
\cline{2-11}
& $ 2.68 \cdot 10^{-3} $ & $ 4.57 \cdot 10^{-3} $ & $ 4.61 \cdot 10^{-3} $ & $ 5.66 \cdot 10^{-3} $ & $ 5.66 \cdot 10^{-3} $ & $ 5.66 \cdot 10^{-3} $ & $ 5.66 \cdot 10^{-3} $ & $ 5.68 \cdot 10^{-3} $ & $ 5.67 \cdot 10^{-3} $ & $ 6.31 \cdot 10^{-1} $ \\
\hline

\end{tabular}
\caption{$X \sim \mathcal{N}(0,1), \forall i \in \lbrace 1,2 \rbrace, \  \varepsilon^{(i)} \sim b\Gamma(2,2,3,3), \ m = 15$}
\label{tab:GaussianBiGamma3}
\end{table}
\end{landscape}
\newpage

\begin{landscape}

\begin{table}
\begin{tabular}
{|c||c|c|c|c|c|c|c|c|c|c|}
\hline
 \multicolumn{11}{|c|}{$\nu_{\text{est}}$} \\
 \hline
 \multicolumn{1}{|c|}{} & $0.33$ & $0.5$ & $1$ & $1.5$ & $2$ & $2.5$ & $3$ & $3.5$ & $4$ & $4.5$\\
\hline
\hline

{}{}{$ 3 $} & $ 1.61 \cdot 10^{-1} $ & $ 1.61 \cdot 10^{-1} $ & $ 1.61 \cdot 10^{-1} $ & $ 1.61 \cdot 10^{-1} $ & $ 1.61 \cdot 10^{-1} $ & $ 1.61 \cdot 10^{-1} $ & $ 1.61 \cdot 10^{-1} $ & $ 1.61 \cdot 10^{-1} $ & $ 1.61 \cdot 10^{-1} $ & $ 1.61 \cdot 10^{-1} $ \\
\cline{2-11}
& $ 9.99 \cdot 10^{-2} $ & $ 9.96 \cdot 10^{-2} $ & $ 9.89 \cdot 10^{-2} $ & $ 9.87 \cdot 10^{-2} $ & $ 9.87 \cdot 10^{-2} $ & $ 9.87 \cdot 10^{-2} $ & $ 9.86 \cdot 10^{-2} $ & $ 9.85 \cdot 10^{-2} $ & $ 9.81 \cdot 10^{-2} $ & $ 9.78 \cdot 10^{-2} $ \\
\cline{2-11}
& $ 6.20 \cdot 10^{-2} $ & $ 6.13 \cdot 10^{-2} $ & $ 5.98 \cdot 10^{-2} $ & $ 5.95 \cdot 10^{-2} $ & $ 5.95 \cdot 10^{-2} $ & $ 5.95 \cdot 10^{-2} $ & $ 5.94 \cdot 10^{-2} $ & $ 5.92 \cdot 10^{-2} $ & $ 5.83 \cdot 10^{-2} $ & $ 5.78 \cdot 10^{-2} $ \\
\cline{2-11}
& $ 3.32 \cdot 10^{-2} $ & $ 3.18 \cdot 10^{-2} $ & $ 2.90 \cdot 10^{-2} $ & $ 2.86 \cdot 10^{-2} $ & $ 2.86 \cdot 10^{-2} $ & $ 2.86 \cdot 10^{-2} $ & $ 2.84 \cdot 10^{-2} $ & $ 2.80 \cdot 10^{-2} $ & $ 2.67 \cdot 10^{-2} $ & $ 2.60 \cdot 10^{-2} $ \\
\cline{2-11}
& $ 1.61 \cdot 10^{-2} $ & $ 1.41 \cdot 10^{-2} $ & $ 1.10 \cdot 10^{-2} $ & $ 1.09 \cdot 10^{-2} $ & $ 1.11 \cdot 10^{-2} $ & $ 1.11 \cdot 10^{-2} $ & $ 1.08 \cdot 10^{-2} $ & $ 1.04 \cdot 10^{-2} $ & $ 9.82 \cdot 10^{-3} $ & $ 9.78 \cdot 10^{-3} $ \\
\cline{2-11}
& $ 1.01 \cdot 10^{-2} $ & $ 7.03 \cdot 10^{-3} $ & $ 4.61 \cdot 10^{-3} $ & $ 5.70 \cdot 10^{-3} $ & $ 5.95 \cdot 10^{-3} $ & $ 5.95 \cdot 10^{-3} $ & $ 5.76 \cdot 10^{-3} $ & $ 5.75 \cdot 10^{-3} $ & $ 7.62 \cdot 10^{-3} $ & $ 9.61 \cdot 10^{-3} $ \\
\hline

\end{tabular}
\caption{$X \sim \mathcal{N}(0,1), \forall i \in \lbrace 1,2 \rbrace, \  \varepsilon^{(i)} \sim b\Gamma(2,2,3,3), \ m \in \lbrace 3,4,5 \rbrace$}
\label{tab:GaussianBiGammaMC1}
\end{table}
\end{landscape}

\begin{landscape}

\begin{table}
\begin{tabular}
{|c||c|c|c|c|c|c|c|c|c|c|}
\hline
 \multicolumn{11}{|c|}{$\nu_{\text{est}}$} \\
 \hline
 \multicolumn{1}{|c|}{} & $0.33$ & $0.5$ & $1$ & $1.5$ & $2$ & $2.5$ & $3$ & $3.5$ & $4$ & $4.5$\\
\hline
\hline
{}{}{$ 6 $} & $ 1.62 \cdot 10^{-1} $ & $ 1.62 \cdot 10^{-1} $ & $ 1.61 \cdot 10^{-1} $ & $ 1.61 \cdot 10^{-1} $ & $ 1.61 \cdot 10^{-1} $ & $ 1.61 \cdot 10^{-1} $ & $ 1.61 \cdot 10^{-1} $ & $ 1.61 \cdot 10^{-1} $ & $ 1.61 \cdot 10^{-1} $ & $ 1.61 \cdot 10^{-1} $ \\
\cline{2-11}
& $ 1.00 \cdot 10^{-1} $ & $ 1.00 \cdot 10^{-1} $ & $ 1.00 \cdot 10^{-1} $ & $ 10.00 \cdot 10^{-2} $ & $ 9.99 \cdot 10^{-2} $ & $ 9.97 \cdot 10^{-2} $ & $ 9.97 \cdot 10^{-2} $ & $ 9.95 \cdot 10^{-2} $ & $ 9.92 \cdot 10^{-2} $ & $ 9.91 \cdot 10^{-2} $ \\
\cline{2-11}
& $ 6.23 \cdot 10^{-2} $ & $ 6.23 \cdot 10^{-2} $ & $ 6.25 \cdot 10^{-2} $ & $ 6.20 \cdot 10^{-2} $ & $ 6.17 \cdot 10^{-2} $ & $ 6.14 \cdot 10^{-2} $ & $ 6.13 \cdot 10^{-2} $ & $ 6.10 \cdot 10^{-2} $ & $ 6.05 \cdot 10^{-2} $ & $ 6.03 \cdot 10^{-2} $ \\
\cline{2-11}
& $ 3.24 \cdot 10^{-2} $ & $ 3.26 \cdot 10^{-2} $ & $ 3.35 \cdot 10^{-2} $ & $ 3.25 \cdot 10^{-2} $ & $ 3.18 \cdot 10^{-2} $ & $ 3.15 \cdot 10^{-2} $ & $ 3.12 \cdot 10^{-2} $ & $ 3.08 \cdot 10^{-2} $ & $ 3.00 \cdot 10^{-2} $ & $ 2.96 \cdot 10^{-2} $ \\
\cline{2-11}
& $ 1.32 \cdot 10^{-2} $ & $ 1.39 \cdot 10^{-2} $ & $ 1.56 \cdot 10^{-2} $ & $ 1.42 \cdot 10^{-2} $ & $ 1.32 \cdot 10^{-2} $ & $ 1.28 \cdot 10^{-2} $ & $ 1.26 \cdot 10^{-2} $ & $ 1.21 \cdot 10^{-2} $ & $ 1.16 \cdot 10^{-2} $ & $ 1.11 \cdot 10^{-2} $ \\
\cline{2-11}
& $ 4.30 \cdot 10^{-3} $ & $ 5.68 \cdot 10^{-3} $ & $ 7.51 \cdot 10^{-3} $ & $ 5.99 \cdot 10^{-3} $ & $ 4.80 \cdot 10^{-3} $ & $ 4.58 \cdot 10^{-3} $ & $ 4.53 \cdot 10^{-3} $ & $ 4.26 \cdot 10^{-3} $ & $ 4.46 \cdot 10^{-3} $ & $ 4.08 \cdot 10^{-3} $ \\
\hline
\end{tabular}
\caption{$X \sim \mathcal{N}(0,1), \forall i \in \lbrace 1,2 \rbrace, \  \varepsilon^{(i)} \sim b\Gamma(2,2,3,3), \ m \in \lbrace 6,7,8 \rbrace$}
\label{tab:GaussianBiGammaMC1bis}
\end{table}
\end{landscape}

\begin{landscape}
\begin{table}
\begin{tabular}{|c||c|c|c|c|c|c|c|c|c|c|}
\hline
 \multicolumn{11}{|c|}{$\nu_{\text{est}}$} \\
 \hline
 \multicolumn{1}{|c|}{} & $0.33$ & $0.5$ & $1$ & $1.5$ & $2$ & $2.5$ & $3$ & $3.5$ & $4$ & $4.5$\\
\hline
\hline
{}{}{$ 9 $} & $ 1.62 \cdot 10^{-1} $ & $ 1.62 \cdot 10^{-1} $ & $ 1.61 \cdot 10^{-1} $ & $ 1.62 \cdot 10^{-1} $ & $ 1.62 \cdot 10^{-1} $ & $ 1.61 \cdot 10^{-1} $ & $ 1.61 \cdot 10^{-1} $ & $ 1.61 \cdot 10^{-1} $ & $ 1.62 \cdot 10^{-1} $ & $ 1.63 \cdot 10^{-1} $ \\
\cline{2-11}
& $ 1.00 \cdot 10^{-1} $ & $ 1.00 \cdot 10^{-1} $ & $ 1.00 \cdot 10^{-1} $ & $ 1.00 \cdot 10^{-1} $ & $ 1.00 \cdot 10^{-1} $ & $ 1.00 \cdot 10^{-1} $ & $ 9.99 \cdot 10^{-2} $ & $ 9.98 \cdot 10^{-2} $ & $ 1.01 \cdot 10^{-1} $ & $ 1.16 \cdot 10^{-1} $ \\
\cline{2-11}
& $ 6.23 \cdot 10^{-2} $ & $ 6.23 \cdot 10^{-2} $ & $ 6.22 \cdot 10^{-2} $ & $ 6.30 \cdot 10^{-2} $ & $ 6.32 \cdot 10^{-2} $ & $ 6.21 \cdot 10^{-2} $ & $ 6.17 \cdot 10^{-2} $ & $ 6.15 \cdot 10^{-2} $ & $ 6.75 \cdot 10^{-2} $ & $ 1.42 \cdot 10^{-1} $ \\
\cline{2-11}
& $ 3.23 \cdot 10^{-2} $ & $ 3.26 \cdot 10^{-2} $ & $ 3.25 \cdot 10^{-2} $ & $ 3.49 \cdot 10^{-2} $ & $ 3.56 \cdot 10^{-2} $ & $ 3.26 \cdot 10^{-2} $ & $ 3.18 \cdot 10^{-2} $ & $ 3.16 \cdot 10^{-2} $ & $ 4.81 \cdot 10^{-2} $ & $ 2.29 \cdot 10^{-1} $ \\
\cline{2-11}
& $ 1.29 \cdot 10^{-2} $ & $ 1.41 \cdot 10^{-2} $ & $ 1.37 \cdot 10^{-2} $ & $ 1.79 \cdot 10^{-2} $ & $ 1.92 \cdot 10^{-2} $ & $ 1.43 \cdot 10^{-2} $ & $ 1.31 \cdot 10^{-2} $ & $ 1.31 \cdot 10^{-2} $ & $ 4.14 \cdot 10^{-2} $ & $ 3.34 \cdot 10^{-1} $ \\
\cline{2-11}
& $ 3.80 \cdot 10^{-3} $ & $ 6.55 \cdot 10^{-3} $ & $ 4.98 \cdot 10^{-3} $ & $ 1.06 \cdot 10^{-2} $ & $ 1.23 \cdot 10^{-2} $ & $ 6.18 \cdot 10^{-3} $ & $ 4.75 \cdot 10^{-3} $ & $ 4.99 \cdot 10^{-3} $ & $ 4.24 \cdot 10^{-2} $ & $ 4.35 \cdot 10^{-1} $ \\
\hline
\end{tabular}
\caption{$X \sim \mathcal{N}(0,1), \forall i \in \lbrace 1,2 \rbrace, \  \varepsilon^{(i)} \sim b\Gamma(2,2,3,3), \ m \in \lbrace 9,10,11 \rbrace$}
\label{tab:GaussianBiGammaMC2}
\end{table}
\end{landscape}
\newpage

\begin{landscape}
\begin{table}
\begin{tabular}{|c||c|c|c|c|c|c|c|c|c|c|}
\hline
 \multicolumn{11}{|c|}{$\nu_{\text{est}}$} \\
 \hline
 \multicolumn{1}{|c|}{} & $0.33$ & $0.5$ & $1$ & $1.5$ & $2$ & $2.5$ & $3$ & $3.5$ & $4$ & $4.5$\\
\hline
\hline
{}{}{$ 12 $} & $ 1.62 \cdot 10^{-1} $ & $ 1.62 \cdot 10^{-1} $ & $ 1.62 \cdot 10^{-1} $ & $ 1.62 \cdot 10^{-1} $ & $ 1.62 \cdot 10^{-1} $ & $ 1.61 \cdot 10^{-1} $ & $ 1.61 \cdot 10^{-1} $ & $ 1.56 \cdot 10^{-1} $ & $ 1.43 \cdot 10^{-1} $ & $ 1.62 \cdot 10^{-1} $ \\
\cline{2-11}
& $ 1.00 \cdot 10^{-1} $ & $ 1.00 \cdot 10^{-1} $ & $ 1.00 \cdot 10^{-1} $ & $ 1.00 \cdot 10^{-1} $ & $ 1.00 \cdot 10^{-1} $ & $ 1.00 \cdot 10^{-1} $ & $ 9.95 \cdot 10^{-2} $ & $ 8.79 \cdot 10^{-2} $ & $ 1.23 \cdot 10^{-1} $ & $ 1.05 \cdot 10^{-1} $ \\
\cline{2-11}
& $ 6.23 \cdot 10^{-2} $ & $ 6.22 \cdot 10^{-2} $ & $ 6.22 \cdot 10^{-2} $ & $ 6.23 \cdot 10^{-2} $ & $ 6.24 \cdot 10^{-2} $ & $ 6.24 \cdot 10^{-2} $ & $ 6.18 \cdot 10^{-2} $ & $ 8.94 \cdot 10^{-2} $ & $ 5.93 \cdot 10^{-1} $ & $ 8.04 \cdot 10^{-2} $ \\
\cline{2-11}
& $ 3.23 \cdot 10^{-2} $ & $ 3.26 \cdot 10^{-2} $ & $ 3.26 \cdot 10^{-2} $ & $ 3.26 \cdot 10^{-2} $ & $ 3.29 \cdot 10^{-2} $ & $ 3.34 \cdot 10^{-2} $ & $ 3.40 \cdot 10^{-2} $ & $ 2.15 \cdot 10^{-1} $ & $ 1.91 $ & $ 1.09 \cdot 10^{-1} $ \\
\cline{2-11}
& $ 1.29 \cdot 10^{-2} $ & $ 1.34 \cdot 10^{-2} $ & $ 1.37 \cdot 10^{-2} $ & $ 1.38 \cdot 10^{-2} $ & $ 1.41 \cdot 10^{-2} $ & $ 1.51 \cdot 10^{-2} $ & $ 1.83 \cdot 10^{-2} $ & $ 4.94 \cdot 10^{-1} $ & $ 4.09 $ & $ 3.32 \cdot 10^{-1} $ \\
\cline{2-11}
& $ 3.68 \cdot 10^{-3} $ & $ 2.44 \cdot 10^{-3} $ & $ 4.99 \cdot 10^{-3} $ & $ 5.01 \cdot 10^{-3} $ & $ 5.18 \cdot 10^{-3} $ & $ 6.25 \cdot 10^{-3} $ & $ 1.22 \cdot 10^{-2} $ & $ 9.24 \cdot 10^{-1} $ & $ 6.74 $ & $ 8.69 \cdot 10^{-1} $ \\
\hline
\end{tabular}
\caption{$X \sim \mathcal{N}(0,1), \forall i \in \lbrace 1,2 \rbrace, \  \varepsilon^{(i)} \sim b\Gamma(2,2,3,3), \ m \in \lbrace 12,13,14 \rbrace$}
\label{tab:GaussianBiGammaMC2bis}
\end{table}
\end{landscape}

\begin{landscape}
\begin{table}
\begin{tabular}{|c||c|c|c|c|c|c|c|c|c|c|}
\hline
 \multicolumn{11}{|c|}{$\nu_{\text{est}}$} \\
 \hline
 \multicolumn{1}{|c|}{} & $0.33$ & $0.5$ & $1$ & $1.5$ & $2$ & $2.5$ & $3$ & $3.5$ & $4$ & $4.5$\\
\hline
\hline
{}{}{$ 15 $} & $ 1.62 \cdot 10^{-1} $ & $ 1.62 \cdot 10^{-1} $ & $ 1.62 \cdot 10^{-1} $ & $ 1.62 \cdot 10^{-1} $ & $ 1.62 \cdot 10^{-1} $ & $ 1.62 \cdot 10^{-1} $ & $ 1.62 \cdot 10^{-1} $ & $ 1.62 \cdot 10^{-1} $ & $ 1.62 \cdot 10^{-1} $ & $ 1.62 \cdot 10^{-1} $ \\
\cline{2-11}
& $ 1.00 \cdot 10^{-1} $ & $ 1.00 \cdot 10^{-1} $ & $ 1.00 \cdot 10^{-1} $ & $ 1.00 \cdot 10^{-1} $ & $ 1.00 \cdot 10^{-1} $ & $ 1.00 \cdot 10^{-1} $ & $ 1.00 \cdot 10^{-1} $ & $ 1.00 \cdot 10^{-1} $ & $ 1.00 \cdot 10^{-1} $ & $ 1.00 \cdot 10^{-1} $ \\
\cline{2-11}
& $ 6.23 \cdot 10^{-2} $ & $ 6.22 \cdot 10^{-2} $ & $ 6.24 \cdot 10^{-2} $ & $ 6.27 \cdot 10^{-2} $ & $ 6.27 \cdot 10^{-2} $ & $ 6.27 \cdot 10^{-2} $ & $ 6.27 \cdot 10^{-2} $ & $ 6.27 \cdot 10^{-2} $ & $ 6.27 \cdot 10^{-2} $ & $ 6.27 \cdot 10^{-2} $ \\
\cline{2-11}
& $ 3.23 \cdot 10^{-2} $ & $ 3.26 \cdot 10^{-2} $ & $ 3.28 \cdot 10^{-2} $ & $ 3.33 \cdot 10^{-2} $ & $ 3.33 \cdot 10^{-2} $ & $ 3.33 \cdot 10^{-2} $ & $ 3.33 \cdot 10^{-2} $ & $ 3.33 \cdot 10^{-2} $ & $ 3.33 \cdot 10^{-2} $ & $ 3.33 \cdot 10^{-2} $ \\
\cline{2-11}
& $ 1.29 \cdot 10^{-2} $ & $ 1.31 \cdot 10^{-2} $ & $ 1.39 \cdot 10^{-2} $ & $ 1.46 \cdot 10^{-2} $ & $ 1.46 \cdot 10^{-2} $ & $ 1.46 \cdot 10^{-2} $ & $ 1.46 \cdot 10^{-2} $ & $ 1.46 \cdot 10^{-2} $ & $ 1.46 \cdot 10^{-2} $ & $ 1.46 \cdot 10^{-2} $ \\
\cline{2-11}
& $ 3.67 \cdot 10^{-3} $ & $ 2.32 \cdot 10^{-3} $ & $ 5.07 \cdot 10^{-3} $ & $ 5.66 \cdot 10^{-3} $ & $ 5.66 \cdot 10^{-3} $ & $ 5.66 \cdot 10^{-3} $ & $ 5.66 \cdot 10^{-3} $ & $ 5.59 \cdot 10^{-3} $ & $ 5.66 \cdot 10^{-3} $ & $ 5.66 \cdot 10^{-3} $ \\
\hline
\end{tabular}
\caption{$X \sim \mathcal{N}(0,1), \forall i \in \lbrace 1,2 \rbrace, \  \varepsilon^{(i)} \sim b\Gamma(2,2,3,3), \ m = 15 $}
\label{tab:GaussianBiGammaMC3}
\end{table}
\end{landscape}



\printbibliography 

@article {MR2853732,
    AUTHOR = {Comte, F. and Lacour, C.},
     TITLE = {Data-driven density estimation in the presence of additive
              noise with unknown distribution},
   JOURNAL = {J. R. Stat. Soc. Ser. B Stat. Methodol.},
  FJOURNAL = {Journal of the Royal Statistical Society. Series B.
              Statistical Methodology},
    VOLUME = {73},
      YEAR = {2011},
    NUMBER = {4},
     PAGES = {601--627},
      ISSN = {1369-7412,1467-9868},
   MRCLASS = {62G07 (62G20)},
  MRNUMBER = {2853732},
MRREVIEWER = {Tristan\ Senga Kiess\'e},
       DOI = {10.1111/j.1467-9868.2011.00775.x},
       URL = {https://doi.org/10.1111/j.1467-9868.2011.00775.x},
}

@article {MR0203769,
    AUTHOR = {Kotlarski, Ignacy},
     TITLE = {On characterizing the gamma and the normal distribution},
   JOURNAL = {Pacific J. Math.},
  FJOURNAL = {Pacific Journal of Mathematics},
    VOLUME = {20},
      YEAR = {1967},
     PAGES = {69--76},
      ISSN = {0030-8730,1945-5844},
   MRCLASS = {60.20 (62.10)},
  MRNUMBER = {203769},
MRREVIEWER = {E.\ Lukacs},
       URL = {http://projecteuclid.org/euclid.pjm/1102992970},
}

@article {MR4349899,
    AUTHOR = {Kurisu, Daisuke and Otsu, Taisuke},
     TITLE = {On linearization of nonparametric deconvolution estimators for
              repeated measurements model},
   JOURNAL = {J. Multivariate Anal.},
  FJOURNAL = {Journal of Multivariate Analysis},
    VOLUME = {189},
      YEAR = {2022},
     PAGES = {Paper No. 104921, 19},
      ISSN = {0047-259X,1095-7243},
   MRCLASS = {62G05 (62G20)},
  MRNUMBER = {4349899},
MRREVIEWER = {Masayuki\ Hirukawa},
       DOI = {10.1016/j.jmva.2021.104921},
       URL = {https://doi.org/10.1016/j.jmva.2021.104921},
}

@article {MR3207174,
    AUTHOR = {Comte, Fabienne and Samson, Adeline and Stirnemann, Julien J.},
     TITLE = {Deconvolution estimation of onset of pregnancy with replicate
              observations},
   JOURNAL = {Scand. J. Stat.},
  FJOURNAL = {Scandinavian Journal of Statistics. Theory and Applications},
    VOLUME = {41},
      YEAR = {2014},
    NUMBER = {2},
     PAGES = {325--345},
      ISSN = {0303-6898,1467-9469},
   MRCLASS = {62P10 (62G05 62G07)},
  MRNUMBER = {3207174},
       DOI = {10.1111/sjos.12029},
       URL = {https://doi.org/10.1111/sjos.12029},
}

@article {MR1460203,
    AUTHOR = {Neumann, Michael H.},
     TITLE = {On the effect of estimating the error density in nonparametric
              deconvolution},
   JOURNAL = {J. Nonparametr. Statist.},
  FJOURNAL = {Journal of Nonparametric Statistics},
    VOLUME = {7},
      YEAR = {1997},
    NUMBER = {4},
     PAGES = {307--330},
      ISSN = {1048-5252,1029-0311},
   MRCLASS = {62G05},
  MRNUMBER = {1460203},
       DOI = {10.1080/10485259708832708},
       URL = {https://doi.org/10.1080/10485259708832708},
}

@article {MR2543693,
    AUTHOR = {Johannes, Jan},
     TITLE = {Deconvolution with unknown error distribution},
   JOURNAL = {Ann. Statist.},
  FJOURNAL = {The Annals of Statistics},
    VOLUME = {37},
      YEAR = {2009},
    NUMBER = {5A},
     PAGES = {2301--2323},
      ISSN = {0090-5364,2168-8966},
   MRCLASS = {62G07 (42A38)},
  MRNUMBER = {2543693},
       DOI = {10.1214/08-AOS652},
       URL = {https://doi.org/10.1214/08-AOS652},
}

@article {LCGL2021,
      AUTHOR = {Gassiat, \'{E}lisabeth and Le Corff, Sylvain and Leh\'{e}ricy, Luc},
     TITLE = {Deconvolution with unknown noise distribution is possible for
              multivariate signals},
   JOURNAL = {Ann. Statist.},
  FJOURNAL = {The Annals of Statistics},
    VOLUME = {50},
      YEAR = {2022},
    NUMBER = {1},
     PAGES = {303--323},
}

@book {MR0208210,
    AUTHOR = {Acz\'{e}l, J.},
     TITLE = {Lectures on functional equations and their applications},
    SERIES = {Mathematics in Science and Engineering},
    VOLUME = {Vol. 19},
    EDITOR = {Oser, Hansj\"{o}rg},
 PUBLISHER = {Academic Press, New York-London},
      YEAR = {1966},
     PAGES = {xx+510},
   MRCLASS = {39.00},
  MRNUMBER = {208210},
MRREVIEWER = {M.\ Hossz\'{u}},
}

@misc{InferSupport,
      title={Support and distribution inference from noisy data}, 
      author={Jérémie Capitao-Miniconi and Elisabeth Gassiat and Luc Lehéricy},
      year={2023},
      eprint={2304.09452},
      archivePrefix={arXiv},
      primaryClass={math.ST}
}

@article {MR2319471,
    AUTHOR = {Meister, A.},
     TITLE = {Deconvolving compactly supported densities},
   JOURNAL = {Math. Methods Statist.},
  FJOURNAL = {Mathematical Methods of Statistics},
    VOLUME = {16},
      YEAR = {2007},
    NUMBER = {1},
     PAGES = {63--76},
      ISSN = {1066-5307,1934-8045},
   MRCLASS = {62G07},
  MRNUMBER = {2319471},
MRREVIEWER = {Jussi\ S.\ Klemel\"{a}},
       DOI = {10.3103/S106653070701005X},
       URL = {https://doi.org/10.3103/S106653070701005X},
}

@book {MR0856411,
    AUTHOR = {Le Cam, Lucien},
     TITLE = {Asymptotic methods in statistical decision theory},
    SERIES = {Springer Series in Statistics},
 PUBLISHER = {Springer-Verlag, New York},
      YEAR = {1986},
     PAGES = {xxvi+742},
      ISBN = {0-387-96307-3},
   MRCLASS = {62-02 (62B15 62Cxx 62E20)},
  MRNUMBER = {856411},
MRREVIEWER = {Reinhard\ Michel},
       DOI = {10.1007/978-1-4612-4946-7},
       URL = {https://doi.org/10.1007/978-1-4612-4946-7},
}

@article {MR2396811,
    AUTHOR = {Delaigle, Aurore and Hall, Peter and Meister, Alexander},
     TITLE = {On deconvolution with repeated measurements},
   JOURNAL = {Ann. Statist.},
  FJOURNAL = {The Annals of Statistics},
    VOLUME = {36},
      YEAR = {2008},
    NUMBER = {2},
     PAGES = {665--685},
      ISSN = {0090-5364,2168-8966},
   MRCLASS = {62G07 (62G08 65R32)},
  MRNUMBER = {2396811},
MRREVIEWER = {\'{E}lie\ Youndj\'{e}},
       DOI = {10.1214/009053607000000884},
       URL = {https://doi.org/10.1214/009053607000000884},
}

@article {MR3299125,
    AUTHOR = {Kappus, Johanna and Mabon, Gwenna\"{e}lle},
     TITLE = {Adaptive density estimation in deconvolution problems with
              unknown error distribution},
   JOURNAL = {Electron. J. Stat.},
  FJOURNAL = {Electronic Journal of Statistics},
    VOLUME = {8},
      YEAR = {2014},
    NUMBER = {2},
     PAGES = {2879--2904},
      ISSN = {1935-7524},
   MRCLASS = {62G07 (62G09)},
  MRNUMBER = {3299125},
MRREVIEWER = {Tae\ Yoon\ Kim},
       DOI = {10.1214/14-EJS976},
       URL = {https://doi.org/10.1214/14-EJS976},
}

@article {MR1625869,
    AUTHOR = {Li, Tong and Vuong, Quang},
     TITLE = {Nonparametric estimation of the measurement error model using
              multiple indicators},
   JOURNAL = {J. Multivariate Anal.},
  FJOURNAL = {Journal of Multivariate Analysis},
    VOLUME = {65},
      YEAR = {1998},
    NUMBER = {2},
     PAGES = {139--165},
      ISSN = {0047-259X,1095-7243},
   MRCLASS = {62G05},
  MRNUMBER = {1625869},
MRREVIEWER = {Thomas\ W.\ Epps},
       DOI = {10.1006/jmva.1998.1741},
       URL = {https://doi.org/10.1006/jmva.1998.1741},
}

@article {MR2396948,
    AUTHOR = {Neumann, Michael H.},
     TITLE = {Deconvolution from panel data with unknown error distribution},
   JOURNAL = {J. Multivariate Anal.},
  FJOURNAL = {Journal of Multivariate Analysis},
    VOLUME = {98},
      YEAR = {2007},
    NUMBER = {10},
     PAGES = {1955--1968},
      ISSN = {0047-259X,1095-7243},
   MRCLASS = {62G05 (62G20)},
  MRNUMBER = {2396948},
       DOI = {10.1016/j.jmva.2006.09.012},
       URL = {https://doi.org/10.1016/j.jmva.2006.09.012},
}

@article {MR3372551,
    AUTHOR = {Comte, Fabienne and Kappus, Johanna},
     TITLE = {Density deconvolution from repeated measurements without
              symmetry assumption on the errors},
   JOURNAL = {J. Multivariate Anal.},
  FJOURNAL = {Journal of Multivariate Analysis},
    VOLUME = {140},
      YEAR = {2015},
     PAGES = {31--46},
      ISSN = {0047-259X,1095-7243},
   MRCLASS = {62G05 (62G07 62G20)},
  MRNUMBER = {3372551},
MRREVIEWER = {Botond\ Szab\'{o}},
       DOI = {10.1016/j.jmva.2015.04.004},
       URL = {https://doi.org/10.1016/j.jmva.2015.04.004},
}

@article {MR0997599,
    AUTHOR = {Carroll, Raymond J. and Hall, Peter},
     TITLE = {Optimal rates of convergence for deconvolving a density},
   JOURNAL = {J. Amer. Statist. Assoc.},
  FJOURNAL = {Journal of the American Statistical Association},
    VOLUME = {83},
      YEAR = {1988},
    NUMBER = {404},
     PAGES = {1184--1186},
      ISSN = {0162-1459,1537-274X},
   MRCLASS = {62G05 (62C10 62J02)},
  MRNUMBER = {997599},
MRREVIEWER = {T.\ Robertson},
       URL =
              {http://links.jstor.org/sici?sici=0162-1459(198812)83:404<1184:OROCFD>2.0.CO;2-R&origin=MSN},
}

@article {MR1045189,
    AUTHOR = {Stefanski, Leonard A.},
     TITLE = {Rates of convergence of some estimators in a class of
              deconvolution problems},
   JOURNAL = {Statist. Probab. Lett.},
  FJOURNAL = {Statistics \& Probability Letters},
    VOLUME = {9},
      YEAR = {1990},
    NUMBER = {3},
     PAGES = {229--235},
      ISSN = {0167-7152,1879-2103},
   MRCLASS = {62G05 (62E20)},
  MRNUMBER = {1045189},
MRREVIEWER = {M.\ Hu\v{s}kov\'{a}},
       DOI = {10.1016/0167-7152(90)90061-B},
       URL = {https://doi.org/10.1016/0167-7152(90)90061-B},
}

@article {MR1054861,
    AUTHOR = {Stefanski, Leonard and Carroll, Raymond J.},
     TITLE = {Deconvoluting kernel density estimators},
   JOURNAL = {Statistics},
  FJOURNAL = {Statistics. A Journal of Theoretical and Applied Statistics},
    VOLUME = {21},
      YEAR = {1990},
    NUMBER = {2},
     PAGES = {169--184},
      ISSN = {0233-1888,1029-4910},
   MRCLASS = {62G05 (62G20)},
  MRNUMBER = {1054861},
MRREVIEWER = {Thomas\ W.\ Sager},
       DOI = {10.1080/02331889008802238},
       URL = {https://doi.org/10.1080/02331889008802238},
}

@article {MR1126324,
    AUTHOR = {Fan, Jianqing},
     TITLE = {On the optimal rates of convergence for nonparametric
              deconvolution problems},
   JOURNAL = {Ann. Statist.},
  FJOURNAL = {The Annals of Statistics},
    VOLUME = {19},
      YEAR = {1991},
    NUMBER = {3},
     PAGES = {1257--1272},
      ISSN = {0090-5364,2168-8966},
   MRCLASS = {62G05},
  MRNUMBER = {1126324},
MRREVIEWER = {M.\ A.\ Mirzakhmedov},
       DOI = {10.1214/aos/1176348248},
       URL = {https://doi.org/10.1214/aos/1176348248},
}

@book {MR2768576,
    AUTHOR = {Meister, Alexander},
     TITLE = {Deconvolution problems in nonparametric statistics},
    SERIES = {Lecture Notes in Statistics},
    VOLUME = {193},
 PUBLISHER = {Springer-Verlag, Berlin},
      YEAR = {2009},
     PAGES = {vi+210},
      ISBN = {978-3-540-87556-7},
   MRCLASS = {62-02 (62G05 62G08 62G20 62H35 62J07 94A12)},
  MRNUMBER = {2768576},
MRREVIEWER = {Su-Yun\ Chen\ Huang},
       DOI = {10.1007/978-3-540-87557-4},
       URL = {https://doi.org/10.1007/978-3-540-87557-4},
}

@article {MR2543590,
    AUTHOR = {Goldenshluger, Alexander and Lepski, Oleg},
     TITLE = {Universal pointwise selection rule in multivariate function
              estimation},
   JOURNAL = {Bernoulli},
  FJOURNAL = {Bernoulli. Official Journal of the Bernoulli Society for
              Mathematical Statistics and Probability},
    VOLUME = {14},
      YEAR = {2008},
    NUMBER = {4},
     PAGES = {1150--1190},
      ISSN = {1350-7265},
   MRCLASS = {62G05},
  MRNUMBER = {2543590},
MRREVIEWER = {Michael Falk},
       DOI = {10.3150/08-BEJ144},
       URL = {https://doi-org.ezproxy.universite-paris-saclay.fr/10.3150/08-BEJ144},
}

\end{document}